\documentclass[a4paper, twoside, leqno]{article}

\usepackage[T1]{fontenc}
\usepackage[utf8x]{inputenc}
\usepackage[UKenglish]{babel}
\usepackage{lmodern}
\usepackage{authblk}
\usepackage[active]{srcltx}
\usepackage{amsmath,amsthm,amssymb}
\usepackage{mathrsfs,esint,bbm,stmaryrd}
\usepackage{enumitem,graphicx,xcolor,subcaption}
\usepackage[textwidth=16cm,centering,headheight=24pt]{geometry}

\usepackage{fullpage}

\newtheorem{theorem}{Theorem}           % Bold title, italic text
\newtheorem{corollary}[theorem]{Corollary}
\newtheorem{lemma}[theorem]{Lemma}
\newtheorem{prop}[theorem]{Proposition}

\theoremstyle{definition}              % Bold title, roman text

\theoremstyle{remark}                  % Italic title, roman text

\newtheorem{remark}[theorem]{Remark}
\newtheorem*{remark*}{Remark}

\DeclareMathOperator{\dist}{dist}
\DeclareMathOperator{\spt}{spt}

\newcommand{\abs}[1]{\left| #1 \right|}
\newcommand{\norm}[1]{\left\| #1 \right\|}
\newcommand{\eps}{\varepsilon}
\newcommand{\Dir}{\Lambda^{\mathrm{Dir}}}
\newcommand{\St}{\Lambda^{\mathrm{St}}}
\renewcommand{\Cap}{\Lambda^{\mathrm{Cap}}}

\DeclareMathAlphabet{\mathpzc}{OT1}{pzc}{m}{it}

\renewcommand{\d}{\mathrm{d}}
\newcommand{\N}{\mathbb{N}}       % Blackboard bold
\newcommand{\R}{\mathbb{R}}
\newcommand{\Z}{\mathbb{Z}}

\newcommand{\T}{\mathbb{T}}

\newcommand{\mcP}{\mathscr{P}}
\newcommand{\mcF}{\mathscr{F}}
\newcommand{\mcG}{\mathscr{G}}

\definecolor{lightblue}{rgb}{0.22,0.45,0.70}   % light blue
\definecolor{darkgray}{gray}{0.4}    % dark grey
\definecolor{lightgray}{gray}{0.8}

\title{\sc A not-so-strange term coming from somewhere}
\def\correspondingauthor{\footnote{Corresponding author: giacomo.canevari@univr.it}}
\author[1]{Giacomo Canevari\correspondingauthor{}}
\author[2]{Kirill Cherednichenko}
\author[3,4,5]{Arghir Zarnescu}
\affil[1]{Università degli Studi di Verona,
Strada le Grazie 15, 37134 Verona, Italy
\vspace{2mm}}
\affil[2]{Department of Mathematical Sciences, University of Bath,
\linebreak Claverton Down, Bath, BA2 7AY, United Kingdom \vspace{2mm}}
\affil[3]{Simion Stoilow Institute of Mathematics of the Romanian Academy, \linebreak
P.O. Box 1-764, RO-014700 Bucharest, Romania \vspace{2mm}}
\affil[4]{BCAM, Basque Center for Applied Mathematics,
\linebreak Mazarredo 14, 48009, Bilbao, Bizkaia, Spain \vspace{2mm}}
\affil[5]{IKERBASQUE, Basque Foundation for Science,
\linebreak Plaza Euskadi 5, 48009, Bilbao, Bizkaia, Spain}

\date{\today}

\begin{document}

\maketitle

\begin{abstract}
 We consider Laplace’s equation in a periodically perforated domain with Robin boundary conditions on the holes, where the Robin coefficient is scaled proportionally to the inverse total surface area of the perforations.

 We identify a regime in which surface and bulk effects contribute at the same order and show that the homogenised equation contains an additional zeroth-order term depending nonlinearly on the Robin parameter. This term is characterised via a Steklov-type spectral problem in which the spectral parameter appears both in the equation and in the boundary condition.

 The resulting term interpolates continuously between the Neumann and Dirichlet limits, recovering the classical capacitary strange term in the strong-coupling limit.
\end{abstract}

\section{Introduction}

Homogenisation of boundary value problems for differential operators on perforated domains is a classical topic. In this paper, we identify a critical regime for Robin boundary conditions in which the effective behaviour interpolates between the classical Neumann and Dirichlet limits. More precisely, when the Robin coupling is scaled by the inverse total surface area of the perforations, the homogenised equation acquires a zeroth-order term whose coefficient depends nonlinearly on the Robin parameter and is characterised through a Steklov-type spectral problem. This provides a unified framework bridging previously distinct asymptotic regimes, for the Dirichlet and Neumann problem, respectively.

\smallskip
The case of the Laplacian on a domain with ``stiff'' inclusions (with Dirichlet boundary conditions) in ${\mathbb R}^n$ was first studied by Marchenko and Hruslov \cite{MarchenkoKhruslov, Marchenko_Khruslov_book} in the context of scattering by domains with fine-grained boundaries, see also \cite[Section 3.6]{JikovKozlovOleinik}, and subsequently by Cioranescu and Murat \cite{CioranescuMurat, CioranescuMurat-eng}, who coined the expression ``le terme \'{e}trange'' for the effect of the perforations on the homogenised description. Denoting by $\varepsilon$ the typical distance between the inclusions, the effect occurs if the radii $r_\varepsilon$ of the inclusions are scaled as $r_\varepsilon\sim c\varepsilon^{n/(n-2)}$. In this case, the perforated domain Laplacian $-\Delta_\varepsilon$ converges to the operator $-\Delta+\mu$ on the domain without perforations, where $\mu$ is proportional to the harmonic capacity of a single perforation in ${\mathbb R}^n$. Perforated-domain problems have since been studied extensively in various variants, including those with Neumann, and more generally Robin, boundary conditions on the boundaries of the perforations --- see, e.g., \cite{Brillard, AnsiniBraides, ConcaMuratTimofte, AnsiniBabadjianZeppieri, Focardi, ChiadoPiat_Nazarov_Piatnitski, Nazarov2015, GomezNazarovPerez, Borisov2022, Borisov2024, ScardiaZemasZeppieri}. They are known (at least in some cases) to yield explicitly characterised limit operators, to which the operators of the perforated problems converge in the norm-resolvent sense \cite{CDR}, sometimes with error control in terms of~$\varepsilon$ \cite{KhrabustovskyiPost, KhrabustovskyiPlum}.

\smallskip
These developments have led to two research trends, one mathematically motivated and the other one related to physical problems applications. The mathematical trend concerns the role of the condition on the boundaries of the perforations on the asymptotic analysis as $\varepsilon\to0$ and its quantitative effect on the limit problem. The applied trend is concerned with the problem of accounting for the effect of the surface energy on the asymptotics of the total energy stored within the medium. This  has been prompted from the modelling perspective by issues in continuum mechanics. For example, the recent works \cite{Calderer_et_al, canevari2020design} have investigated questions of this type in the setting of liquid crystals.

\smallskip
Clearly, the existing results for the Dirichlet problem (``infinitely stiff'' inclusions) and the Robin problem (including the Neumann problem in the case of vanishing Robin coupling) leave a gap between the two extreme cases of Robin boundary conditions, namely the Dirichelt and the Neumann case. A natural question is whether there exists a family of (possibly $\varepsilon$-parametrised) perforated problems whose effective descriptions bridge this gap.

In the present paper, we provide such a  family of problems on perforated domains with $\varepsilon$-dependent surface energy terms --- see Remark \ref{rk:CM} below. Perhaps unsurprisingly, the scaling coefficient for the Robin parameter is provided by the inverse total surface area of the perforations. This balances the bulk and surface energy components for functions whose pointwise magnitude remains uniformly controlled and creates a genuine competition between the two effects for minimising sequences of the total energy; see \eqref{energy} below.

\smallskip
The proof of our main result, Theorem \ref{th:strange_term}, reveals a connection between the convergence of energy minimisers in this critical regime and a trace inequality, i.e. an inequality between suitably scaled boundary and bulk energy integrals, see \eqref{trace-intro}, where the optimal constant is provided by the lowest eigenvalue of a Steklov-type boundary-value problem, see \eqref{eigenvalue}, with the spectral parameter present in both the equation and the boundary condition. The said eigenvalue is characterised by the minimisation problem \eqref{lambda_eps-}, involving a key parameter $\kappa$ that enters the effective ``strange-term'' formulation \eqref{hom_eq}.

\smallskip
This connection with Steklov-type eigenvalue problems places the problem we study among modelling frameworks where Dirichlet-to-Neumann maps emerge in the quantification of the effect of inclusions on long-wave propagation, mainly in periodic settings. The inclusions can be small relative to the separation distance, may have contrasting material properties, or may involve a specific scaling for the Robin coupling between the solution and the boundary flux. The related works either begin by assuming Steklov boundary conditions \cite{ChiadoPiat_Nazarov_Piatnitski, GomezNazarovPerez, Girouard_et_al, Girouard_Lagace} and derive the effective medium, or lead to Dirichlet-to-Neumann maps through asymptotic analysis of transmission conditions \cite{CKS_2022, CKVZ}. It is noteworthy that, in our setting, the asymptotic behaviour of \eqref{main_eq} is governed by eigenvalues of a Steklov-type problem (i.e., Problem~\eqref{eigenvalue} below), despite the original formulation involving Robin conditions scaled by the total surface area of the perforations.

We conclude this section by noting that $L^2$ convergence estimates of the kind we obtain here form a natural step towards the quantitative analysis of evolution problems in inhomogeneous media. A relevant early contribution is \cite{Cioranescu_et_al_wave}. Another research direction is the study of analogous surface effects when the perforations are located along lower-dimensional manifolds; see, for example, \cite{Borisov_Cardone_Durante}.

\medskip
\noindent\textbf{Main novelty.}
The main novelty of this work lies in the identification of a critical scaling of the Robin interaction that balances surface and bulk effects, leading to a new homogenised term characterised via a Steklov-type spectral problem and interpolating between the classical Neumann and Dirichlet regimes.

\medskip
We now proceed to the description of the model problem we use to illustrate the general method developed in the present work (Section \ref{sect:abstract}) and its results --- in particular, the convergence estimates stated in Theorem \ref{th:mainintro} for the model case.

\paragraph{Statement of the problem.}

Let~$\Omega\subseteq\R^n$ be a bounded Lipschitz domain.
While some of our results (particularly, those
related to the ``model example'' described
in~\eqref{model_example}--\eqref{gamma_intro} below)
assume that the dimension satisfies~$n\geq 3$, others
(those in the ``abstract setting'', see Sections~\ref{sect:abstract}, \ref{sect:eigen},
and~\ref{sect:homog}) cover the case~$n = 2$ as well.
For each~$\eps\in (0, \, 1)$,
let~$\mcP_\eps\subseteq\R^n$ be a closed,
$\eps\Z^n$-periodic set (i.e., $\eps z + x\in\mcP_\eps$
for all~$z\in\Z^n$ and all~$x\in\mcP_\eps$).
% which in what follows represents the ``holes'' or ``perforations''.
Note that~$\mcP_\eps$ is \emph{not} a subset of~$\Omega$.
The intersection~$\mcP_\eps\cap\Omega$ represents
the ``holes'' or ``perforations'' in the domain~$\Omega$.
However, by abuse of terminology, we will say that~$\mcP_\eps$,
as a whole, represents the perforations.
We work in the perforated domain~$\Omega_\eps := \Omega\setminus\mcP_\eps$
and impose a Robin boundary condition on the boundary of the holes,
$\Gamma_\eps := \Omega\cap\partial\mcP_\eps$.
More precisely, for each~$\eps\in (0, \, 1)$ we consider the solution~$u_\eps\in H^1(\Omega_\eps)$ of
\begin{equation} \label{main_eq} \tag{P$_\eps$}
 \begin{cases}
  -\Delta u_\eps + \alpha u_\eps = f & \textrm{in } \Omega_\eps, \\
  \hspace{0.8mm} \partial_\nu u_\eps
   + \dfrac{\beta\,u_\eps}{\mu_\eps}
   = \dfrac{g}{\mu_\eps}
   & \textrm{on } \Gamma_\eps, \\
  u_\eps = 0 &\textrm{on } \partial\Omega_\eps\cap\partial\Omega.
 \end{cases}
\end{equation}
Here~$\nu = \nu_\eps$ is the exterior unit normal
to~$\Omega_\eps$, $\alpha\geq 0$ and~$\beta > 0$ are given parameters,
$f\in L^2(\Omega)$ a given source density,
$g\in C^0(\overline{\Omega})$ is also given,
and~$(\mu_\eps)_{0 < \eps < 1}$
is any family of positive numbers that satisfy
\begin{equation} \label{mu_eps}
 \lim_{\eps\to 0} \frac{\abs{\Omega} \mu_\eps}
  {\sigma(\Gamma_\eps)} = 1,
\end{equation}
where~$\abs{\Omega}$ denotes the Lebesgue measure of~$\Omega$
and~$\sigma$ is the surface area
(i.e., the $(n-1)$-di\-men\-sion\-al Hausdorff measure).
In particular, $\mu_\eps$
is comparable to the total surface area of~$\Gamma_\eps$
--- namely, there exist $\eps$-independent positive constants~$c_1$, $c_2$ satisfying
\begin{equation} \label{mu_eps-surface}
 c_1 \, \sigma(\Gamma_\eps) \leq \mu_\eps
 \leq c_2 \, \sigma(\Gamma_\eps)
\end{equation}
for all sufficiently small~$\eps$.
The term~$g/\mu_\eps$ is the surface density of sources located on~$\Gamma_\eps$.
We are interested in the asymptotic behaviour of solutions~$u_\eps$
in the limit as~$\eps\to 0$.
% Homogenisation problems of this type
% have been extensively studied for a long time (see, for instance,
% \cite{MarchenkoKhruslov, RauchTaylor, CioranescuMurat, Kaizu}).
As we explained earlier, the main difference between Problem~\eqref{main_eq} and earlier ones in this subject area (see e.g.~\cite{MarchenkoKhruslov, RauchTaylor, CioranescuMurat, Kaizu, JikovKozlovOleinik}) is the normalisation factor~$\mu_\eps$ in the Robin boundary conditions.
As we shall see later, this factor produces a
``strange term'' in the homogenised equation that depends nonlinearly on~$\beta$.

Problem~\eqref{main_eq} arises as the first-order optimality condition for minimisers of the functional
\begin{equation} \label{energy}
 \mcF_\eps(u) :=
  \int_{\Omega}\left(\frac{1}{2}\abs{\nabla u}^2
   + \frac{\alpha}{2} u^2 - f u\right) \d x
   + \frac{1}{\mu_\eps}
   \int_{\Gamma_\eps} \left(\frac{\beta}{2} u^2 - g u \right) \d\sigma,
\end{equation}
over~$u\in H^1(\Omega_\eps)$ such that~$u = 0$
on~$\partial\Omega_\eps\cap\partial\Omega$.
Functionals with this structure appear naturally in many contexts.
For instance, in materials science, they serve
as models for composite materials
--- liquid crystal colloids, among many others ---
consisting of many micro- or nano-particles dispersed
in a host material, with molecular interactions between
the host and the particles localised around the surfaces~\cite{Bennett-PhysRevE,Calderer_et_al,canevari2020design}.
%Other contexts in which they appear are, for instance,
Other applications include, e.g.,
thermal insulation modelling~\cite{della2021optimization} or in biology, modelling suspensions of rod-shaped viruses \cite{van2023homogenization}, lipid vesicle membranes \cite{wang2008modelling}.

From the variational point of view, it is natural
for the normalisation coefficient~$\mu_\eps$
in~\eqref{energy} to satisfy the bounds~\eqref{mu_eps-surface}.
Indeed, this choice guarantees that surface and bulk energies
% of a smooth, bounded competitor~$u\in (H^1\cap L^\infty)(\Omega)$
are of the same order as~$\eps \to 0$.
Condition~\eqref{mu_eps} is more convenient
than~\eqref{mu_eps-surface} for the purposes of our analysis,
because it allows solutions of~\eqref{main_eq}
to converge to a uniquely defined limit problem as~$\eps\to 0$,
while~\eqref{mu_eps-surface} only provides compactness
of the solutions.
However, if a family~$(\mu_\eps)_{0 < \eps < 1}$
satisfies~\eqref{mu_eps-surface}, then we can reduce to the
case~\eqref{mu_eps} by extracting
a subsequence~$\eps_j\to 0$ in such a way
that~$\mu_{\eps_j}/\sigma(\Gamma_{\eps_j})$ converges,
then suitably rescaling~$\mu_{\eps_j}$ and~$\beta$.

%In our analysis, the perforation set~$\mcP_\eps$ is
%supposed to satisfy certain conditions.
%Instead of describing the general setting,
%we first focus on a model example to which our analysis applies;
%we will describe the more general setting in Section~\ref{sect:abstract}.
We will prove a homogenisation result
for all perforation sets~$\mcP_\eps$ satisfying certain abstract
conditions (see Theorem~\ref{th:strange_term} below).
However, to simplify the exposition, we first present
a model example to which our analysis applies
and leave the more general framework to Section~\ref{sect:abstract}.
As mentioned above, in our model example we assume~$n\geq 3$ and define
\begin{equation} \label{model_example}
 \mcP_\eps := \eps\Z^n + \eps^\gamma \overline{D}
 = \left\{\eps z + \eps^\gamma x\colon z\in\Z^n, \ x\in \overline{D}\right\} \!,
\end{equation}
where~$D\subseteq\R^n$ is a bounded open set
of class~$C^2$ such that~$\R^n\setminus\overline{D}$ is connected,
and~$\gamma$ is set to be
\begin{equation} \label{gamma_intro}
 \gamma := \frac{n}{n - 2}.
\end{equation}
Note that we assume some additional regularity on the set~$D$, i.e.~$C^2$ instead of Lipschitz, because in the sequel we will need to apply elliptic regularity estimates near the hole boundary; see the proof of Proposition~\ref{prop:example} and condition~\eqref{hp:lambda0} that is used there.
As an example of a family~$(\mu_\eps)_{0 < \eps < 1}$
satisfying~\eqref{mu_eps} in this case, we can take
\begin{equation} \label{mu_eps_bar_intro}
 \mu_\eps = \overline{\mu}_\eps := \eps^{\gamma}\sigma(D).
\end{equation}
By an explicit computation, we see that such~$\overline{\mu}_\eps$
is precisely the total surface area of~$\eps^{-n}$
copies of a single ``hole'', $\eps^\gamma\overline{D}$,
hence it satisfies~\eqref{mu_eps}
(see Remark~\ref{rk:mu_eps_bar} for the details).
Other choices of~$\mu_\eps$ are admissible, so long
as they satisfy~\eqref{mu_eps}.
Holes of the form~\eqref{model_example}--\eqref{gamma_intro}
correspond exactly to the critical scaling
for the homogenisation of a Dirichlet problem, i.e.
\begin{equation} \label{main_eq_Dir}
 u^{\mathrm{Dir}}_\eps \in H^1_0(\Omega_\eps), \qquad
 -\Delta u^{\mathrm{Dir}}_\eps
% + \alpha u^{\mathrm{Dir}}_\eps
  = f \quad \textrm{in } \Omega_\eps.
\end{equation}
Cioranescu and Murat~\cite{CioranescuMurat} proved that
solutions to this problem
satisfy~$\norm{u_\eps - u_0}_{L^2(\Omega_\eps)}\to 0$
as~$\eps\to 0$, where~$u_0$ is
the unique solution~$u^{\mathrm{Dir}}_0 \in H^1_0(\Omega)$ to
\begin{equation} \label{hom_eq_Dir}
 -\Delta u^{\mathrm{Dir}}_0 %+ \left(\alpha
  + \Cap_* %\right)
  u^{\mathrm{Dir}}_0 = f
   \quad \textrm{in } \Omega.
\end{equation}
Here~$\Cap_*$ is a ``strange term'', defined as the solution
to a capacity problem, i.e.~$\Cap_*$ is the Dirichlet energy
of a harmonic function that is equal to one on~$\partial D$
and decays to zero at infinity.
For instance, when~$D$ is the unit ball~$B^n$ we have
\[
 \Cap_* = (n-2) \sigma_n,
\]
where~$\sigma_n := \sigma(\partial B^n)$.
By contrast, Kaizu~\cite{Kaizu} proved a homogenisation
result for the Robin problem
\begin{equation} \label{main_eq_K}
 \begin{cases}
  -\Delta u^{\mathrm{K}}_\eps %+ \alpha u^{\mathrm{K}}_\eps
   = f
   &\textrm{in } \Omega_\eps, \\
  \hspace{2mm} \partial_{\nu} u^{\mathrm{K}}_\eps
   + \beta u^{\mathrm{K}}_\eps = 0
   &\textrm{on } \partial\Omega_\eps,
 \end{cases}
\end{equation}
where the critical scaling of the holes is different,
namely, each component of~$\mcP_\eps$
is a ball of radius~$\eps^{n/(n-1)}\ll \eps^\gamma$.
In this case, the limit problem takes the form
\begin{equation} \label{hom_eq_K}
 \begin{cases}
  -\Delta u^{\mathrm{K}}_0 %+ \left(\alpha
   + \beta\sigma_n%\right)
   u^{\mathrm{K}}_0 = f
   &\textrm{in } \Omega, \\
  \hspace{2mm} \partial_{\nu} u^{\mathrm{K}}_0
   + \beta u^{\mathrm{K}}_0 = 0
   &\textrm{on } \partial\Omega.
 \end{cases}
\end{equation}
Therefore, introducing the normalisation coefficient~$\mu_\eps$
in~\eqref{main_eq} makes the problem closer to~\eqref{main_eq_Dir}
than to~\eqref{main_eq_K}, in terms of the critical scaling but also,
as we will see shortly, in terms of the expression for
the ``strange term'' we obtain in the limit problem.

\paragraph{The homogenisation limit (in the model example).}

The homogenised problem is defined in terms of auxiliary
variational problems that depend on the holes~$\mcP_\eps$.
For instance, for our model example~\eqref{model_example},
we take a parameter~$\kappa\in (0, \, 1)$ and consider
\begin{equation} \label{model_lambda*}
 \lambda_*(\kappa)
  = -\inf\left\{\int_{\R^{n}\setminus \overline{D}}
   \abs{\nabla z}^2\colon z\in \dot{H}^1(\R^n\setminus\overline{D})
   \oplus\R , \ \fint_{\partial D} z^2 \,\d\sigma
   - \kappa \, z(\infty)^2 = -1 \right\} \! .
\end{equation}
Here~$\dot{H}^1(\R^n\setminus\overline{D})$ is a
homogenous Sobolev space, whose elements
decay to zero at infinity and have square-integrable gradient,
while~$\dot{H}^1(\R^n\setminus\overline{D})\oplus\R$ is the set
of functions~$z$ for which there is a constant~$z(\infty)\in\R$
such that~$z - z(\infty)\in \dot{H}^1(\R^n\setminus\overline{D})$.
The notation~$\fint_{\partial D}$ denotes the integral
average over~$\partial D$, that is,
\[
 \fint_{\partial D} z^2 \, \d\sigma
 := \frac{1}{\sigma(\partial D)} \int_{\partial D} z^2 \, \d\sigma.
\]
The negative sign on the right-hand side of~\eqref{model_lambda*}
is introduced for later convenience.
It can be shown (see Theorem~\ref{th:lambda*} below)
that for each~$\beta > 0$ there exists a
unique~$\kappa_* = \kappa_*(\beta) \in (0, \, 1)$
satisfying
\[
 -\lambda_*(\kappa_*(\beta)) = \beta.
\]
We define our candidate homogenised problem in
terms of~$\kappa_*(\beta)$ as follows:
\begin{equation} \label{hom_eq} \tag{P$_0$}
 \begin{cases}
  -\Delta u_0 + \alpha u_0 + \beta\kappa_*(\beta) u_0
   = f + \kappa_*(\beta) g & \textrm{in } \Omega, \\
  u_0 = 0 &\textrm{on } \partial\Omega.
 \end{cases}
\end{equation}
The boundary terms proportional
to~$\mu_\eps^{-1}$ in~\eqref{main_eq} lead to
bulk terms proportional to~$\kappa_*(\beta)$ in~\eqref{hom_eq}.
Problem~\eqref{hom_eq} has a unique solution~$u_0\in H^1(\Omega)$.

\begin{theorem} \label{th:mainintro}
 Let~$\mcP_\eps$ be defined as in~\eqref{model_example}--\eqref{gamma_intro},
 with~$\mu_\eps$ satisfying~\eqref{mu_eps}. Then,
 for all~$f\in L^2(\Omega)$, all~$g\in C^0(\overline{\Omega})$
 and all~$\alpha\geq 0$, $\beta > 0$, the solutions~$u_\eps$
 to~\eqref{main_eq} satisfy
 \[
  \lim_{\eps\to 0} \norm{u_\eps - u_0}_{L^2(\Omega_\eps)} = 0.
 \]
\end{theorem}

Actually, we can prove a convergence result in a more general,
abstract framework (see Theorem~\ref{th:strange_term} below),
and obtain an ``abstract'' convergence rate
(Proposition~\ref{prop:abstractrate})
along the lines of~\cite[Th\'eor\`eme~1.1]{KacimiMurat}.
Theorem~\ref{th:mainintro} extends to the case
of $\eps$-dependent source terms~$f_\eps\in L^2(\Omega)$
and~$g_\eps\in C^0(\overline{\Omega})$ in the
right-hand side of~\eqref{main_eq},
so long as~$f_\eps\rightharpoonup f$ weakly in~$L^2(\Omega)$
and~$g_\eps\to g$ uniformly (see Remark~\ref{rk:f_eps} below).
When the holes are spherical, i.e.~$D = B^n$,
the minimisation problem~\eqref{model_lambda*}
leading to a formula for the ``(not so) strange term'' in~\eqref{hom_eq}:
\begin{equation} \label{kappa_*-spheres}
 \kappa_*(\beta) = \frac{\sigma_n (n-2)}{\sigma_n (n-2) + \beta},
\end{equation}
where~$\sigma_n$ is the surface area of the unit sphere in~$\R^n$.
In particular, $\beta\kappa_*(\beta)$ converges to zero in the limit
as~$\beta\to 0$, which approximates the Neumann boundary condition.
On the other hand, in the limit~$\beta\to +\infty$
(which approximates the Dirichlet boundary condition in the limit problem)
the coefficient $\beta\kappa_*(\beta)$ tends to~$\sigma_n (n-2)$,
which is precisely the ``strange term'' as given in~\cite{CioranescuMurat}.
For more general choices of~$D$, we do not have an
explicit expression of~$\kappa_*(\beta)$,
but it remains true that~$\beta\kappa_*(\beta)$
converges to zero as $\beta\to 0$ and to the capacitary term~$\Cap_*$ in~\eqref{hom_eq_Dir} as~$\beta\to+\infty$ (see Remark~\ref{rk:CM}).
 The main technical tool in the proof of Theorem~\ref{th:mainintro}
is the analysis of the following Steklov-type eigenvalue problem.
The problem is set in a rescaled unit cell~$Y := (-1.2, \, 1/2)^n$, which we identify with unit torus~$\T^n$ by imposing periodic boundary conditions on~$\partial Y$. Given a parameter~$\kappa > 0$ with~$\kappa\neq 1$ and~$\eps$ small enough, we consider ground-state solutions~$v_{\eps,\kappa}$ to
\begin{equation} \label{eigenvalue-intro}\tag{EV$_{\eps,\kappa}$}
 \begin{cases}
  -\Delta v_{\eps,\kappa}
   + \kappa\, \lambda(\eps, \, \kappa) \,  v_{\eps,\kappa} = 0
   & \textrm{in } \T^n\setminus \eps^{\gamma - 1}\overline{D}, \\[8pt]
  \partial_\nu v_{\eps,\kappa}
   = \dfrac{\lambda(\eps, \, \kappa)\, v_{\eps,\kappa}}{\eps^{\gamma + 2}\sigma(\partial D)}
   & \textrm{on }  \partial(\eps^{\gamma - 1} D),
 \end{cases}
\end{equation}
where~$\nu$ is the unit normal to~$\partial(\eps^{\gamma - 1} D)$
pointing inside~$H_\eps$. The purpose of Problem~\eqref{eigenvalue-intro} in our analysis is two-fold: when~$0 < \kappa < 1$, we find non-trivial solutions~$v_{\eps,\kappa}$ with \emph{negative} spectral parameter~$\lambda(\eps, \, \kappa)$ and use them to construct correctors for the homogenisation problem (see Lemma~\ref{lemma:corrector}); when~$\kappa > 1$, we consider the least \emph{positive} value~$\lambda(\eps, \, \kappa)$ for which nontrivial solutions to~\eqref{eigenvalue-intro} exist, which is involved in a trace inequality on the perforated domain~$\Omega_\eps$ (see Lemma~\ref{lemma:trace}).

We investigate other scalings for the hole diameter, %of the perforations,
that is, we consider the set
$\mcP_\eps = \eps\Z^n + r_\eps\overline{D}$
for more general choices of~$r_\eps > 0$
(see Proposition~\ref{prop:otherregimes}
in Section~\ref{sect:otherregimes}).
It turns out that when
\[
 \lim_{\eps\to 0}\frac{r_\eps}{\eps^\gamma} = 0,
\]
the solutions of~\eqref{main_eq} converge to the unique
solution~$u_0\in H^1_0(\Omega)$ of
\[
 -\Delta u_0 + \alpha u_0 = f
 \qquad \textrm{in } \Omega,
\]
while if
\[
 \lim_{\eps\to 0}\frac{r_\eps}{\eps^\gamma} = +\infty,
 \qquad \lim_{\eps\to 0} \frac{r_\eps}{\eps} = 0,
\]
then the limit~$u_0\in H^1_0(\Omega)$ satisfies
\[
 -\Delta u_0 + \alpha u_0 + \beta u_0 = f + g
 \qquad \textrm{in } \Omega.
\]
The scaling~$r_\eps = \eps^{\gamma}$
is critical, in that it is the only one
leading to ``(not so) strange terms''
with a coefficient~$0 < \kappa_*(\beta) < 1$.

\paragraph*{Outline of the paper structure.}

In Section \ref{sect:abstract}, we provide a general description of the setting to which our methodology and the main result concerning the $L^2$-convergence of solutions $u_\varepsilon$ (see Theorem \ref{th:mainintro}) as well as the estimate for the case of bounded data $f$ (see Proposition \ref{prop:abstractrate}) apply. This is then followed by showing how the model example introduced above fits into the general framework (Section \ref{model_example_fit_sec}).

In Section \ref{sect:eigen} we discuss the function  $\lambda_*(\kappa),$ see (\ref{model_lambda*}), which represents the limit (as $\varepsilon\to0$) of the scaled eigenvalues $\lambda(\varepsilon, \kappa)$ of the eigenvalue problem~\eqref{eigenvalue} below, with the spectral parameter present in both the equation and the boundary condition. As discussed above, the function $\lambda_*(\kappa)$ is crucial for determining an appropriate variant of the ``strange term'' for the problem \ref{hom_eq}. The link between $\lambda_*(\kappa)$ and the values $\St_*$ and $\Dir_*$, introduced in Section~\ref{sect:abstract} (see Theorem~\ref{th:lambda*}
and Proposition~\ref{prop:lambda0}), is proved in Section \ref{sect:lambda*}.

Section \ref{main_proof} begins with a discussion of trace inequalities (Lemma \ref{lemma:trace} and Lemma \ref{lemma:singletrace-N}) and continues with a proof of the main result of the paper, Theorem \ref{th:strange_term}.

In Section \ref{sect:example} the spectral problem \eqref{eigenvalue} and the associated function $\lambda_*(\kappa)$ is considered for the model example. The section is concluded by a discussion of ``subcritical'' (smaller perforations) and ``supercritical'' (larger perforations) asymptotic regimes (Section \ref{sect:otherregimes}).

Finally, in Section~\ref{sect:many_holes}
we discuss another example, namely, a perforation set~$H_\eps$
that satisfies all the assumptions of our abstract setting,
but has $\sigma(\Gamma_\eps)\to +\infty$
as~$\eps\to 0$, which implies~$\mu_\eps\to 0$
as~$\eps\to 0$ in view of~\eqref{mu_eps_bar_intro}.
This example falls within the applicability range
of our homogenisation result and produces a non-trivial
``strange term'' in the limit problem, even though
the Robin coefficient~$\beta/\mu_\eps$ in~\eqref{main_eq}
becomes negligible as~$\eps\to 0$.

% Number the equations and the theorems according to the section (after the introduction of the paper)
\numberwithin{equation}{section}
\numberwithin{definition}{section}
\numberwithin{theorem}{section}

\section{An abstract setting and the model example}
\label{sect:abstract}

Instead of restricting our attention to the model
example~\eqref{model_example}, we prove a homogenisation result
for a more general family of holes that satisfy a set
of abstract assumptions, \eqref{hp:first}--\eqref{hp:last} below.
Then, we will show that the model example
satisfies~\eqref{hp:first}--\eqref{hp:last}.
Throughout the paper, we denote the Lebesgue measure and the
$(n-1)$-dimensional Hausdorff measure
of a set~$E\subseteq\R^n$ by~$\abs{E}$, $\sigma(E)$, respectively.
Given a set~$E$ with~$0 < \sigma(E) < +\infty$,
we will denote by~$\fint_{E}$ the integral average
over~$E$, that is,
\[
 \fint_{E} u \, \d\sigma :=
 \frac{1}{\sigma(E)} \int_E u\, \d\sigma
\]
for any integrable function~$u\colon E\to\R$.
We will denote by~$B_\rho^n(x_0)$ the open ball in~$\R^n$
of radius~$\rho > 0$ centred at~$x_0$.
In the case when the centre is the origin,
we will write~$B_\rho^n := B_\rho^n(0)$ and~$B^n := B^n_1$.

\begin{figure}[t]
 \centering
 \includegraphics[height=.3\textheight]{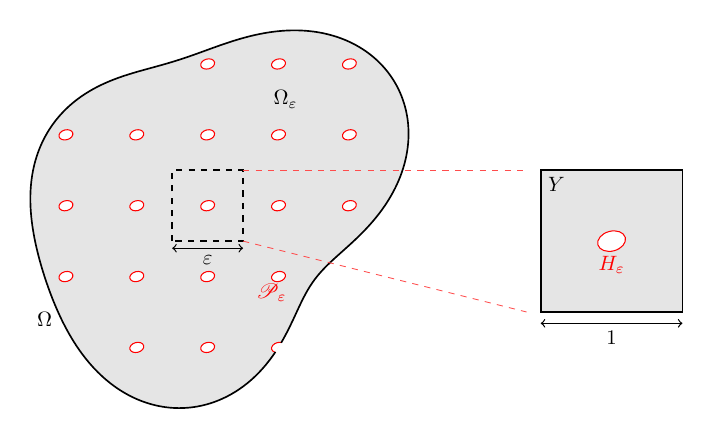}
 \caption{The domain and the unit periodicity cell~$Y$.}
 \label{fig:domain}
\end{figure}

\subsection{Description of the abstract setting}
\paragraph*{Geometry of the holes.}

Let~$Y := (-1/2, \, 1/2)^n\subseteq\R^n$ be the reference
unit cube in dimension~$n$, representing the unit periodicity cell.
For~$\eps\in (0, \, 1)$, we consider a closed subset~$H_\eps\subseteq Y$.
The set~$H_\eps$ represents a single ``hole'' in the domain,
rescaled by a factor of~$\eps$, as illustrated in Figure~\ref{fig:domain}.
(For instance, in our model example~\eqref{model_example},
we have~$H_\eps = \eps^{\gamma - 1} \overline{D}$.)

We make the following assumptions on~$H_\eps$:

\begin{enumerate}[label=(C\textsubscript{\arabic*}),
 ref=C\textsubscript{\arabic*}]
 \item \label{hp:H} \label{hp:first} %For all~$\eps\in (0, \, 1)$,
 $H_\eps$ is the closure of an open Lipschitz set.
 %such that~$Y\setminus H_\eps$ is connected.

 \item \label{hp:decreas} There exists a compact set~$K_0\subseteq Y$ such that~$H_\eps \subseteq K_0$ for all~$\eps \in (0, \, 1)$.

 \item \label{hp:extension} For all~$\eps \in (0, \, 1)$,
 there exists a bounded linear extension operator
 $E_\eps\colon H^1(Y\setminus H_\eps)\to H^1(Y)$
 that satisfies
 \[
  \norm{\nabla(E_\eps v)}_{L^2(Y)}
   \leq C \norm{\nabla v}_{L^2(Y\setminus H_\eps)}
 \]
 for all~$v\in H^1(Y\setminus H_\eps)$
 and some constant~$C>0$ that does \emph{not} depend on~$\eps$, $v$.
\end{enumerate}

It would be interesting to investigate whether
assumption~\eqref{hp:extension} could be weakened,
but we do not pursue this direction here.
There are other assumptions we make on the set~$H_\eps$,
namely, assumptions~\eqref{hp:critical} and~\eqref{hp:limit} below.
However, we need to introduce some notation
before stating them.

\paragraph*{The critical scaling.}

Let~$\T^n$ be the $n$-dimensional (flat) torus, defined as
as the quotient of the closed cube~$\overline{Y}:=[-1/2, \, 1/2]^n$
up to identification of the opposite faces.
We identify~$H_\eps$ with a subset of~$\T^n$.
In particular, we will write~e.g.~$\varphi\in H^1(\T^n)$
or~$\varphi\in H^1(\T^n\setminus H_\eps)$, meaning that the
function~$\varphi\in H^1(Y)$ or~$\varphi\in H^1(Y\setminus H_\eps)$,
respectively, satisfies periodic boundary
conditions on~$\partial Y$ in the sense of traces.
For each~$\eps$, we let~$\Dir(\eps) > 0$ be the first eigenvalue
of the Laplacian in~$\T^n\setminus H_\eps$ with Dirichlet boundary
some condition on~$\partial H_\eps$, i.e.
\begin{equation} \label{Dir_eps}
 \Dir(\eps) := \inf\left\{\int_{\T^n}
  \abs{\nabla\varphi}^2 \, \d x\colon
 \varphi\in H^1(\T^n), \quad \varphi = 0
 \textrm{ in } H_\eps, \quad
 \int_{\T^n\setminus H_\eps} \varphi^2 \, \d x = 1  \right\} \! .
\end{equation}
The infimum on the right-hand side of~\eqref{Dir_eps} is attained
at a unique minimiser~$\varphi_\eps$,
which is positive in~$\T^n\setminus H_\eps$
and satisfies
\begin{equation} \label{Dir_eps_eigen}
 \begin{cases}
  -\Delta\varphi_\eps = \Dir(\eps) \varphi_\eps
   & \textrm{in } \T^n\setminus H_\eps, \\
  \varphi_\eps = 0
   & \textrm{on } \partial H_\eps.
 \end{cases}
\end{equation}
Henceforth, boundary conditions are understood
in the sense of traces. We also consider the value
\begin{equation} \label{St_eps}
 \St(\eps) := \inf\left\{\int_Y\abs{\nabla\psi}^2 \, \d x\colon
 \psi\in H^1(Y\setminus H_\eps), \quad
 \fint_{\partial H_\eps}\psi^2 \, \d\sigma=1,
 \quad \psi = 0 \textrm{ on } \partial Y \right\} \! ,
\end{equation}
which is, up to a normalisation constant,
the lowest positive eigenvalue for a
Steklov-type spectral problem on
the perforated cell~$Y\setminus H_\eps$. (For the classical Steklov problem, see, e.g., \cite{GirouardPolterovich, Levitin_et_al}.)
More precisely, there exists a minimiser~$\psi_\eps$ for the
right-hand side of~\eqref{St_eps}, unique up to the sign,
and it satisfies
\begin{equation} \label{St_eps_eigen}
 \begin{cases}
  -\Delta\psi_\eps = 0 & \textrm{in } Y\setminus H_\eps, \\
  \partial_\nu \psi_\eps
   = \dfrac{\St(\eps)\,\psi_\eps}{\sigma(\partial H_\eps)}
   & \textrm{on } \partial H_\eps, \\
  \psi_\eps = 0 & \textrm{on } \partial Y,
 \end{cases}
\end{equation}
where~$\nu$ is the unit normal to~$\T^n\setminus H_\eps$
that points inside~$H_\eps$.
We are interested in the asymptotic regime
for which both~$\Dir(\eps)$ and~$\St(\eps)$ are of order~$\eps^2$,
i.e.~we make the following assumption:

\begin{enumerate}[label=(C\textsubscript{\arabic*}),
 ref=C\textsubscript{\arabic*}, resume]
 \item \label{hp:critical} The limits
 \[
  \Dir_* := \lim_{\eps\to 0} \frac{\Dir(\eps)}{\eps^2},
  \qquad \St_* := \lim_{\eps\to 0} \frac{\St(\eps)}{\eps^2}
 \]
 exist, are strictly positive and finite.
\end{enumerate}

This particular scaling turns out to be
related to some trace inequality for the boundary trace
on~$\Gamma_\eps$ (see Equation~\eqref{trace-intro} below).

Before proceeding with the last assumption~\eqref{hp:limit},
we mention that~$\Dir_*$ has a different
interpretation, in terms of a capacitary problem.
For each~$\eps\in (0, \, 1)$, we let
\begin{equation} \label{Cap_eps}
 \Cap(\eps) := \inf\left\{\int_Y\abs{\nabla\zeta}^2 \, \d x\colon
 \zeta\in H^1(Y), \quad \zeta = 0
 \textrm{ a.e. in } H_\eps, \quad
 \zeta = 1 \textrm{ on } \partial Y \right\}
\end{equation}
% where again the boundary condition on~$\partial Y$
% is imposed in the sense of traces.
Instead of imposing~$\zeta = 1$ in~$H_\eps$
and~$\zeta = 0$ on~$\partial Y$,
as is usually done, we did the opposite,
to facilitate comparison with~\eqref{Dir_eps}.
The infimum on the right-hand side of~\eqref{Cap_eps}
is attained at a unique minimiser~$\zeta_\eps$,
which is positive in~$Y\setminus H_\eps$ and satisfies
\begin{equation} \label{Cap_eps_eigen}
 \begin{cases}
  -\Delta\zeta_\eps = 0 & \textrm{in } Y\setminus H_\eps, \\
  \zeta_\eps = 0        & \textrm{in } H_\eps, \\
  \zeta_\eps = 1        & \textrm{on } \partial Y.
 \end{cases}
\end{equation}
We also define
\begin{equation} \label{Cap_*}
 \Cap_* := \lim_{\eps\to 0} \frac{\Cap(\eps)}{\eps^2} \, ,
\end{equation}
so long as the limit exists.
If so, we have~$\St_* \leq \Cap_*$, because
$1 - \zeta_\eps$ is an admissible competitor
for the problem~\eqref{St_eps}, which defines~$\St(\eps)$.

\begin{prop} \label{prop:DirCap}
 Under assumptions~\eqref{hp:first}--\eqref{hp:critical},
 we have~$\Cap_* = \Dir_*$. In particular,
 the limit in~\eqref{Cap_*} exists.
\end{prop}

\begin{remark} \label{rk:zerocapacity}
 Proposition~\ref{prop:DirCap} and
 the Poincar\'e inequality in~$Y$ imply that
 \[
  \abs{H_\eps} \leq \norm{1 - \zeta_\eps}^2_{L^2(Y)}
  \leq \Cap(\eps) \leq C \eps^2
 \]
 for some $\eps$-independent constant~$C$.
% \textbf{(I'm applying the Poincar\'e inequality
% for functions defined in~$Y$, so the constant it's
% uniform (probably $1$?).}
 Moreover, the quantity~$\Cap(\eps)$ is the condenser capacity of~$H_\eps$
 in~$Y$. If the sets~$H_\eps$ are monotonically increasing with~$\eps$, i.e.~$H_{\eps}\subseteq H_\delta$ for all~$0 < \eps < \delta$, then Proposition~\ref{prop:DirCap} implies
 that the $2$-capacity of~$H_0 := \cap_{\eps> 0} H_\eps$
 is equal to zero, and the Hausdorff
 dimension of~$H_0$ is at most~$n-2$.
% and we have $\abs{H_\eps} \to 0$ as~$\eps\to 0$.
\end{remark}

The value~$\Cap_*=\Dir_*$ is related to the ``strange term''
\`a la Cioranescu-Murat~\cite{CioranescuMurat} for the problems with
Dirichlet boundary conditions. Indeed, the results
of~\cite{CioranescuMurat} imply that the solution
to Problem~\eqref{main_eq_Dir}
converge, as~$\eps\to 0$, to the solution of~\eqref{hom_eq_Dir},
where the ``strange term'' is precisely~$\Cap_* = \Dir_*$.
It will be interesting to compare~$\Dir_*$ with the behaviour
of our strange term~$\beta \kappa_*(\beta)$ in~\eqref{hom_eq},
% for the problem with Robin boundary conditions,
in the limit as the Robin parameter~$\beta$ tends to infinity
(see, for instance, Remark~\ref{rk:CM}).

\paragraph*{A ``not-so-strange'' eigenvalue problem.}

For any given~$\eps \in (0, \, 1)$, $\kappa > 0$
and~$v\in H^1(\T^n\setminus H_\eps)$, let
\begin{equation} \label{G_eps}
 \mcG_{\eps,\kappa}(v) := \fint_{\partial H_\eps} v^2 \, \d\sigma
   - \kappa\int_{\T^n\setminus H_\eps} v^2 \, \d x.
\end{equation}
We define a quantity~$\lambda(\eps, \, \kappa)\in [-\infty, \, +\infty]$
by solving a suitable minimisation problem, as follows.
For~$\kappa > 1$, we set
\begin{equation} \label{lambda_eps+} \tag{MIN$_{\eps,\kappa}^+$}
 \lambda(\eps, \, \kappa) := \inf\left\{
  \int_{\T^n\setminus H_\eps}\abs{\nabla v}^2 \, \d x \colon
  v\in H^1(\T^n\setminus H_\eps),
  \quad \mcG_{\eps,\kappa}(v) = 1\right\} \! .
\end{equation}
Furthermore, for~$0 < \kappa \leq 1$, we set
\begin{equation} \label{lambda_eps-} \tag{MIN$_{\eps,\kappa}^-$}
 \lambda(\eps, \, \kappa) := -\inf\left\{
  \int_{\T^n\setminus H_\eps}\abs{\nabla v}^2 \, \d x \colon
  v\in H^1(\T^n\setminus H_\eps),
  \quad \mcG_{\eps,\kappa}(v) = -1\right\} \! .
\end{equation}
There is a negative sign in front of infimum on the right-hand side
of~\eqref{lambda_eps-}, but this is not the case for~\eqref{lambda_eps+};
this guarantees monotonicity for the
function~$\kappa\mapsto\lambda(\eps, \, \kappa)$
(see Equation~\eqref{lambdamonotone} below).
% In case~$\kappa > 1$, the quantity~$\lambda(\eps,\,\kappa)$
% is essentially the same as before, except that I
% have replaced~$1/\lambda(\eps, \, \kappa)$ with~$\lambda(\eps, \, \kappa)$
% and replaced natural boundary conditions
% on~$\partial Y$ with periodic ones. However, I think
% it makes sense for us to consider the case~$0 < \kappa < 1$ as well.}
As it turns out (see Lemma~\ref{lemma:exists} below),
if~$\kappa\abs{\T^n\setminus H_\eps} > 1$
or~$0 < \kappa \leq 1$, then~$\lambda(\eps, \, \kappa)$
is finite and non-zero and the infimum on the right-hand side of
either~\eqref{lambda_eps+} or~\eqref{lambda_eps-}, respectively, is attained.
Moreover, minimisers~$v_{\eps,\kappa}$ satisfy
\begin{equation} \label{eigenvalue} \tag{EV$_{\eps,\kappa}$}
 \begin{cases}
  -\Delta v_{\eps,\kappa}
   + \kappa\, \lambda(\eps, \, \kappa) \,  v_{\eps,\kappa} = 0
   & \textrm{in } \T^n\setminus H_\eps, \\[8pt]
  \partial_\nu v_{\eps,\kappa}
   = \dfrac{\lambda(\eps, \, \kappa)\, v_{\eps,\kappa}}{\sigma(\partial H_\eps)} & \textrm{on } \partial H_\eps,
 \end{cases}
\end{equation}
where~$\nu$ is the unit normal to~$\T^n\setminus H_\eps$ pointing inside~$H_\eps$.
In Problem~\eqref{eigenvalue}, the eigenvalue appears
both in the equation and the boundary condition
--- a situation considered in, e.g.,~\cite{Shkalikov, vonBelowGilles, Girouard_et_al}.
% \begin{equation} \label{lessstrange}
%  \begin{cases}
%    -\Delta v = \Lambda \,  v
%     & \textrm{in } \T^n\setminus H_\eps \\[8pt]
%    \hspace{2mm} \partial_\nu v = \sigma \Lambda v
%     & \textrm{on } \partial H_\eps,
%  \end{cases}
% \end{equation}
% where~$\sigma > 0$ is given.
However, Problem~\eqref{eigenvalue} differs
from the ones considered in~\cite{Shkalikov, vonBelowGilles, Girouard_et_al}
in that the eigenvalue~$\lambda(\eps, \, \kappa)$ appears
\emph{with different signs} in the equation and
in the boundary condition.

When~$0 < \kappa\leq 1$, minimisers~$v_{\eps,\kappa}$
satisfy a Robin boundary condition (with positive Robin coefficient),
which is rather natural in the context of our problem. On the other hand,
when~$\kappa > 1$ the quantity~$\lambda(\eps,\,\kappa)$
is also important, as it allows us to prove a trace inequality on
the perforated domain~$\Omega_\eps$ (see Lemma~\ref{lemma:trace} below):
for any~$\kappa > 1$, any~$\eps$ small enough
and any~$u\in H^1(\Omega_\eps)$ such that~$u=0$
on~$\partial\Omega_\eps\cap\partial\Omega$, it holds that
\begin{equation} \label{trace-intro}
 \frac{1}{\mu_\eps} \int_{\Gamma_\eps} u^2 \, \d\sigma
  \leq C\left(\frac{\eps^2}{\lambda(\eps, \, \kappa)}
   \int_{\Omega_\eps} \abs{\nabla u}^2
  + \kappa \int_{\Omega_\eps} u^2\right) \! ,
\end{equation}
where~$C$ is a constant that does not depend on~$\eps$, $\kappa$, nor~$u$.
This motivates us to make the following (and last) assumption:
\begin{enumerate}[label=(C\textsubscript{\arabic*}),
 ref=C\textsubscript{\arabic*}, resume]
 \item \label{hp:limit} \label{hp:last}
 For any~$\kappa > 0$, $\kappa\neq 1$, the limit
 \[
  \lambda_*(\kappa) := \lim_{\eps\to 0}
   \frac{\lambda(\eps, \, \kappa)}{\eps^2}
 \]
 exists.
\end{enumerate}

\begin{remark} \label{rk:subsequence}
 As it turns out (see Remark~\ref{rk:Helly} below),
 assumptions~\eqref{hp:first}--\eqref{hp:critical}
 imply that there exists a (non-relabelled) subsequence~$\eps_j\to 0$
 satisfying~\eqref{hp:limit}, and that~$\lambda_*(\kappa)$
 is finite and nonzero for all~$\kappa > 0$, $\kappa\neq 1$.
 Our arguments would still remain valid without condition~\eqref{hp:limit}
 so long as we took the limit along a  subsequence~$\eps_j\to 0$ only;
 in that case, however, the function~$\lambda_*$ and the homogenised
 problem~\eqref{hom_eq} would depend on the choice of~$(\eps_j)_{j\in\N}$.
 Instead, the main point of assumption~\eqref{hp:limit} is that we
 have the same limit for all subsequences~$\eps_j\to 0$.
\end{remark}

Our first goal is to study the properties of~$\lambda_*(\kappa)$.
We extend~$\lambda_*$ to a function~$(0, \, +\infty)\to\R$,
by setting~$\lambda_*(1) := 0$.
This choice is rather natural, given that~$\lambda(\eps, \, \kappa)$
is negative for~$0 < \kappa < 1$ (see~\eqref{lambda_eps-})
and positive for~$\kappa > 1$ (see~\eqref{lambda_eps+}).
In fact, it guarantees continuity of~$\lambda_*$ even at~$\kappa = 1$,
as stated in our next result.

\begin{figure}[t]
 \centering
 \begin{subfigure}{0.48\textwidth}
  \centering
  \includegraphics[width=\linewidth]{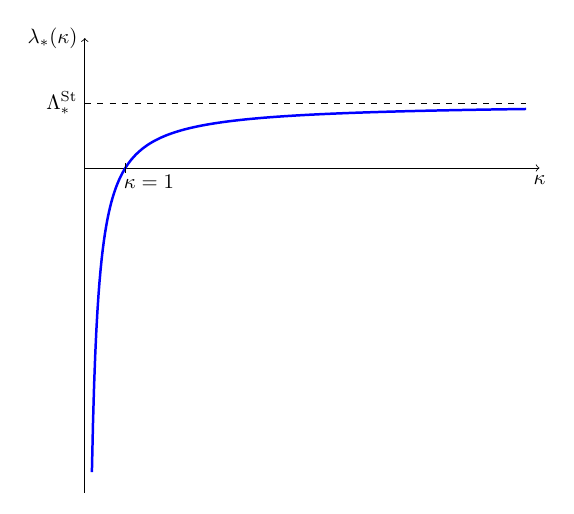}
 \end{subfigure}
 \hfill
 \begin{subfigure}{0.48\textwidth}
  \centering
  \includegraphics[width=\linewidth]{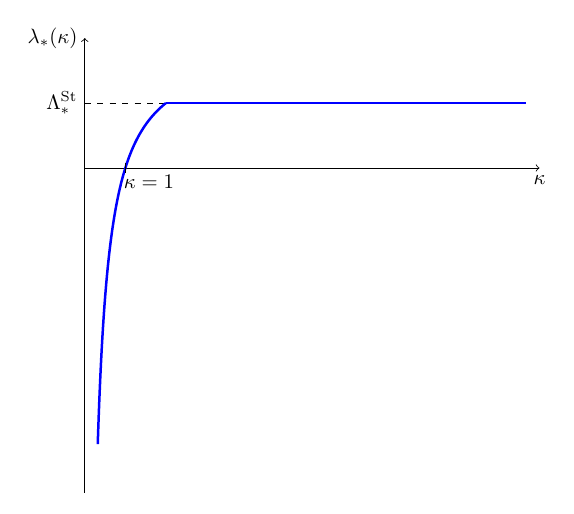}
 \end{subfigure}
 \caption{Possible behaviours for the function~$\lambda_*$,
 consistent with Theorem~\ref{th:lambda*}. The function~$\lambda_*$ may be strictly increasing or it may have a plateau corresponding to the value~$\St_*$ (although we do not have an example of the latter). Either way, $\lambda_*$ is strictly increasing in~$(0, \, 1)$ and has a vertical asymptote at~$\kappa = 0$. It is a continuous function, but we do \emph{not} know whether it is differentiable, in general.}
\end{figure}

\begin{theorem} \label{th:lambda*}
 Under assumptions~\eqref{hp:first}--\eqref{hp:last},
 the function $\lambda_*\colon (0, \, +\infty)\to\R$
 is continuous and is either strictly increasing
 everywhere, or strictly increasing in an interval of the form~$(0, \, \kappa_*)$ with~$\kappa_* > 1$, then constant in~$[\kappa_*, \, +\infty)$.
 Moreover, the limits
 \begin{equation} \label{limits0infty}
  \lambda_0 := - \lim_{\kappa\to 0} \kappa \, \lambda_*(\kappa),
  \qquad
  \lambda_\infty := \lim_{\kappa\to +\infty} \lambda_*(\kappa)
 \end{equation}
 exist, are finite and positive, and satisfy
 \[
  \lambda_\infty = \St_*, \qquad
  \St_* \leq \lambda_0 \leq \Dir_*,
 \]
 where~$\Dir_*$ and~$\St_*$ are given in~\eqref{hp:critical}.
\end{theorem}

Theorem~\ref{th:lambda*} implies, in particular,
that~$\lambda_*(\kappa) \sim -\lambda_0/\kappa \to -\infty$
as~$\kappa\to 0$, that~$\lambda_*(\kappa)\to 0$ as~$\kappa\to 1$,
and that~$\lambda_*$ is strictly increasing in~$(0, \, 1)$.
Therefore, the function~$(0, \, 1)\to (0, \, +\infty)$
given by~$\kappa\mapsto -\lambda_*(\kappa)$ is a bijection,
and for any~$\beta > 0$ there exists
a unique~$\kappa_* = \kappa_*(\beta)\in (0, \, 1)$ such that
\begin{equation} \label{kappa_*}
 -\lambda_*(\kappa_*(\beta)) = \beta.
\end{equation}
The number~$\kappa_*(\beta)$ is precisely the one that appears
in our candidate homogenised equation~\eqref{hom_eq}.
% \begin{remark*}
%  As a by-product of the proof of Theorem~\ref{th:lambda*},
%  we obtain
%  \[
%   \limsup_{\kappa\to 1^+} \frac{\lambda_*(\kappa)}{\kappa - 1}
%   \leq \St_*, \qquad
%   \limsup_{\kappa\to 1^-} \frac{\lambda_*(\kappa)}{\kappa - 1}
%   \leq \Dir_*
%  \]
%  (see Equations~\eqref{lambda20} and~\eqref{lambda30}.
%  However, it's not clear to me whether we should expect
%  the function~$\lambda_*$ to be differentiable at~$\kappa = 1$;
%  it might have a corner instead.
% \end{remark*}

Unfortunately, we do not have any example
which shows the strict inequality~$\lambda_0 < \Dir_*$,
nor a proof that~$\lambda_0 = \Dir_*$ in general.
The best we can offer is the following conditional result.

\begin{prop} \label{prop:lambda0}
 In addition to~\eqref{hp:first}--\eqref{hp:last},
 assume that there exists a constant~$C > 0$
 such that
 \begin{equation} \label{hp:lambda0_geometry}
  \sigma(\partial H_\eps
   \cap B^n_\rho(x)) \geq C \rho^{n-1}
 \end{equation}
 for all~$\eps \in (0, \, 1)$, all~$x\in\partial H_\eps$
 and all~$\rho$ with~$0 < \rho <
  \sigma(\partial H_\eps)^{\frac{1}{n-1}}$.
 Moreover, suppose that there exists a continuous
 function~$\omega\colon [0, \, +\infty) \to [0, \, +\infty)$
 that satisfies~$\omega(0) = 0$ and
 \begin{equation} \label{hp:lambda0}
  \kappa^{1/2} \abs{v_{\eps,\kappa}(x) - v_{\eps,\kappa}(y)}
  \leq \omega\!\left(\frac{\abs{x-y}}
   {\sigma(\partial H_\eps)^{\frac{1}{n-1}}}\right)
 \end{equation}
 for all~$x\in\partial H_\eps$, $y\in\partial H_\eps$
 and all~$\kappa$ and~$\eps$ small enough
 (say, all~$\kappa$ smaller than some~$\kappa_0 > 0$
 and all~$\eps$ smaller than some~$\eps_0(\kappa) > 0$).
 Then, $\lambda_0 = \Dir_*$.
\end{prop}

Condition~\eqref{hp:lambda0_geometry} is a geometric one,
and could be interpreted as a curvature bound on~$\partial H_\eps$.
(We need to restrict our attention to radii~$\rho < \sigma(\partial H_\eps)^{\frac{1}{n-1}}$ so that the right-hand side of~\eqref{hp:lambda0_geometry} does not exceed~$C\sigma(\partial H_\eps)$.
Instead, condition~\eqref{hp:lambda0} is an equicontinuity
estimate for solutions to the boundary value problem~\eqref{eigenvalue}.
Checking whether condition~\eqref{hp:lambda0} is satisfied
in a concrete example requires proving (boundary) elliptic regularity
estimates, uniformly with respect to the parameter~$\eps$,
for solutions to~\eqref{eigenvalue}.

\begin{remark} \label{rk:CM}
% Proposition~\ref{prop:DirCap} and
 Theorem~\ref{th:lambda*}
 implies that the ``strange term''~$\beta \kappa_*(\beta)$
 on the left-hand side of Equation~\eqref{hom_eq} tends
 to zero as~$\beta\to 0$ (i.e., when the Robin boundary
 conditions approach the Neumann case)
 and to~$\lambda_0$ as~$\beta\to +\infty$
 (i.e., when the boundary conditions approach
 the Dirichlet case).
 In particular, when the equality~$\lambda_0 = \Dir_* = \Cap_*$ holds
 (for instance, in our model example~\eqref{model_example})
 our strange term converges to the capacitary one
 (see Equation~\eqref{hom_eq_Dir}), in the limit of large~$\beta$.
\end{remark}

\paragraph*{The homogenisation problem.}

We define the set~$\mcP_\eps\subseteq\R^n$ as the union of
rescaled copies of~$H_{\eps}$ periodically
distributed on the cubic grid of size~$\eps$:
\begin{equation} \label{P_eps}
 \mcP_\eps := \eps \Z^n + \eps H_{\eps}
 = \left\{\eps z + \eps x\colon z\in\Z^n, \ x\in\mcP_\eps \right\} \! .
\end{equation}
Given a bounded, Lipschitz domain~$\Omega\subseteq\R^n$,
we set~$\Omega_\eps := \Omega\setminus\mcP_\eps$,
$\Gamma_\eps := \partial\mcP_\eps\cap\Omega$.
We also take a family $(\mu_\eps)_{0 < \eps < 1}$
of positive numbers that satisfies~\eqref{mu_eps}.
%\begin{equation}
% \lim_{\eps\to 0} \frac{\abs{\Omega} \mu_\eps}
%  {\sigma(\Gamma_\eps)} = 1.
%\end{equation}

\begin{remark} \label{rk:mu_eps_bar}
 For instance, the family defined by
 \begin{equation} \label{mu_eps_bar}
  \overline{\mu}_\eps
   := \frac{\sigma(\partial H_{\eps})}{\eps},
   \qquad \eps \in (0, \, 1)
 \end{equation}
 satisfies~\eqref{mu_eps}.
 Indeed, the number~$N_\eps$ of points~$z\in\Z^n$ such that the
 cube~$\eps z + \eps Y$ intersects~$\Omega$, as well as
 the number~$M_\eps$ of points~$z\in\Z^n$ such that the
 cube~$\eps z + \eps Y$ is contained in~$\Omega$,
 are both of order~$\eps^{-n}$; %by periodicity;
 more precisely, we have
 $\lim_{\eps\to 0} \eps^n N_\eps
  = \lim_{\eps\to 0} \eps^n M_\eps
  = \abs{\Omega}$.
 As a consequence, we have
 \[
  \left(\abs{\Omega} + \mathrm{o}_{\eps\to 0}(1)\right)
   \overline{\mu}_\eps
  \leq N_\eps \eps^{n-1} \sigma(\partial H_\eps)
  \leq \sigma(\Gamma_\eps)
  \leq M_\eps \eps^{n-1} \sigma(\partial H_\eps)
  \leq \left(\abs{\Omega} + \mathrm{o}_{\eps\to 0}(1)\right)
   \overline{\mu}_\eps,
 \]
 and it follows that~$\overline{\mu}_\eps$ satisfies~\eqref{mu_eps}.
\end{remark}

Now all the terms of the $\eps$-problem~\eqref{main_eq}
and of the candidate limit problem~\eqref{hom_eq}
are rigorously defined. We next state the main convergence result.

\begin{theorem} \label{th:strange_term}
 Let~$(\mu_\eps)_{0 < \eps < 1}$ satisfy~\eqref{mu_eps}.
 Let~$\alpha \geq 0$, $\beta > 0$,
 $f\in L^2(\Omega)$, and $g\in C^0(\overline{\Omega})$.
 Under assumptions~\eqref{hp:first}--\eqref{hp:last},
 Problems~\eqref{main_eq}, \eqref{hom_eq} admit unique
 solutions~$u_\eps$, $u_0$ respectively, which satisfy
 \[
  \norm{u_\eps - u_0}_{L^2(\Omega_\eps)}
  \to 0 \qquad \textrm{as } \eps\to 0.
 \]
\end{theorem}

For Problem~\eqref{main_eq} with homogeneous boundary conditions
(that is, when~$g=0$),
we also have the following estimate for the convergence rate,
along the lines of \cite[Th\'eor\`eme~1.1]{KacimiMurat}.
Given~$\beta > 0$, we take~$\kappa = \kappa_*(\beta)$ and
consider a minimiser~$v_{\eps, \kappa}$ for the right-hand
side of~\eqref{lambda_eps-}, identified with
a~$\Z^n$-periodic function on~$\R^n\setminus(\mcP_\eps/\eps)$.
We define a corrector, $w_\eps\colon\R^n\setminus\mcP_\eps\to\R$, as
\[
 w_{\eps}(x) :=
  \frac{1}{M_{\eps,\kappa}}
  v_{\eps,\kappa}\!\left(\frac{x}{\eps}\right)
  \qquad \textrm{where } x\in\R^n\setminus\mcP_\eps,
  \quad \ M_{\eps,\kappa} := \int_{\T^n\setminus H_\eps}
  v_{\eps,\kappa}(y) \, \d y.
\]
We will see later on (see Lemma~\ref{lemma:corrector}
in Section~\ref{sect:strange_term}) that
$\norm{w_\eps - 1}_{L^2(\Omega_\eps)} \to 0$ as~$\eps\to 0$.
We also define
\[
 \eta_\eps := \abs{\frac{\lambda(\eps, \, \kappa_*(\beta))}
   {\eps^2} + \beta} + \abs{\frac{\mu_\eps}{\overline{\mu}_\eps} - 1} \! ,
\]
where~$\overline{\mu}_\eps := \eps^{-1} \sigma(\partial H_\eps)$.
Assumption~\eqref{hp:limit} and the
definition of~$\kappa_*(\beta)$ (see Equation~\eqref{kappa_*}) imply
that $\eps^{-2}\lambda(\eps, \, \kappa_*(\beta))\to
\lambda_*(\kappa_*(\beta)) = -\beta$ as~$\eps\to 0$,
while assumption~\eqref{mu_eps} and Remark~\ref{rk:mu_eps_bar}
imply that~$\mu_\eps/\overline{\mu}_\eps\to 1$
as~$\eps\to 0$. Therefore, we have~$\eta_\eps\to 0$ as~$\eps\to 0$.

\begin{prop} \label{prop:abstractrate}
 Suppose that~$\Omega\subseteq\R^n$ is a bounded domain
 of class~$C^2$, $g=0$, and~$(\mu_\eps)_{0 < \eps < 1}$
 satisfies~\eqref{mu_eps}.
 Under assumptions~\eqref{hp:first}--\eqref{hp:last},
 there exists a constant~$C$ such that for all~$f\in L^\infty(\Omega)$
 and all~$\eps \in (0, \, 1)$, the solutions~$u_\eps$, $u_0$
 of problems~\eqref{main_eq}, \eqref{hom_eq} satisfy
 \begin{equation} \label{abstractrate}
  \norm{u_\eps - w_\eps u_0}_{H^1(\Omega_\eps)}
  \leq C (\eps + \eta_\eps + \sigma(\Gamma_\eps))
   \norm{f}_{L^\infty(\Omega)}
 \end{equation}
\end{prop}

In Proposition~\ref{prop:abstractrate}, we assume
boundedness of the source term (as in~\cite{KacimiMurat})
and higher regularity for the domain~$\Omega$
(i.e.~$C^2$ instead of Lipschitz continuity)
because the proof requires higher regularity on~$u_0$.
So far, we do not have an estimate for general~$f\in L^2(\Omega)$.
For the estimate~\eqref{abstractrate} to be interpreted
as a convergence rate, we need to know
that~$\sigma(\Gamma_\eps)\to 0$ as~$\eps\to 0$.
While this is the case for the model example~\eqref{model_example},
this may not be true in the general setting.
In Section~\ref{sect:many_holes}, we provide
a family of examples for which
conditions~\eqref{hp:first}--\eqref{hp:last} are satisfied,
yet~$\sigma(\Gamma_\eps)$ can be made arbitrarily large,
possibly even~$\sigma(\Gamma_\eps)\to +\infty$ as~$\eps\to 0$.

\subsection{The model example fits into the abstract setting}
\label{model_example_fit_sec}

Assume that~$n\geq 3$. Let~$D\subseteq\R^{n}$ be a bounded
open set of class~$C^2$ such that~$\R^n\setminus\overline{D}$
is connected. We define
\begin{equation} \label{Hdelta}
 H_\eps := \eps^{\gamma-1} \overline{D}
 = \left\{\eps^{\gamma-1} y\colon y\in\overline{D}\right\} \!,
\end{equation}
where
\begin{equation} \label{deltaeps}
 \gamma := \frac{n}{n - 2} ,
\end{equation}
for all~$\eps > 0$ small enough,
say~$\eps \in (\eps, \, \eps_0]$ for some~$\eps_0 > 0$,
so that~$H_\eps \subseteq Y$ for~$\eps \in (0, \, \eps_0]$.
We define~$H_\eps := H_{\eps_0}$ for~$\eps\in (\eps_0, \, 1)$.
We denote by~$\dot{H}^1(\R^{n}\setminus\overline{D})$
the set of functions~$z_0\in L^2_{\mathrm{loc}}(\R^n\setminus\overline{D})$
that satisfy~$\nabla z_0\in L^2(\R^n\setminus\overline{D}, \, \R^n)$
and ``decay at infinity'' in a suitable sense,
i.e.~the set~$\{x\in\R^n\setminus\overline{D}\colon
\abs{z_0(x)}\geq c\}$ has finite measure
for all~$c > 0$. The set~$\dot{H}^1(\R^{n}\setminus\overline{D})$
of ``finite-energy'' functions is modelled after
the space~$D^1(\R^n)$ defined by Lieb and Loss~\cite{LiebLoss}.
We also consider the direct sum
\[
 V := \dot{H}^1(\R^{n}\setminus\overline{D})\oplus\R.
\]
Elements of~$V$ can be canonically identified with
functions of the form~$z = z_0 + L$,
where~$z_0\in\dot{H}^1(\R^{n}\setminus\overline{D})$
and~$L\in\R$ is a constant. For such a function~$z$,
we write~$z(\infty) := L$.

\begin{prop} \label{prop:example}
%  Let~$D\subseteq\R^n$ be a bounded open set
%  of class~$C^2$, such that~$\R^n\setminus\overline{D}$ is connected, with~$n\geq 3$.
 The set~$H_\eps$ defined by~\eqref{Hdelta}--\eqref{deltaeps}
 satisfies conditions \eqref{hp:first}--\eqref{hp:last}, with
 \[
  \lambda_*(\kappa)
  = \min\left\{\int_{\R^{n}\setminus \overline{D}}
   \abs{\nabla z}^2\colon z\in V, \
   \fint_{\partial D} z^2 \,\d\sigma
   - \kappa \, z(\infty)^2 = 1 \right\}
 \]
 if~$\kappa > 1$ and
 \[
  \lambda_*(\kappa)
  = -\min\left\{\int_{\R^{n}\setminus \overline{D}}
   \abs{\nabla z}^2\colon z\in V, \
   \fint_{\partial D} z^2 \,\d\sigma
   - \kappa \, z(\infty)^2 = -1 \right\}
 \]
 if~$0 < \kappa < 1$. Moreover, we have
 \begin{align*}
  \Dir_* &= \min\left\{\int_{\R^{n}\setminus \overline{D}}
   \abs{\nabla z}^2\colon
   z\in V, \ z = 0 \textrm{ on } \partial D, \
   z(\infty) = 1 \right\} \, , \\
  \St_* &= \min\left\{\int_{\R^{n}\setminus H}
   \abs{\nabla z}^2\colon z\in V, \
   \fint_{\partial D} z^2\, \d\sigma = 1,
   \ z(\infty)= 0 \right\} \! .
 \end{align*}
 There exist minimisers for the
 minimisation problems on the right-hand sides
 of all these equations. Moreover, we have~$\lambda_0 = \Dir_*$,
 where~$\lambda_0$ is defined in~\eqref{limits0infty}.
\end{prop}

In the case when the holes are spherical, we can evaluate explicitly
the ``strange term'' in the limit problem.

\begin{prop} \label{prop:ball}
 When~$D$ is the unit ball, $D = B^n\subseteq\R^n$,
 we have
 \[
  \lambda_*(\kappa) = \sigma_n (n-2) \, \frac{\kappa - 1}{\kappa}
 \]
 for all~$\kappa > 0$, where~$\sigma_n := \sigma(\partial B^n)$.
 As a consequence, the function~$\kappa_*(\cdot)$ defined
 by~\eqref{kappa_*} is given by~\eqref{kappa_*-spheres}.
\end{prop}

The following sections are devoted to proving the results
we stated so far.

\section{Properties of~$\lambda_*(\kappa)$}
\label{sect:eigen}

\subsection{Analysis of Problem~\eqref{eigenvalue}
for fixed~$\eps > 0$}
\label{sect:eigenvalue_eps}

We collect some properties of the
minimisers for Problems~\eqref{lambda_eps+}, \eqref{lambda_eps-}
for fixed values of~$\eps \in (0, \, 1)$ and~$\kappa$.
The results contained in this section
% depend on assumptions~\eqref{hp:first}--\eqref{hp:extension}
do \emph{not} require assumptions~\eqref{hp:critical},
\eqref{hp:limit} unless otherwise stated.

\begin{lemma} \label{lemma:exists}
 Assume that either~$\kappa\abs{\T^n\setminus H_\eps} > 1$
 or~$0 < \kappa \leq 1$.
 Then, the quantity~$\lambda(\eps, \, \kappa)$
 defined respectively by~\eqref{lambda_eps+}
 or~\eqref{lambda_eps-} is strictly positive and finite.
 In both cases, there exists a minimiser~$v_{\eps,\kappa}$
 which is unique up to the sign and does not attain zero value in~$\T^n\setminus H_\eps$.
 The minimiser is a solution to Problem~\eqref{eigenvalue} ---
 in fact, it is the unique solution, up to multiplication by a constant.
\end{lemma}
\begin{proof}%[Proof of Lemma~\ref{lemma:exists}]
 We focus on the case~$\kappa\abs{\T^n\setminus H_\eps} > 1$;
 the other case is similar and we omit the details.
 First, we show that~$\lambda(\eps, \, \kappa) < +\infty$,
 or equivalently, there exists $v\in H^1(\T^n\setminus H_\eps)$
 such that~$\mcG_{\eps,\kappa}(v)=1$.
 Let~$\varphi\in C^\infty_{\mathrm{c}}(Y)$ be a cut-off function,
 such that $\varphi = 1$ in a neighbourhood of~$H_\eps$
 and~$0 \leq \varphi \leq 1$. By taking
 the support of~$\varphi$ small enough, we can make sure that
 $\mcG_{\eps,\kappa}(\varphi) > 0$.
 Then, $v := \varphi/(\mcG_{\eps,\kappa}(\varphi))^{1/2}$
 satisfies~$\mcG_{\eps,\kappa}(v)=1$.

 Next, we prove that a minimiser exists.
 Let~$v_j\in H^1(\T^n\setminus H_\eps)$ be a minimising
 sequence for~\eqref{lambda_eps+}. We write
 \[
  v_j = w_j + M_j,
 \]
 where~$w_j\in H^1(\T^n\setminus H_\eps)$ is such that
 $\int_{\T^n\setminus H_\eps} w_j \, \d x = 0$ and~$M_j\in\R$.
%  The gradients~$\nabla w_j = \nabla v_j$ are
%  bounded in~$L^2(\T^n\setminus H_\eps)$, so
 The constraint~$\mcG_{\eps,\kappa}(v_j) = 1$
 can be written as
 \begin{equation} \label{exists1}
  \left(1 - \kappa\abs{\T^n\setminus H_\eps} \right) M_j^2
   + 2M_j \fint_{\partial H_\eps} w_j \, \d\sigma
   + \fint_{\partial H_\eps} w_j^2 \, \d\sigma
   - \kappa \int_{\T^n\setminus H_\eps} w_j^2 \, \d x = 1.
 \end{equation}
%  As a function of~$M_j$, the left-hand side of~\eqref{exists1}
%  is a quadratic polynomial whose coefficients are bounded
%  in terms of~$\eps$ and~$\kappa$ only, because
%  $w_j$ is bounded in~$H^1(\T^n\setminus H_\eps)$.
%  Moreover, the leading coefficient is nonzero, by assumption.
 Equation~\eqref{exists1} can be solved explicitly
 for~$M_j$, which provides a bound for~$M_j$
 in terms of~$w_j$, $\eps$ and~$\kappa$ only.
 Since~$w_j$ is bounded in~$H^1(\T^n\setminus H_\eps)$
 by the Poincar\'e inequality, it follows
 that~$M_j$ is bounded and hence, $v_j$ is bounded
 in~$H^1(\T^n\setminus H_\eps)$.
 Let~$v_{\eps,\kappa}$ be the weak limit of
 any subsequence of~$v_j$. Then, the compact Sobolev embeddings
 $H^1(\T^n\setminus H_\eps)\hookrightarrow L^2(\T^n\setminus H_\eps)$
 and~$H^{1/2}(\partial H_\eps)\hookrightarrow L^2(\partial H_\eps)$
 imply that~$v_{\eps,\kappa}$ satisfies
 $\mcG_{\eps,\kappa}(v_{\eps,\kappa})=1$ and, hence,
 is a minimiser for~\eqref{lambda_eps+}.
 The function~$v_{\eps,\kappa}$ cannot be constant,
 for otherwise we would have
 \[
  \mcG_{\eps,\kappa}(v_{\eps,\kappa})
  = (1 - \kappa\abs{\T^n\setminus H_\eps})v_{\eps,\kappa}^2 \leq 0,
 \]
 which is a contradiction.
 Therefore, $\lambda(\eps, \, \kappa) > 0$.

 By the Lagrange multiplier theorem, any minimiser~$v_{\eps,\kappa}$
 of~\eqref{lambda_eps+} must satisfy
 \begin{equation} \label{Lagrangemult}
  \int_{\T^n\setminus H_\eps} \nabla v_{\eps,\kappa}\cdot\nabla\varphi \, \d x
  + \mu \fint_{\partial H_\eps} v_{\eps,\kappa} \, \varphi \, \d\sigma
  - \kappa \mu \int_{\T^n\setminus H_\eps} v_{\eps,\kappa}\, \varphi \, \d x = 0
 \end{equation}
 for some Lagrange multiplier~$\mu\in\R$
 and any~$\varphi\in H^1(\T^n\setminus H_\eps)$.
 Taking~$\varphi = v_{\eps,\kappa}$ in~\eqref{Lagrangemult},
 we obtain
 \[
  \lambda(\eps, \, \kappa)
  = \int_{\T^n\setminus H_\eps} \abs{\nabla v_{\eps,\kappa}}^2 \, \d x
%   = - \mu \left(\fint_{\partial H_\eps} v_{\eps,\kappa}^2 \, \d\sigma
%   - \kappa \int_{\T^n\setminus H_\eps} v_{\eps,\kappa}^2 \, \d x\right)
  = - \mu \, \mathscr{G}(v_{\eps,\kappa})
  = -\mu,
 \]
 so~$v_{\eps,\kappa}$ satisfies~\eqref{eigenvalue}.
 It only remains to show that the solution of~\eqref{eigenvalue}
 is unique up to multiplication by a constant and has constant
 sign inside~$\T^n\setminus H_\eps$.
%  This is a consequence of the strong maximum principle,
%  and can be proved by classical arguments.
 Let~$v\neq v_{\eps,\kappa}$ be another solution of~\eqref{eigenvalue}.
 Testing the equation for~$v$ with~$\varphi = v$ itself,
%  similarly to~\eqref{Lagrangemult},
 we deduce that
 \[
  \int_{\T^n\setminus H_\eps}\abs{\nabla v}^2 \d x
   = \lambda(\eps,\kappa) \, \mcG_{\eps,\kappa}(v),
 \]
 so that~$\mcG_{\eps,\kappa}(v) > 0$
 and~$\tilde{v} := v/(\mcG_{\eps,\kappa}(v))^{1/2}$
 is another minimiser of~\eqref{lambda_eps+}. Then, $\abs{\tilde{v}}$,
 too, is a minimiser for~\eqref{lambda_eps+}
 and hence, a solution to~\eqref{eigenvalue}.
 By applying the strong maximum principle to~$\abs{\tilde{v}}$,
 we conclude that~$v$ is either zero everywhere
 or non-zero anywhere in~$Y\setminus H_\eps$.
 Now, if~$x_0$ is an arbitrary point in the
 interior of~$\T^n\setminus H_\eps$,
 then~$w := v(x_0) v_{\eps,\kappa} - v_{\eps,\kappa}(x_0) v$
 is a solution of~\eqref{eigenvalue}
 that takes zero value at~$x_0$. By the argument above,
 we conclude that~$w = 0$ identically,
 i.e. $v$ and~$v_{\eps,\kappa}$ are linearly dependent.
\end{proof}

\begin{remark}
 If~$H_\eps$ satisfies the outer ball condition,
 then~$v_{\eps,\kappa}$ is nowhere equal to zero
 in the closure of~$\T^n\setminus H_\varepsilon$, by Hopf's lemma.
\end{remark}

\begin{remark} \label{rk:bdav}
 By testing Problem~\eqref{eigenvalue}
 against the constant function~$1$ (i.e.,
 choosing~$\varphi=1$ in~\eqref{Lagrangemult}),
 we obtain the equality
 \begin{equation} \label{bdav}
  \fint_{\partial H_\eps} v_{\eps,\kappa} \,\d\sigma
   = \kappa\int_{\T^n\setminus H_\eps} v_{\eps,\kappa} \, \d x,
 \end{equation}
 for all~$\kappa$, $\eps$
 such that~$\kappa\abs{\T^n\setminus H_\eps} > 1$
 or~$0 < \kappa \leq 1$.
\end{remark}

% \begin{remark}
%  By testing Problem~\eqref{eigenvalue} against~$1/v_{\eps,\kappa}$
%  (which is an admissible test function, because~$v_{\eps,\kappa}$
%  is continuous and nonzero in the closure of~$\T^n\setminus H_\eps$),
%  we obtain
%  \[
%    \int_{\T^n\setminus H_\eps}
%    \frac{\abs{\nabla v_{\eps,\kappa}}^2}{v_{\eps,\kappa}^2}
%    = \lambda(\eps, \, \kappa)
%    \left(\kappa\abs{\T^n\setminus H_\eps} - 1\right)
%  \]
%  as soon as~$\kappa\abs{\T^n\setminus H_\eps} > 1$
%  or~$0 < \kappa \leq 1$.
% \end{remark}

The following lemma provides an equivalent characterisation
of~$\lambda(\eps, \, \kappa)$.

\begin{lemma} \label{lemma:singletrace}
 If~$\kappa\abs{\T^n\setminus H_\eps} > 1$, then
 $1/\lambda(\eps, \, \kappa)$ is the minimum among
 all numbers~$\mu>0$ that satisfy the inequality
 \begin{equation} \label{singletrace+}
  \fint_{\partial H_\eps} v^2 \, \d\sigma
  \leq \mu \int_{\T^n\setminus H_\eps} \abs{\nabla v}^2 \, \d x
  + \kappa \int_{\T^n\setminus H_\eps} v^2 \, \d x
 \end{equation}
 for all (periodic) functions~$v\in H^1(\T^n\setminus H_\eps)$.
 In a similar way, when~$0 < \kappa \leq 1$
 the value~$-1/\lambda(\eps, \, \kappa)$ is the
 minimum among all~$\mu>0$ with which the inequality
 \begin{equation} \label{singletrace-}
  \kappa \int_{Y\setminus H_\eps} v^2 \, \d x
  \leq \mu \int_{Y\setminus H_\eps} \abs{\nabla v}^2 \, \d x
  + \fint_{\partial H_\eps} v^2 \, \d\sigma
 \end{equation}
 holds for all~$v\in H^1(\T^n\setminus H_\eps)$.
\end{lemma}
\begin{proof}
 Again, we focus on the case~$\kappa\abs{\T^n\setminus H_\eps} > 1$,
 as the other one is analogous.
 We claim that~$\mu = 1/\lambda(\eps, \, \kappa)$
 satisfies~\eqref{singletrace+} for all~$v\in H^1(\T^n\setminus H_\eps)$.
 Indeed, let~$v\in H^1(\T^n\setminus H_\eps)$. If~$\mcG_{\eps,\kappa}(v) \leq 0$,
 then~$v$ satisfies~\eqref{singletrace+} for all~$\mu\geq 0$.
 If~$\mcG_{\eps,\kappa}(v) >0$, then
 $w := v/(\mcG_{\eps,\kappa}(v))^{1/2}$
 satisfies~$\mcG_{\eps,\kappa}(w)=1$ and hence
 \[
  \frac{1}{\mcG_{\eps,\kappa}(v)}
   \int_{\T^n\setminus H_\eps} \abs{\nabla v}^2 \, \d x
  = \int_{\T^n\setminus H_\eps} \abs{\nabla w}^2 \, \d x
  \geq \lambda(\eps, \, \kappa) .
 \]
 In both cases, the inequality~\eqref{singletrace+} follows,
 with~$\mu = 1/\lambda(\eps, \, \kappa)$.
 Now, take~$\mu > 0$ to be any other constant with which
 \eqref{singletrace+} holds for all~$v\in H^1(\T^n\setminus H_\eps)$.
 Let~$v_{\eps,\kappa}$ be a minimiser for~\eqref{lambda_eps+}.
 Then, one has
 \[
  \mu \, \lambda(\eps, \, \kappa)
  = \mu \int_{\T^n\setminus H_\eps} \abs{\nabla v_{\eps,\kappa}}^2
  \stackrel{\eqref{singletrace+}}{\geq}
   \mcG_{\eps,\kappa}(v_{\eps,\kappa}) = 1,
 \]
 and hence~$\mu \geq 1/\lambda(\eps, \, \kappa)$.
\end{proof}

\begin{corollary} \label{cor:lambdamonotone}
 Assume that condition~\eqref{hp:limit} holds,
 in addition to~\eqref{hp:first}--\eqref{hp:extension}.
 Then, $\kappa\mapsto\lambda_*(\kappa)$
 is a nondecreasing function of~$\kappa \in (0, \, +\infty)$,
 while~$\kappa\mapsto -\kappa\,\lambda_*(\kappa)$
 is a nonincreasing function of~$\kappa\in (0, \, 1)$.
\end{corollary}
\begin{proof}
 Lemma~\ref{lemma:singletrace} implies
 that, for any~$\eps \in (0, \, 1)$, the function
 $\kappa\mapsto 1/\lambda(\eps, \, \kappa)$
 is nonincreasing for~$\kappa > \abs{\T^n\setminus H_\eps}^{-1}$,
 while $\kappa\mapsto -1/\lambda(\eps, \, \kappa)$
 is nondecreasing for~$0 < \kappa\leq 1$. As a consequence,
 \begin{equation} \label{lambdamonotone}
  \kappa\mapsto \lambda(\eps, \, \kappa)
  \textrm{ is a nondecreasing function of } \kappa \in (0, \, 1]
   \cup \left(\abs{\T^n\setminus H_\eps}^{-1}, \, +\infty\right)
 \end{equation}
 and, if condition~\eqref{hp:limit}
 is satisfied, the function $\kappa\mapsto\lambda_*(\kappa)$
 is also nondecreasing.
 Now, the monotonicity of~$\kappa\in (0, \, 1)\mapsto\kappa\lambda_*(\kappa)$
 does not follow immediately from that of~$\kappa\mapsto\lambda_*(\kappa)$,
 because~$\lambda_*(\kappa) < 0$ for~$\kappa < 1$.
 However, Lemma~\ref{lemma:singletrace} implies that, when~$0 < \kappa \leq 1$,
 the value~$-1/(\kappa\,\lambda(\eps, \, \kappa))$
 is the minimum of all~$\mu > 0$ satisfying the inequality
 \[
  \int_{Y\setminus H_\eps} v^2 \, \d x
  \leq \mu \int_{Y\setminus H_\eps} \abs{\nabla v}^2  \, \d x
  + \frac{1}{\kappa}\fint_{\partial H_\eps} v^2 \, \d\sigma
 \]
 for all~$v\in H^1(\T^n\setminus H_\eps)$.
 Therefore,
 \begin{equation*}
  \kappa\mapsto - \frac{1}{\kappa\,\lambda(\eps, \, \kappa)}
  \textrm{ is a nondecreasing function of }
  \kappa\in (0, \, 1)
 \end{equation*}
 and so is~$\kappa\mapsto\kappa\,\lambda_*(\kappa)$
 so long as~\eqref{hp:limit} is satisfied.
\end{proof}

Before we proceed further with the analysis of Problem~\eqref{eigenvalue},
we recall that the existence of a uniformly bounded extension operator
provided by assumption~\eqref{hp:extension} implies a uniform
Poincar\'e inequality for the sets~$Y\setminus H_\eps$.
This is a classical result, but we include its proof
for reader's convenience.

\begin{lemma} \label{lemma:Poincare}
 If assumptions~\eqref{hp:first}--\eqref{hp:extension}
 are satisfied, then there exists~$C_P>0$
 such that, for all~$\eps \in (0, \, 1)$ and
 all~$u\in H^1(Y\setminus H_\eps)$, there holds
 \begin{equation} \label{Poincare-average}
  \int_{Y\setminus H_\eps}
   \abs{u(x) - \fint_{Y\setminus H_\eps} u(y) \, \d y}^2 \, \d x
  \leq C_P \int_{Y\setminus H_\eps} \abs{\nabla u}^2 \, \d x.
 \end{equation}
 Moreover, there exists~$C_0>0$
 such that, for all~$\eps \in (0, \, 1)$
 and all~$u\in H^1(Y\setminus H_\eps)$
 with $u = 0$ on~$\partial Y$
 (in the sense of traces), there holds
 \begin{equation} \label{Poincare-0}
  \int_{Y\setminus H_\eps} u^2 \, \d x \leq C_0
  \int_{Y\setminus H_\eps} \abs{\nabla u}^2 \, \d x.
 \end{equation}
\end{lemma}
\begin{proof}
 Let~$E_\eps\colon H^1(Y\setminus H_\eps)\to H^1(Y)$ be the
 linear extension operator given by assumption~\eqref{hp:extension},
 and let~$u\in H^1(Y\setminus H_\eps)$. Invoking the Poincar\'e
 inequality in~$Y$ and then assumption~\eqref{hp:extension}, we have
 \begin{equation} \label{Poinc1}
  \norm{E_\eps u - \int_{Y} (E_\eps u)(x) \, \d x}_{L^2(Y)}
  \leq \norm{\nabla (E_\eps u)}_{L^2(Y)}
  \leq C\norm{\nabla u}_{L^2(Y\setminus H_\eps)} \! .
 \end{equation}
 Furthermore,
 \begin{equation} \label{Poinc2}
  \begin{split}
   \abs{\fint_{Y\setminus H_\eps} u(y) \, \d y
    - \int_{Y} (E_\eps u)(x) \, \d x}
   &\leq \fint_{Y\setminus H_\eps} \abs{u(y)
    - \int_{Y} (E_\eps u)(x) \, \d x} \d y \\
   &\leq \frac{1}{\abs{Y\setminus H_\eps}}
    \int_{Y} \abs{(E_\eps u)(y)
    - \int_{Y} (E_\eps u)(x) \, \d x} \d y \\
   &\leq  C\norm{\nabla u}_{L^2(Y\setminus H_\eps)} \! .
  \end{split}
 \end{equation}
 The last inequality follows from the H\"older inequality
 and~\eqref{Poinc1}. Combining~\eqref{Poinc1}
 with~\eqref{Poinc2}, we obtain~\eqref{Poincare-average}.

 Finally, if~$u$ has zero trace on~$\partial Y$,
 the inequality~\eqref{Poincare-0} follows immediately by
 applying the Poincar\'e inequality to~$E_\eps u\in H^1_0(Y)$.
\end{proof}

\begin{remark} \label{rk:Poincarep}
 Let~$n\geq 3$. Under the same
 assumptions~\eqref{hp:first}--\eqref{hp:extension},
 there exists a constant~$C > 0$ such that, for all~$\eps \in (0, \, 1)$
 and all~$u\in H^1(Y\setminus H_\eps)$, there holds
 \begin{equation*}
  \norm{u - \fint_{Y\setminus H_\eps} u(y) \, \d y}_{L^p(Y\setminus H_\eps)} \leq C \norm{\nabla u}_{L^2(Y\setminus H_\eps)} \! ,
 \end{equation*}
 where~$p := 2n/(n - 2)$.
\end{remark}

We conclude this section by stating and proving some inequalities
between the values~$\Dir(\eps)$, $\St(\eps)$,
$\Cap(\eps)$, $\lambda(\eps, \, \kappa)$.
%Given a value~$t\in\R$, we define
%its positive part as~$t_+ := \max(t, \, 0)$.
% In the inequalities~\eqref{Cap-Dir}--\eqref{lambda_bound+} below,
% the right-hand side is defined to be equal to~$+\infty$
% if the denominator is zero.

\begin{lemma} \label{lemma:inequalities}
 There exist positive constants~$\eps_0$, $\tau$, $C$
 such that following inequalities hold for all~$\eps \in (0, \, \eps_0]$:
 \begin{align}
  \St(\eps)  &\leq \Cap(\eps), \label{St-Cap} \\[2mm]
  \Cap(\eps) &\leq
   \frac{\Dir(\eps)}{\left(1 - C \Dir(\eps)\right)^2}
   \hspace{1.6cm} \textrm{if } \Dir(\eps) \leq \tau,
   \label{Cap-Dir} \\[2mm]
  \Dir(\eps) &\leq
   \frac{\Cap(\eps)}{\left(1 - \Cap(\eps)\right)^2}
   \hspace{1.8cm} \textrm{if } \Cap(\eps) \leq \tau,
   \label{Dir-Cap} \\[2mm]
  \lambda(\eps, \, \kappa) &\leq
   \frac{\St(\eps)}{\left(1 - C\kappa\St(\eps)\right)^2}
   \hspace{1.4cm} \textrm{if }
   \kappa\St(\eps) \leq \tau \ \textrm{ and } \
   \kappa\abs{\T^n\setminus H_\eps}>1,
   \label{lambda_bound+0} \\[2mm]
 -\kappa \, \lambda(\eps, \, \kappa) &\leq
  \Dir(\eps) \hspace{3cm} \textrm{if } 0 < \kappa < 1.
  \label{lambda_bound-0}
 \end{align}
\end{lemma}
\begin{proof}
 We prove each inequality separately.

 \medskip
 \noindent
 \textrm{Proof of~\eqref{St-Cap}.}
 This inequality follows immediately
 from the definitions of~$\St(\eps)$ and~$\Cap(\eps)$
 because, if~$\zeta$ is an admissible competitor
 for the minimisation problem~\eqref{Cap_eps},
 then~$1 - \zeta$ is admissible for~\eqref{St_eps}.

 \medskip
 \noindent
 \textit{Proof of~\eqref{Cap-Dir}.}
 Over the course of the proof we will assume with no loss of generality
 that~$\Dir(\eps)$ is ``small'' --- that is,
 $\Dir(\eps) \leq \tau$ for some $\eps$-independent constant~$\tau$.
 Let~$\varphi_\eps\in H^1(\T^n)$ be
 the minimiser for Problem~\eqref{Dir_eps}, and let
 \[
  M_\eps := \int_{\T^n} \varphi_\eps \, \d x.
 \]
 Since~$\int_{\T^n} \varphi_\eps^2 \, \d x = 1$ by definition,
 the Poincar\'e inequality in~$\T^n$ ---
 combined with assumption~\eqref{hp:critical} --- implies
 \begin{equation} \label{CapDir1}
  1 - M_\eps^2
  = \int_{\T^n} \left(\varphi_\eps - M_\eps\right)^2 \d x
  \leq C\Dir(\eps).
 \end{equation}
%  This shows that~$M_\eps \to 1$ as~$\eps\to 0$.
 The function~$\varphi_\eps - M_\eps$ satisfies
 the elliptic equation
 \begin{equation} \label{CapDir2}
  -\Delta(\varphi_\eps - M_\eps) = \Dir(\eps) \varphi_\eps
  \qquad \textrm{in } \T^n\setminus H_\eps,
 \end{equation}
 because of~\eqref{Dir_eps_eigen}.
 By assumption~\eqref{hp:decreas}, there exists
 a compact set~$K_0\subseteq Y$ that contains~$H_\eps$
 for all~$\eps\in (0, \, 1)$. Let~$U_0 := Y\setminus K_0$.
 By applying (interior) elliptic regularity estimates
 and a bootstrap argument to the equation~\eqref{CapDir2},
 we deduce
 \begin{equation*}
  \begin{split}
   \norm{\varphi_\eps - M_\eps}_{L^\infty(U_0)}
   &\leq C \norm{\varphi_\eps - M_\eps}_{L^2(\T^n\setminus H_\eps)}
    + C \Dir(\eps) \norm{\varphi_\eps}_{L^2(\T^n\setminus H_\eps)}
    \\
   &\leq C \norm{\varphi_\eps - M_\eps}_{L^2(\T^n)}
    + C \Dir(\eps)
  \end{split}
 \end{equation*}
 for some~$\eps$-independent constant~$C$.
 This implies, again by the Poincar\'e inequality in~$\T^n$
 and assumption~\eqref{hp:critical},
 \begin{equation*}
  \norm{\varphi_\eps - M_\eps}_{L^\infty(U_0)}
  \leq C \Dir(\eps)^{1/2}.
 \end{equation*}
 This inequality, combined with~\eqref{CapDir1}, proves that
 \begin{equation} \label{CapDir3}
  \norm{\varphi_\eps - 1}_{L^\infty(U_0)} \leq C \Dir(\eps)^{1/2}
 \end{equation}
 and, in particular, $\inf_{U_0} \varphi_\eps > 0$
 if~$\Dir(\eps)$ is small enough
 (that is, if~$\Dir(\eps)$ is smaller than
 some quantity depending only on~$C$, not on~$\eps$).
 Now, we identify~$\varphi_\eps$ with a function
 defined in the unit cube and let
 \[
  \widetilde{\varphi}_\eps
  := \min\left(1, \, \frac{\varphi_\eps}{\inf_{U_0} \varphi_\eps}\right) \!.
 \]
 The function~$\widetilde{\varphi}_\eps \in H^1(Y)$
 is equal to~$1$ in~$U_0$ and is an admissible
 competitor for the minimisation problem
 in~\eqref{Cap_eps}. By definition of~$\Cap(\eps)$, we obtain
 \[
  \Cap(\eps)
%   \leq \int_{\T^n} \abs{\nabla\widetilde{\varphi}_\eps}^2 \, \d x
  \leq \frac{\Dir(\eps)}{\left(\inf_{U_0}\varphi_\eps\right)^2}
  \stackrel{\eqref{CapDir3}}{\leq}
   \frac{\Dir(\eps)}{(1 - C\Dir(\eps)^{1/2})^2},
 \]
 so~\eqref{Cap-Dir} follows.

 \medskip
 \noindent
 \textit{Proof of~\eqref{Dir-Cap}.}
 We can assume without loss of generality that
 $\Cap(\eps) \leq \tau$ for some~$\eps$-independent~$\tau$.
 Let~$\zeta_\eps\in H^1(Y)$ be the minimiser for Problem~\eqref{Cap_eps}.
 By definition, $\zeta_\eps - 1$ has zero trace on~$\partial Y$.
 Therefore, by applying the Poincar\'e inequality in~$Y$
 and the inequality~\eqref{Cap<Dir} we proved in the previous step,
 we obtain
 \begin{equation} \label{DirCap1}
  \abs{\norm{\zeta_\eps}_{L^2(Y)} - 1}
   \leq \norm{\zeta_\eps - 1}_{L^2(Y)}
   \leq \Cap(\eps)^{1/2}.
 \end{equation}
 Now, $\zeta_\eps$ can be identified with an element
 of~$H^1(\T^n)$, because it is constant on~$\partial Y$.
 Hence, the function
 \[
  \widetilde{\zeta_\eps}
  := \frac{\zeta_\eps}{\norm{\zeta_\eps}_{L^2(Y)}}
 \]
 is an admissible competitor for
 the minimisation problem~\eqref{Dir_eps},
%  which defined~$\Dir(\eps)$.
 and therefore
 \begin{align*}
  \Dir(\eps)
  \leq \frac{\Cap(\eps)}{\norm{\zeta_\eps}_{L^2(Y)}^2}
  \leq \frac{\Cap(\eps)}{(1 - \Cap(\eps)^{1/2})^2},
 \end{align*}
 as claimed.

 \medskip
 \noindent
 \textit{Proof of~\eqref{lambda_bound+0}.}
 Let~$\kappa\abs{\T^n\setminus H_\eps}>1$ and
 let~$\psi_\eps$ be the minimiser for~\eqref{St_eps}.
 As~$\psi_\eps = 0$ on~$\partial Y$, the Poincar\'e inequality
 (see Lemma~\ref{lemma:Poincare}) %combined with~\eqref{St_eps}
 yields
 \begin{equation} \label{psi_eps-L2}
  \int_{Y\setminus H_\eps} \psi_\eps^2 \, \d x
%    \leq C_0 \int_{Y\setminus H_\eps} \abs{\nabla\psi_\eps}^2 \, \d x
%    \stackrel{\eqref{St_eps}}{=}
   \leq C_0 \, \St(\eps),
 \end{equation}
 where~$C_0$ is the constant given by Lemma~\ref{lemma:Poincare}.
 As a consequence, we have
 $\mcG_{\eps,\kappa}(\psi_\eps) \geq 1 - C_0\kappa\,\St(\eps) > 0$
 if~$\kappa\St(\eps)$ is small enough, which we can
 assume without loss of generality. Then,
 $w_\eps := \psi_\eps/(\mcG_{\eps,\kappa}(\psi_\eps))^{1/2}$
 is an admissible competitor for the minimisation
 problem~\eqref{lambda_eps+}, and by
 definition of~$\lambda(\eps, \, \kappa)$ we obtain
 \[
  \lambda(\eps, \, \kappa)
  \leq \int_{\T^n\setminus H_\eps} \abs{\nabla w_\eps}^2 \, \d x
  \leq \frac{1}{\mcG_{\eps,\kappa}(\psi_\eps)}
   \int_{\T^n\setminus H_\eps} \abs{\nabla \psi_\eps}^2 \, \d x
  \leq \frac{\St(\eps)}{1 - C_0\kappa\,\St(\eps)} .
 \]

 \medskip
 \noindent
 \textit{Proof of~\eqref{lambda_bound-0}.}
 Assume that~$0 < \kappa \leq 1$.
 We consider the minimiser~$\varphi_\eps$
 for Problem~\eqref{Dir_eps}.
 Since $\mcG_{\eps,\kappa}(\kappa^{-1/2}\varphi_\eps) = -1$,
 we can use~$\kappa^{-1/2} \varphi_\eps$ as a test function
 in the definition~\eqref{lambda_eps-} of~$\lambda_*(\eps, \, \kappa)$
 and obtain~\eqref{lambda_bound-0}.
\end{proof}

The inequalities~\eqref{lambda_bound+0} and~\eqref{lambda_bound-0}
in Lemma~\ref{lemma:inequalities}, combined with
assumption~\eqref{hp:critical}, provide uniform bounds
for~$\lambda(\eps, \,\kappa)$ in terms of~$\Dir(\eps)$
and~$\St(\eps)$. More precisely, for~$\eps$ small enough we have
\begin{align}
 \lambda(\eps, \, \kappa)
  &\leq \frac{\St(\eps)}{1 - C\kappa\,\eps^2}
  \hspace{1.5cm} \textrm{if } \kappa\abs{\T^n\setminus H_\eps}>1,
  \label{lambda_bound+} \\
 -\kappa \, \lambda(\eps, \, \kappa) &\leq \Dir(\eps)
  \hspace{1.95cm} \textrm{if } 0 < \eps < 1.
  \label{lambda_bound-}
\end{align}

\begin{remark} \label{rk:Helly}
 Assume that conditions~\eqref{hp:first}--\eqref{hp:critical}
 are satisfied. For~$\eps\in (0, \, 1)$, we consider the functions
 $f_\eps(\kappa) := \eps^{-2} \lambda(\eps, \, \kappa)$,
 $\kappa\in (1, \, +\infty)$.
 These functions are positive, nondecreasing
 (by Corollary~\ref{cor:lambdamonotone}),
 and satisfy~$\limsup_{\eps\to 0} \sup_{\kappa\in K} f_\eps(\kappa) < +\infty$
 for all compact~$K\subseteq (1, \, +\infty)$,
 because of the estimate~\eqref{lambda_bound+}
 and assumption~\eqref{hp:critical}.
 Therefore, by applying Helly's selection
 theorem and a diagonal argument,
 we can extract a subsequence~$\eps_j\to 0$
 so as to have pointwise convergence
 $f_{\eps_j}(\kappa)\to\lambda_*(\kappa)$ as~$j\to +\infty$,
 for~$\kappa\in (1, \, +\infty)$. The limit function satisfies
 \begin{equation}  \label{lambda*bd+}
  \lambda_*(\kappa) \leq \St_* \qquad
   \textrm{for } \kappa > 1,
 \end{equation}
 by passing to the limit as~$j\to +\infty$ in the
 estimate~\eqref{lambda_bound+}.
 By a similar argument, under the
 assumptions~\eqref{hp:first}--\eqref{hp:critical}
 we can extract a further (non-relabelled)
 subsequence so that we also have pointwise convergence
 $f_{\eps_j}(\kappa)\to\lambda_*(\kappa)$
 for~$\kappa\in (0, \, 1)$ and
 \begin{equation}  \label{lambda*bd-} \
  -\kappa\lambda_*(\kappa) \leq \Dir_* \qquad
   \textrm{for } 0 < \kappa < 1.
 \end{equation}
 In particular, $\lambda_*(\kappa)$ is finite for all~$\kappa > 0$.
 We will see later, in Remark~\ref{rk:lambdainfinity},
 that~$\lambda_*(\kappa) \neq 0$ for all~$\kappa\neq 1$.
 Moreover, the inequalities \eqref{lambda*bd+}
 and~\eqref{lambda*bd-}, combined with
 Corollary~\ref{cor:lambdamonotone}, imply that the limits
 \begin{equation} \label{lambda_limits}
  \lambda_0 := -\lim_{\kappa\to 0} \kappa\,\lambda_*(\kappa),
  \qquad \lambda_\infty := \lim_{\kappa\to+\infty} \lambda_*(\kappa)
 \end{equation}
 exist, are finite and satisfy~$\lambda_0 \leq \Dir_*$,
 $\lambda_\infty \leq \St_*$.
\end{remark}

\subsection{Analysis of Problem~\eqref{eigenvalue} as~$\eps\to 0$}
\label{sect:eigenvalue_0}

While the results in Section~\ref{sect:eigenvalue_eps}
are independent of assumptions~\eqref{hp:critical} and~\eqref{hp:limit},
the results in this section (and the next ones)
depend crucially on them.
Throughout the sequel, we denote by~$C$
several positive constants, that do not depend on~$\eps$ and~$\kappa$
(but may depend on the values~$\St_*$ and~$\Dir_*$,
given by~\eqref{hp:critical}).
% First, we prove Proposition~\ref{prop:DirCap}.

\begin{proof}[Proof of Proposition~\ref{prop:DirCap}]
 We need to prove that the limit~$\Cap_* := \lim_{\eps\to 0} \eps^{-2} \Cap(\eps)$ exists and is equal to~$\Dir_*$,
 defined in assumption~\eqref{hp:critical}.
 Under~\eqref{hp:critical}, we have~$\Dir(\eps) \leq C\eps^2$
 for all small enough~$\eps$. Therefore, the inequality~\eqref{Dir-Cap} in Lemma~\ref{lemma:inequalities} takes the form
 \[
  \Cap(\eps)
  \leq \frac{\Dir(\eps)}{(1 - C\eps)^2}.
 \]
 Dividing both sides of this inequality by~$\eps^2$
 and passing to the limit as~$\eps\to 0$, we obtain
 \begin{equation} \label{Cap<Dir}
  \limsup_{\eps\to 0} \frac{\Cap(\eps)}{\eps^2} \leq \Dir_*.
 \end{equation}
 In particular, this shows~$\Cap(\eps) \leq C\eps^2$
 for all small enough~$\eps$.
 As a consequence, we may write the inequality~\eqref{Dir-Cap} as
 \[
  \Dir(\eps) \leq \frac{\Cap(\eps)}{(1 - C\eps)^2}.
 \]
 Dividing both sides of this inequality by~$\eps^2$
 and passing to the limit as~$\eps\to 0$, we obtain
 \begin{equation} \label{Cap>Dir}
  \Dir_* \leq \liminf_{\eps\to 0} \frac{\Cap(\eps)}{\eps^2} .
 \end{equation}
 The proposition follows from~\eqref{Cap<Dir} and~\eqref{Cap>Dir}.
\end{proof}

% From now on, we will always assume that all
% the conditions~\eqref{hp:first}--\eqref{hp:last} are satisfied.
Given a sufficiently small~$\eps \in (0, \, 1)$ and~$\kappa > 0$,
let~$v_\eps := v_{\eps,\kappa}$ be the unique positive minimiser for
either Problem~\eqref{lambda_eps+} or Problem~\eqref{lambda_eps-},
depending on whether~$\kappa>1$ or~$0 < \kappa \leq 1$.
% We know that such~$v_\eps$ is unique, by Lemma~\ref{lemma:exists}.
% Given a function~$g\in L^2(\T^n\setminus H_\eps)$,
% we denote by~$\overline{g}$ the extension by zero inside~$H_\eps$.
Let~$\overline{v_{\eps,\kappa}}\in L^2(\T^n)$,
$\overline{\nabla v_{\eps,\kappa}}\in L^2(\T^n, \, \R^n)$
be the extensions of~$v_{\eps,\kappa}$, $\nabla v_{\eps,\kappa}$
by zero inside~$H_\eps$. We next study the convergence
of~$v_{\eps,\kappa}$ as~$\eps\to 0$. We consider the
cases~$\kappa > 1$ and~$0 < \kappa < 1$ separately.

\begin{lemma} \label{lemma:constantlimit+}
 Assume that~\eqref{hp:first}--\eqref{hp:last} are satisfied.
 Then, for any~$\kappa > 1$, there exists
 a constant~$v_*(\kappa)$ and a (non-relabelled)
 subsequence~$\eps\to 0$ such that
 \begin{equation} \label{constantlimit-conv}
  \norm{\overline{v_{\eps,\kappa}} - v_*(\kappa)}_{L^2(\T^n)}
  + \norm{\overline{\nabla v_{\eps,\kappa}}}_{L^2(\T^n)}
  \to 0 \qquad \textrm{as } \eps\to 0.
 \end{equation}
 Moreover, any possible value for the limit~$v_*(\kappa)$ satisfies
 \begin{equation} \label{v0+}
  0 \leq v_*(\kappa)\leq \left(\kappa^2 - \kappa\right)^{-1/2}
   \qquad \textrm{if } \kappa > 1.
 \end{equation}
\end{lemma}
\begin{proof}
 Let
 \begin{equation} \label{M_eps_kappa}
  M_{\eps,\kappa} :=
  \int_{\T^n\setminus H_{\eps}} v_{\eps,\kappa} \, \d x \geq 0.
 \end{equation}
 Applying the Poincar\'e inequality,
 Lemma~\ref{lemma:Poincare}, and assumption~\eqref{hp:limit},
 we obtain
 \begin{equation} \label{constantlimit1}
  \begin{split}
   \int_{\T^n\setminus H_{\eps}}
   v_{\eps,\kappa}^2 \, \d x - M_{\eps,\kappa}^2
   &= \int_{\T^n\setminus H_{\eps}}
    \left(v_\eps - M_{\eps, \kappa}\right)^2 \d x
   \leq C\,\eps^2 \abs{\lambda_*(\kappa)} \! .
  \end{split}
 \end{equation}
 The constraint $\mcG_{\eps,\kappa}(v_{\eps,\kappa}) = 1$
 can be written as
 \begin{equation} \label{constantlimit2}
  \begin{split}
   \fint_{\partial H_{\eps}} v_{\eps,\kappa}^2 \,\d\sigma
   = 1 + \kappa
    \int_{\T^n\setminus H_{\eps}} v_{\eps,\kappa}^2\, \d x
   = 1 + \kappa M_{\eps,\kappa}^2 + \mathrm{O}(\kappa\eps^2)
  \end{split}
 \end{equation}
 because of~\eqref{constantlimit1}, and so Remark~\ref{rk:bdav} implies
 \begin{equation} \label{bdav-w}
  \fint_{\partial H_\eps} v_{\eps,\kappa} \,\d\sigma
   = \kappa M_{\eps,\kappa}.
 \end{equation}
 Thus, combining~\eqref{constantlimit2} and~\eqref{bdav-w}
 with the H\"older inequality, we deduce
 \begin{equation*}
  \begin{split}
   0 \leq
%    \fint_{\partial H_{\eps}}
%    \left(v_{\eps,\kappa} - \fint_{\partial H_{\eps}}
%     v_{\eps,\kappa} \, \d\sigma\right)^2  \, \d\sigma &=
   \fint_{\partial H_{\eps}} v_{\eps,\kappa}^2 \, \d\sigma
     - \left(\fint_{\partial H_{\eps}} v_{\eps,\kappa} \, \d\sigma\right)^2
   &= 1 + \kappa M_{\eps,\kappa}^2 - \kappa^2 M_{\eps,\kappa}^2
     + \mathrm{O}(\kappa\eps^2).
  \end{split}
 \end{equation*}
 This implies
 \begin{equation*}
  \begin{split}
   M_{\eps,\kappa}^2
   \leq \frac{1 + \mathrm{O}(\kappa\eps^2)}{\kappa^2 - \kappa}.
  \end{split}
 \end{equation*}
 We can now extract a (non-relabelled) subsequence~$\eps\to 0$
 in such a way that~$M_{\eps, \kappa}\to v_*(\kappa)$,
 where the limit satisfies~\eqref{v0+}.
 Condition~\eqref{constantlimit-conv} follows
 from~\eqref{hp:limit}, \eqref{constantlimit1}
 and the fact that~$\abs{H_\eps}\to 0$
 as~$\eps\to 0$ (see Remark~\ref{rk:zerocapacity}).
\end{proof}

In the next lemma, we assume that the limit~$\lambda_\infty$
defined in~\eqref{lambda_limits} is nonzero.
Under the assumptions~\eqref{hp:first}--\eqref{hp:last},
this condition is always satisfied, but we will give the
proof of this claim later (see Remark~\ref{rk:lambdainfinity}
in Section~\ref{sect:lambda*}).

\begin{lemma} \label{lemma:constantlimit-}
 Assume that~\eqref{hp:first}--\eqref{hp:last} are satisfied
 and~$\lambda_\infty\neq 0$.
 Then, for any~$\kappa \in (0, \, 1)$, there exists
 a constant~$v_*(\kappa)$ and a (non-relabelled)
 subsequence~$\eps\to 0$ that satisfies~\eqref{constantlimit-conv}.
 Moreover, any possible value for the limit~$v_*(\kappa)$ satisfies
 \begin{equation} \label{v0-}
  \frac{1}{\left(\kappa - \kappa^2\right)^{1/2}}
  \leq v_*(\kappa)
  \leq \frac{\abs{\lambda_*(\kappa)}^{1/2}
   + \left(\kappa \abs{\lambda_*(\kappa)}
   - \lambda_\infty\abs{\kappa - 1}
   \right)^{1/2}}{\lambda_\infty^{1/2} \abs{\kappa - 1}}
 \end{equation}
\end{lemma}
\begin{proof}%[Proof of Lemma~\ref{lemma:constantlimit-}]
 Let~$w_{\eps,\kappa} := v_{\eps,\kappa} - M_{\eps,\kappa}$,
 where~$M_{\eps,\kappa}$ is defined in~\eqref{M_eps_kappa}.
 Thanks to the assumption~\eqref{hp:extension},
 we can extend~$w_{\eps,\kappa}$ to a
 function in~$H^1(\T^n)$, still denoted~$w_{\eps,\kappa}$
 for simplicity, in such a way that the~$H^1(\T^n)$-norm
 of~$w_{\eps,\kappa}$ is controlled by
 its~$H^1(\T^n\setminus H_{\eps})$-norm, up to
 factors that do not depend on~$\eps$.
 Then, the assumption~\eqref{hp:limit} implies
 that~$w_{\eps,\kappa}\to 0$ strongly in~$H^1(\T^n)$
 as~$\eps\to 0$.
%  Moreover, Remark~\ref{rk:bdav} implies
%  \begin{equation} \label{bdav-w}
%   \fint_{\partial H_\eps} w_{\eps,\kappa} \,\d\sigma
%    = \fint_{\partial H_\eps} v_{\eps,\kappa} \,\d\sigma
%     - \int_{\T^n\setminus H_\eps} v_{\eps,\kappa} \,\d x
%    = (\kappa - 1) M_{\eps,\kappa}.
%  \end{equation}
 Now suppose, for instance, that~$\kappa > 1$.
 The constraint~$\mcG_{\eps,\kappa}(v_{\eps,\kappa}) = -1$
 can be written as a quadratic polynomial of~$M_{\eps,\kappa}$:
 \begin{equation} \label{secondogrado}
  \begin{split}
   \left(1 - \kappa\abs{\T^n\setminus H_{\eps}} \right)
    M_{\eps,\kappa}^2
   + 2M_{\eps,\kappa} \fint_{\partial H_{\eps}}
    w_{\eps,\kappa} \, \d\sigma
   + \fint_{\partial H_{\eps}} w_{\eps,\kappa}^2 \, \d\sigma
   - \kappa \int_{\T^n\setminus H_{\eps}} w_{\eps,\kappa}^2 + 1 = 0.
  \end{split}
 \end{equation}
 Equation~\eqref{secondogrado} clearly has a real
 solution~$M_\eps$, so the discriminant of the said polynomial
 must be non-negative:
 \begin{equation} \label{secondogrado0.5}
  \begin{split}
   \left(\fint_{\partial H_{\eps}}
    w_{\eps,\kappa} \, \d\sigma\right)^2
    - \left(1 - \kappa\abs{\T^n\setminus H_{\eps}} \right)
    \left(\fint_{\partial H_{\eps}} w_{\eps,\kappa}^2 \, \d\sigma
   - \kappa \int_{\T^n\setminus H_{\eps}} w_{\eps,\kappa}^2 \, \d x
   + 1\right) \geq 0.
  \end{split}
 \end{equation}
 Let
 \[
  S_{\eps, \kappa} :=
   \fint_{\partial H_\eps} w_{\eps,\kappa}^2 \,\d\sigma .
 \]
 Equation~\eqref{secondogrado0.5}, combined with the
 H\"older inequality and~\eqref{constantlimit1}, implies
 \[
  S_{\eps,\kappa}
   - \left(1 - \kappa\abs{\T^n\setminus H_{\eps}} \right)
    \left(S_{\eps,\kappa} + 1 + \mathrm{O}(\kappa\eps^2)\right) \geq 0
 \]
 and, since~$\abs{H_\eps}\to 0$ as~$\eps\to 0$
 (by Remark~\ref{rk:zerocapacity}),
 \begin{equation} \label{secondogrado1}
  \liminf_{\eps\to 0} S_{\eps,\kappa}
  \geq \frac{1 - \kappa}{\kappa}.
 \end{equation}
 For each~$\tau > 1$ and small enough~$\eps \in (0, \, 1)$,
 Lemma~\ref{lemma:singletrace} implies
 \begin{equation*}
  \begin{split}
   S_{\eps,\kappa} &\leq \frac{1}{\lambda(\eps, \, \tau)}
    \int_{\T^n\setminus H_{\eps}} \abs{\nabla w_{\eps,\kappa}}^2
    + \tau\int_{\T^n\setminus H_{\eps}}
    w_{\eps,\kappa}^2
   \leq -\frac{\lambda(\eps, \, \kappa)}
    {\lambda(\eps, \, \tau)} + \mathrm{O}(\tau\eps^2).
  \end{split}
 \end{equation*}
 We pass to the limit in both sides of the last inequality,
 first as~$\eps\to 0$, then as~$\tau\to+\infty$.
 Note that~$\lambda_*(\tau)\neq 0$ if~$\tau$ is large enough,
 as we have assumed that~$\lambda_\infty\neq 0$. We therefore obtain
 \begin{equation} \label{secondogrado2}
  \begin{split}
   \limsup_{\eps\to 0} S_{\eps,\kappa}
   \leq -\frac{\lambda_*(\kappa)}{\lambda_\infty} .
  \end{split}
 \end{equation}
 Now, we consider Equation~\eqref{secondogrado} again
 and solve for~$M_{\eps,\kappa}$. Setting
 \[
  L_{\eps,\kappa}
  := \fint_{\partial H_{\eps}} w_{\eps,\kappa} \, \d\sigma
 \]
 and recalling that~$w_\eps\to 0$ strongly in~$L^2(\T^n)$,
 we obtain
 \begin{equation*} %\label{secondogrado3}
  \begin{split}
   M_{\eps,\kappa}
   &= \frac{-L_{\eps,\kappa} \pm \left(L_{\eps,\kappa}^2
     - (1 - \kappa)(S_{\eps,\kappa} + 1)
     + \mathrm{o}_{\eps\to 0}(1)\right)^{1/2}}{1 - \kappa\abs{\T^n\setminus H_\eps}}.
  \end{split}
 \end{equation*}
 Using the estimate $\abs{L_{\eps,\kappa}}\leq S_{\eps,\kappa}^{1/2}$,
 which follows from the H\"older inequality, we obtain
 \begin{equation} \label{secondogrado3}
  \begin{split}
   M_{\eps,\kappa}
%    &\leq \frac{S_{\eps,\kappa}^{1/2} + \left(S_{\eps,\kappa}
%      - (1 - \kappa)(S_{\eps,\kappa} + 1)\right)^{1/2}}{1 - \kappa}
%     + \mathrm{o}_{\eps\to 0}(1)
   \leq \frac{S_{\eps,\kappa}^{1/2} + \left(\kappa S_{\eps,\kappa}
    + \kappa - 1\right)^{1/2}}{1 - \kappa}
    + \mathrm{o}_{\eps\to 0}(1).
  \end{split}
 \end{equation}
 Since~$S_{\eps,\kappa}$ is bounded by~\eqref{secondogrado2},
 it follows that~$M_{\eps,\kappa}$ is bounded, too.
 We can then extract a subsequence in such a way that
 $M_{\eps,\kappa}$ has a limit as~$\eps\to 0$, which we call~$v_*(\kappa)$.
 The limit~$v_*(\kappa)$ satisfies~\eqref{constantlimit-conv}
 because of assumption~\eqref{hp:limit}
 and bound~\eqref{constantlimit1}. By combining~\eqref{secondogrado2}
 and~\eqref{secondogrado3}, we obtain
 \begin{equation} \label{secondogrado4}
  \begin{split}
   M_{\eps,\kappa}
   \leq \frac{\abs{\lambda_*(\kappa)}^{1/2}
    + \left(\kappa \abs{\lambda_*(\kappa)}
    + \lambda_\infty(\kappa - 1)\right)^{1/2}}
    {\lambda_\infty^{1/2} (1 - \kappa)}
    + \mathrm{o}_{\eps\to 0}(1),
  \end{split}
 \end{equation}
 which proves the upper bound in~\eqref{v0-}.
 Furthermore, the same argument as the one we invoked
 in the proof of Lemma~\ref{lemma:constantlimit+} shows that
 \begin{equation} \label{secondogrado5}
  \begin{split}
   M_{\eps,\kappa}^2
   \geq \frac{1 + \mathrm{O}(\kappa\eps^2)}{\kappa - \kappa^2},
  \end{split}
 \end{equation}
 which proves the lower bound in~\eqref{v0-}.
\end{proof}

\begin{remark} \label{rk:futureref}
 Combining the estimate~\eqref{secondogrado4}
 with~\eqref{lambda*bd-}, we obtain the inequality
 \begin{equation} \label{secondogrado-futref}
  \begin{split}
   \kappa^{1/2} \limsup_{\eps\to 0} \int_{\T^n\setminus H_\eps} v_{\eps,\kappa} \, \d x
   \leq \frac{2\abs{\lambda_*(\kappa)}^{1/2} \kappa^{1/2}}
    {\lambda_\infty^{1/2} (1 - \kappa)}
   \leq \frac{2\left(\Dir_*\right)^{1/2}}
    {\lambda_\infty^{1/2} (1 - \kappa)} ,
  \end{split}
 \end{equation}
 which is valid for all~$0 < \kappa < 1$ so long as~$\lambda_\infty\neq 0$.
 This estimate will be useful later on.
\end{remark}

The next lemma is a by-product of the previous proof.

\begin{lemma} \label{lemma:lambda_lower_bd}
 Assume that~\eqref{hp:first}--\eqref{hp:last} are satisfied.
 Then, for any~$\kappa > 0$, there holds
 \begin{equation} \label{lambda_lower_bd}
  \abs{\lambda_*(\kappa)}
   \geq \frac{\lambda_\infty \abs{\kappa - 1}}{\kappa},
 \end{equation}
 where~$\lambda_\infty$ is given by~\eqref{lambda_limits}.
\end{lemma}
\begin{proof}
 We can assume without loss of generality
 that~$\lambda_\infty\neq 0$, for otherwise
 there is nothing to prove. Then,
 the estimate~\eqref{lambda_lower_bd} follows immediately
 from~\eqref{secondogrado1} and~\eqref{secondogrado2}
 when~$0 < \kappa < 1$. If~$\kappa > 1$,
 the argument is analogous.
\end{proof}

Lemma~\ref{lemma:lambda_lower_bd} implies
the function~$\lambda_*$ is nowhere equal to zero
if its limit~$\lambda_\infty$ is nonzero.

\subsection{Proof of Theorem~\ref{th:lambda*}
and of Proposition~\ref{prop:lambda0}}
\label{sect:lambda*}

\begin{proof}[Proof of Theorem~\ref{th:lambda*}]
We already know that the function~$\kappa\mapsto\lambda_*(\kappa)$
is nondecreasing (Corollary~\ref{cor:lambdamonotone}),
nonpositive for~$\kappa\in (0, \, 1)$
and nonnegative for~$\kappa > 1$
(as an immediate consequence of~\eqref{lambda_eps+}, \eqref{lambda_eps-}).
We split the rest of the proof into several steps.

 \medskip
 \noindent
 \emph{Step~1: The function~$\lambda_*$ is continuous away from~$1$.}
 Let us fix~$\kappa_1$, $\kappa_2$ with~$1 < \kappa_1 < \kappa_2$.
 We know that
 \begin{equation} \label{lambda11}
  \lambda_*(\kappa_1) \leq \lambda_*(\kappa_2),
 \end{equation}
 since~$\lambda_*$ is nondecreasing.
 To obtain the opposite inequality,
 we take~$\eps \in (0, \, 1)$ small enough, so that~$\kappa_1
 \abs{\T^n\setminus H_{\eps}} > 1$, and
 use~$v_{\eps, \kappa_1}/\mcG_{\eps,\kappa_2}(v_{\eps,\kappa_1})^{1/2}$
 as a competitor in the definition of~$\lambda(\eps, \, \kappa_2)$.
 We have
 \[
  \mcG_{\eps, \, \kappa_2}(v_{\eps,\kappa_1})
  = \mcG_{\eps, \, \kappa_1}(v_{\eps,\kappa_1})
   - (\kappa_2 - \kappa_1) \int_{\T^n\setminus H_{\eps}}
   v_{\eps,\kappa_1}^2 \, \d x
  = 1 - (\kappa_2 - \kappa_1) \, v_*(\kappa_1)^2
  + \mathrm{o}_{\eps\to 0}(1),
 \]
 at least along a subsequence~$\eps\to 0$, by virtue of
 Lemma~\ref{lemma:constantlimit+}. Therefore,
 \begin{equation*}
  \lambda(\eps, \, \kappa_2)
  \leq \frac{1}{\mcG_{\eps,\kappa_2}(v_{\eps,\kappa_1})}
   \int_{\T^n\setminus H_{\eps}}
   \abs{\nabla v_{\eps,\kappa_1}}^2 \, \d x
  = \frac{\lambda(\eps, \, \kappa_1)}{1 - (\kappa_2 - \kappa_1) \,
   v_*(\kappa_1)^2 + \mathrm{o}_{\eps\to 0}(1)}
 \end{equation*}
 and, by passing to the limit as~$\eps\to 0$
 (possibly, along a subsequence),
 \begin{equation} \label{lambda12}
  \lambda_*(\kappa_2) \leq \frac{\lambda_*(\kappa_1)}
   {1 - (\kappa_2 - \kappa_1) \, v_*(\kappa_1)^2}.
 \end{equation}
 As the function~$\kappa\mapsto v_*(\kappa)$ is bounded
 on any compact subset of~$(1, \, +\infty)$
 (see Lemmas~\ref{lemma:constantlimit+}
 and~\ref{lemma:constantlimit-}),
 the inequalities~\eqref{lambda11} and~\eqref{lambda12}
 imply that~$\lambda_*$ is continuous on~$(1, \, +\infty)$.
 Continuity on~$(0, \, 1)$ is established by a similar argument.

 \medskip
 \noindent
 \emph{Step~2: Left continuity at~$1$.}
 We claim that
 \begin{equation} \label{lambda20}
  \limsup_{\kappa\to 1^-}\frac{\lambda_*(\kappa)}{\kappa - 1}
   \leq \Dir_* ,
 \end{equation}
 which in particular implies that~$\lambda_*$
 is left-continuous at~$\kappa = 1$.
 The proof of~\eqref{lambda20}
 is based on a comparison argument.
 Let~$\zeta_\eps$ be the unique minimiser
 for Problem~\eqref{Cap_eps}.
 Since~$\zeta_\eps = 1$ on~$\partial Y$, %the H\"older inequality,
 the Poincar\'e inequality (applied in~$Y$)
 and Proposition~\eqref{prop:DirCap} imply that
 \begin{equation} \label{lambda20.5}
%   \abs{\norm{\zeta_\eps}_{L^1(Y)} - 1} \leq
  \norm{\zeta_\eps - 1}_{L^1(Y)}
  \leq \norm{\zeta_\eps - 1}_{L^2(Y)}
  \leq (\Cap(\eps))^{1/2} \leq C\eps
 \end{equation}
 for some~$\eps$-independent constant~$C$.
 Moreover, by applying the maximum principle
 to Equation~\eqref{Cap_eps_eigen}, we see
 that~$0 \leq \zeta_\eps \leq 1$ in~$Y$.
 Therefore, the function
 \[
  w_{\eps,\kappa} := \left(\frac{1}{2 - 2\kappa}
   + \zeta_\eps \right)^{1/2}
 \]
 is well-defined and can be identified with an element
 of~$H^1(\T^n)$, since~$\zeta_\eps = 1$ on~$\partial Y$.
 We have
 \[
  \begin{split}
   \mcG_{\eps,\kappa}(w_{\eps,\kappa})
   = \frac{1 - \kappa\abs{\T^n\setminus H_{\eps}}}
    {2 - 2\kappa}
     - \kappa \int_{\T^n\setminus H_\eps}\zeta_\eps \, \d x
   = \frac{1}{2} - \kappa + \mathrm{o}_{\eps\to 0}(1),
  \end{split}
 \]
 because~$\zeta_\eps = 0$ in~$H_\eps$,
 $\int_{Y} \zeta_\eps \, \d x \to 1$ as~$\eps\to 0$
 (by~\eqref{lambda20.5})
 and~$\abs{\T^n\setminus H_\eps}\to 1$ as~$\eps\to 0$
 (see Remark~\ref{rk:zerocapacity}).
 In particular, when~$1/2 < \kappa < 1$ we have
 $\mcG_{\eps,\kappa}(w_\eps) < 0$ for~$\eps$ small enough.
 Then, by definition of~$\lambda(\eps, \, \kappa)$,
 the following upper bound holds:
 \[
  \begin{split}
   -\lambda(\eps, \, \kappa)
   \leq -\frac{1}{\mcG_{\eps,\kappa}(w_\eps)}
    \int_{\T^n\setminus H_{\eps}} \abs{\nabla w_\eps}^2 \, \d x
   \leq \frac{2 - 2\kappa}{2\kappa - 1 + \mathrm{o}_{\eps\to 0}(1)}
    \int_{\T^n\setminus H_\eps}
     \frac{\abs{\nabla\zeta_\eps}^2}
     {2 + 4(1 - \kappa)\zeta_\eps} \, \d x.
  \end{split}
 \]
 Since~$\zeta_\eps\geq 0$, we can further estimate
 the right-hand side as
 \[
  \begin{split}
   -\lambda(\eps, \, \kappa)
   \leq \frac{1 - \kappa}
    {2\kappa - 1 + \mathrm{o}_{\eps\to 0}(1)}
    \int_{\T^n\setminus H_\eps} \abs{\nabla\zeta_\eps}^2 \d x.
  \end{split}
 \]
 Dividing both sides of this inequality
 by~$\eps^2(1 - \kappa)$, passing to the limit as~$\eps\to 0$,
 and applying Proposition~\ref{prop:DirCap},
 we obtain
 \begin{equation} \label{lambda21}
  \begin{split}
   -\frac{\lambda_*(\kappa)}{1 - \kappa}
   &\leq \frac{\Cap_*}{2\kappa - 1}
   = \frac{\Dir_*}{2\kappa - 1}
  \end{split}
 \end{equation}
 for all~$\kappa\in (1/2, \, 1)$.
 The claim~\eqref{lambda20} follows.

 \medskip
 \noindent
 \emph{Step~3: Right continuity at~$1$.}
 We claim that
 \begin{equation} \label{lambda30}
  \limsup_{\kappa\to 1^+} \frac{\lambda_*(\kappa)}{\kappa - 1}
  \leq \St_*.
 \end{equation}
 Again, the proof is based on a comparison argument,
 but this time we consider
 \[
  w_{\eps,\kappa} := \frac{\psi_\eps}{(\kappa - 1)^{1/2}},
 \]
 where~$\psi_\eps$ is the minimiser for Problem~\eqref{St_eps}.
 Since~$\psi_\eps = 0$ on~$\partial Y$,
 we can apply the uniform Poincar\'e inequality
 given by Lemma~\ref{lemma:Poincare} and obtain
 \[
  \norm{\psi_\eps}_{L^2(Y\setminus H_\eps)}
  \leq C\norm{\nabla \psi_\eps}_{L^2(Y\setminus H_\eps)}
  \leq C\eps
 \]
% for some constant~$C$ that does not depend on~$\eps$,
 thanks to assumption~\eqref{hp:critical}.
 Now, we have
 \[
  \begin{split}
   \mcG_{\eps,\kappa}(w_{\eps,\kappa})
   = \frac{1}{\kappa - 1}
    - \frac{\kappa}{\kappa - 1}
    \int_{\T^n\setminus H_\eps}\psi_\eps^2 \, \d x
   = \frac{1}{\kappa - 1} + \mathrm{o}_{\eps\to 0}(1)
  \end{split}
 \]
 and
 \[
  \begin{split}
   \lambda(\eps, \, \kappa)
   \leq \frac{1}{\mcG_{\eps,\kappa}(w_\eps)}
    \int_{\T^n\setminus H_{\eps}} \abs{\nabla w_\eps}^2 \, \d x
   \leq \left(\kappa - 1 + \mathrm{o}_{\eps\to 0}(1)\right)
    \St(\eps).
  \end{split}
 \]
 Now~\eqref{lambda30} follows by dividing both sides of this
 inequality by~$\eps^2(\kappa -1)$ and passing to the limit,
 first as~$\eps\to 0$, then as~$\kappa\to 1^+$.

 \medskip
 \noindent
 \textit{Step~4: One has~$\lambda_\infty = \St_*$.}
%  We consider the limit~$\lambda_\infty :=
%  \lim_{\kappa\to 0} \lambda_*(\kappa)$.
 From Remark~\eqref{rk:Helly}, we know that~$\lambda_\infty$
 exist and that~$\lambda_\infty\leq\St_*$.
 We claim that the opposite inequality also holds, i.e.
 \begin{equation} \label{lambda40}
  \St_* \leq \lambda_\infty.
 \end{equation}
 We will prove~\eqref{lambda40} by testing~\eqref{St_eps}
 against a suitable competitor.
 Given~$\eps \in (0, \, 1)$, $\kappa > 1$, let~$v_{\eps,\kappa}$
 be the positive minimiser of Problem~\eqref{lambda_eps+} and
 \[
  M_{\eps,\kappa} := \int_{\T^n\setminus H_\eps}
  v_{\eps,\kappa} \, \d x.
 \]
 Let~$K_0\subseteq Y$ be a compact set such that
 $H_{\eps} \subseteq K_0$ for all~$\eps \in (0, \, 1)$
 (as given by assumption~\eqref{hp:decreas}),
 and let~$U_0 := \T^n\setminus K_0$.
 The uniform Poincar\'e inequality in Lemma~\ref{lemma:Poincare}
 and assumption~\eqref{hp:limit} imply
 \begin{equation*}
  \norm{v_{\eps,\kappa} -
   M_{\eps,\kappa}}_{L^2(\T^n\setminus H_\eps)}
  \leq C\lambda(\eps, \, \kappa)^{1/2} \leq C\eps,
 \end{equation*}
 where~$C$ does not depend on~$\eps$.
 (In fact, we can choose~$C$ independently of~$\kappa$ too,
 since~$\lambda_*(\kappa) \leq \lambda_\infty \leq \St_*$
 by Remark~\ref{rk:Helly}.) By applying (interior)
 elliptic regularity estimates and a bootstrap argument
 to Equation~\eqref{eigenvalue}, exactly as
 in the proof of~\eqref{Cap-Dir}, we deduce
 \begin{equation} \label{lambda41}
  \begin{split}
   \norm{v_{\eps,\kappa}
    - M_{\eps,\kappa}}_{L^\infty(U_0)}
   \leq C \kappa \, \eps.
  \end{split}
 \end{equation}
%  There is a factor of~$\kappa$ because
%  $\Delta v_{\eps,\kappa} = \kappa \, v_{\eps,\kappa}$
%  in Equation~\eqref{eigenvalue}.
 As a consequence, we can extract a (non-relabelled)
 subsequence~$\eps\to 0$ in such a way
 that~$v_{\eps,\kappa}\to v_*(\kappa)$ uniformly
 in~$U_0$, where~$v_*(\kappa)$
 is the constant given by Lemma~\ref{lemma:constantlimit+}.
%  Note that~$v_*(\kappa)\to 0$ as~$\kappa\to +\infty$,
%  because of the estimate~\eqref{v0+}.
 Now, let
 \[
  \sigma_{\eps,\kappa} := \sup_{U_0} v_{\eps,\kappa}, \qquad
  c_{\eps,\kappa} := \fint_{\partial H_\eps}
  \left(v_{\eps,\kappa} - \sigma_{\eps,\kappa}\right)_+^2 \, \d\sigma,
 \]
 where~$t_+ := \max(t, \, 0)$ denotes the positive part
 of a number~$t\in\R$.
 Assuming for a moment that~$c_{\eps,\kappa}\neq 0$, we define
 \[
  w_{\eps,\kappa} := \frac{\left(v_{\eps,\kappa} - \sigma_{\eps,\kappa}\right)_+}{c_{\eps,\kappa}^{1/2}},
 \]
 which is an admissible competitor for Problem~\eqref{St_eps}.
 However, we need to make sure that~$c_{\eps,\kappa}\neq 0$.
 By applying the  inequalities
 $0 \leq x^2 - (x-a)_+^2 \leq 2ax$, which are
 satisfied by any positive~$x$ and~$a$, we obtain
 \begin{equation} \label{lambda41.5}
  \begin{split}
   0 \leq \fint_{\partial H_\eps} \left(v_{\eps,\kappa}^2 -
   (v_{\eps,\kappa} - \sigma_{\eps,\kappa})_+^2\right)\d\sigma
   \leq 2\sigma_{\eps,\kappa}
    \fint_{\partial H_\eps} v_{\eps,\kappa} \, \d\sigma
   = 2\kappa \sigma_{\eps,\kappa}
    \int_{\T^n\setminus H_\eps} v_{\eps,\kappa} \, \d x.
  \end{split}
 \end{equation}
 The last equality follows from Remark~\ref{rk:bdav}.
 We can pass to the limit as~$\eps\to 0$, with the help
 of Lemma~\ref{lemma:constantlimit+} and~\eqref{lambda41}:
 \[
  \begin{split}
   0 \leq \limsup_{\eps\to 0}
   \fint_{\partial H_\eps} \left(v_{\eps,\kappa}^2 -
   (v_{\eps,\kappa} - \sigma_{\eps,\kappa})_+^2\right)\d\sigma
   \leq 2\kappa \, v_*(\kappa)^2 .
  \end{split}
 \]
 Moreover, the constraint~$\mcG_{\eps,\kappa}(v_{\eps,\kappa}) = 1$
 implies~$\fint_{\partial H_\eps} v_{\eps,\kappa}^2 \, \d\sigma \geq 1$, so
 \begin{equation} \label{lambda42}
  \begin{split}
   \liminf_{\eps\to 0} c_{\eps,\kappa}
   \geq 1 - 2\kappa \, v_*(\kappa)^2
   \stackrel{\eqref{v0+}}{\geq}
    1 - \frac{2\kappa}{\kappa^2 - \kappa} .
  \end{split}
 \end{equation}
 In particular, $c_{\eps,\kappa} > 0$ for~$\eps$
 small enough and~$\kappa$ large enough. Since~$w_{\eps,\kappa}$
 is an admissible competitor in the definition~\eqref{St_eps}
 of~$\St(\eps)$, we obtain
 \begin{equation} \label{lambda43}
  \begin{split}
   \St(\eps)
   &\leq \frac{\lambda(\eps,\kappa)}{c_{\eps,\kappa}} .
  \end{split}
 \end{equation}
 By dividing both sides of this inequality by~$\eps^2$
 and passing to the limit as~$\eps\to 0$, with the help of~\eqref{lambda42} we deduce
 \begin{equation*}
  \begin{split}
   \St_*
   &\leq \lambda_*(\kappa)
   \left(1 - \frac{2\kappa}{\kappa^2 - \kappa}\right)^{-1}.
  \end{split}
 \end{equation*}
 Now~\eqref{lambda40} follows by taking
 the limit as~$\kappa\to+\infty$.

 \begin{remark} \label{rk:lambdainfinity}
  The inequality~\eqref{lambda40}
%   combined with~\eqref{hp:critical},
  implies~$\lambda_\infty = \St_*\neq 0$
  and then~$\lambda_*(\kappa) \neq 0$ for all~$\kappa > 0$,
  $\kappa \neq 1$, by Lemma~\ref{lemma:lambda_lower_bd}.
 \end{remark}

 \begin{remark} \label{rk:v*zero}
  Suppose that~$\kappa > 1$ is such that~$v_*(\kappa) = 0$
  for some~$\kappa > 1$. Then, the inequalities~\eqref{lambda42},
  \eqref{lambda43} above take the form~$c_{\eps,\kappa} \geq 1$,
  $\St(\eps) \leq \lambda(\eps, \, \kappa)$ respectively, yielding
  $\St_* \leq \lambda_*(\kappa)$ in the limit as~$\eps\to 0$.
  Equivalently, we have proved that~$v_*(\kappa) > 0$
  for all~$\kappa > 1$ such that~$\lambda_*(\kappa) < \St_*$.
 \end{remark}

 \medskip
 \noindent
 \textit{Step~5: One has~$\lambda_0 \geq \lambda_\infty$.}
 By multiplying by~$\kappa$ both sides of~\eqref{lambda_lower_bd}
 and passing to the limit as~$\kappa\to 0$, we obtain
 \begin{equation} \label{lambda50}
  \lambda_0 \geq \lambda_\infty.
 \end{equation}
 In particular, combining this inequality with~\eqref{lambda40}
 and assumption~\eqref{hp:critical}, we deduce that~$\lambda_0 > 0$.

 \medskip
 \noindent
 \textit{Step~6: Strict monotonicity of~$\lambda_*$.}
 We already know that~$\lambda_*$ is nondecreasing.
 We claim that~$\lambda_*$ is actually strictly increasing,
 except \emph{possibly} for a neighbourhood of~$+\infty$
 where it might be equal to~$\St_*$.
 Let~$\kappa_1$, $\kappa_2$ be such that
 $0 < \lambda_1 \leq \lambda_2$
 and~$\lambda_*(\kappa_1) = \lambda_*(\kappa_2) < \St_*$.
 Since~$\lambda_*$ changes sign across~$\kappa = 1$
 and is always nonzero except at~$\kappa = 1$
 (thanks to~\eqref{lambda_lower_bd}),
 we must have either~$1 < \kappa_1 \leq \kappa_2$
 or $0 < \kappa_1 \leq \kappa_2 < 1$.
 Suppose, for instance, that~$1 < \kappa_1 \leq \kappa_2$.
 Let~$v_{\eps, \kappa_1}$, $v_{\eps,\kappa_2}$
 be the positive minimisers of~(MIN$^+_{\eps,\kappa_1}$),
 (MIN$^+_{\eps,\kappa_2}$) respectively.
 By testing Problem~(MIN$^+_{\eps,\kappa_1}$)
 against~$v_{\eps,\kappa_2}$, testing Problem~(MIN$^+_{\eps,\kappa_2}$)
 against~$v_{\eps,\kappa_1}$, taking the difference
 and dividing by~$\eps^2$, we obtain
 \begin{equation} \label{lambda61}
  \begin{split}
   \frac{\lambda(\eps, \, \kappa_2)
   - \lambda(\eps, \, \kappa_1)}{\eps^2}
    \fint_{\partial H_\eps} v_{\eps,\kappa_1} v_{\eps,\kappa_2} \, \d\sigma
   = \frac{\kappa_2\lambda(\eps, \, \kappa_2)
    - \kappa_1\lambda(\eps, \, \kappa_1)}{\eps^2}
    \int_{\T^n\setminus H_\eps} v_{\eps,\kappa_1} v_{\eps,\kappa_2}
  \end{split}
 \end{equation}
 The average of~$v_{\eps,\kappa_1} v_{\eps,\kappa_2}$
 on~$\partial H_\eps$ remains bounded as~$\eps\to 0$, because
 \[
  \begin{split}
   \fint_{\partial H_\eps} \abs{v_{\eps,\kappa_1}
    v_{\eps,\kappa_2}} \d\sigma
   \leq \frac{1}{2} \fint_{\partial H_\eps}
    \left(v_{\eps,\kappa_1}^2 + v_{\eps,\kappa_2}^2 \right)
    \d\sigma
   &= 1 + \frac{1}{2} \int_{\T^n\setminus H_\eps}
    \left(\kappa_1 v_{\eps,\kappa_1}^2
    + \kappa_2 v_{\eps,\kappa_2}^2 \right) \d x \\
   &= 1 + \frac{1}{2} \left(\kappa_1 v_*(\kappa_1)^2
    + \kappa_2 v_*(\kappa_2)^2 \right) + \mathrm{o}_{\eps\to 0}(1)
  \end{split}
 \]
 by, Lemma~\ref{lemma:constantlimit+}.
 Since we have assumed that~$\lambda_*(\kappa_1) = \lambda_*(\kappa_2)$,
 the left-hand side of~\eqref{lambda61} converges to zero as~$\eps\to 0$.
 We can apply Lemma~\ref{lemma:constantlimit+}
 to pass to the limit as~$\eps\to 0$ on the right-hand side,
 at least along a subsequence. We thus obtain
 \begin{equation*}
  \begin{split}
   \left(\kappa_2 - \kappa_1\right) \, \lambda_*(\kappa_1) \,
    v_*(\kappa_1) \, v_*(\kappa_2) = 0,
  \end{split}
 \end{equation*}
 where~$v_*(\kappa_1)$, $v_*(\kappa_2)$
 are given by Lemma~\ref{lemma:constantlimit+}.
 Note that~$v_*(\kappa_1) > 0$, $v_*(\kappa_2) > 0$,
 due to the assumption that~$\lambda_*(\kappa_1) = \lambda_*(\kappa_2) < \St_*$ and Remark~\ref{rk:v*zero}.
 Therefore, one has~$\kappa_1 = \kappa_2$, which proves the claim.
 In the case~$0 < \kappa < 1$ the proof is analogous, using
 the fact that~$v_*(\kappa_1)$, $v_*(\kappa_2)$ are nonzero
 as an immediate consequence of the lower bound in~\eqref{v0-}.
 This completes the proof of Theorem~\ref{th:lambda*}.
\end{proof}

\begin{proof}[Proof of Proposition~\ref{prop:lambda0}]
 From Remark~\ref{rk:Helly}, we know that the
 limit~$\lambda_0 := -\lim_{\kappa\to 0} \kappa \lambda_*(\kappa)$
 exists and that~$\lambda_0\leq \Dir_*$.
 Assuming conditions~\eqref{hp:lambda0_geometry}
 and~\eqref{hp:lambda0} are satisfied,
 we claim that the opposite inequality is true, i.e.
 \begin{equation} \label{lambda00}
  \Dir_* \leq \lambda_0.
 \end{equation}
 We will prove~\eqref{lambda00} by testing Problem~\eqref{Dir_eps}
 against a suitable competitor. Given~$0 < \kappa < 1$
 and~$\eps \in (0, \, 1)$, let~$v_{\eps,\kappa}$ be the minimiser
 for Problem~\eqref{lambda_eps+} and
 \[
  \sigma_{\eps, \, \kappa} :=
   \sup_{\partial H_\eps} v_{\eps,\kappa}, \qquad
  c_{\eps, \, \kappa} := \kappa \int_{\T^n\setminus H_\eps}
  (v_{\eps,\kappa} - \sigma_{\eps,\kappa})_+^2 \, \d x.
 \]
 Assuming~$c_{\eps,\kappa}\neq 0$, let
 \[
  w_{\eps,\kappa} := \frac{\kappa^{1/2}\left(v_{\eps,\kappa}
   - \sigma_{\eps,\kappa}\right)_+}{c_{\eps,\kappa}^{1/2}}.
 \]
 From the constraint~$\mcG(v_{\eps,\kappa}) = -1$
 and Lemma~\ref{lemma:constantlimit-}, we have
 \begin{equation} \label{lambda01}
  \begin{split}
   &\limsup_{\kappa\to 0} \limsup_{\eps\to 0} \kappa^{1/2}
    \fint_{\partial H_\eps} v_{\eps,\kappa}^2 \, \d\sigma \\
   &\hspace{2cm} = \limsup_{\kappa\to 0} \limsup_{\eps\to 0}
    \left(- \kappa^{1/2} + \kappa^{3/2}
    \int_{\T^n\setminus H_\eps} v_{\eps,\kappa}^2 \, \d x \right) \\
   &\hspace{2cm} \leq \limsup_{\kappa\to 0}
   \left(-\kappa + \frac{\kappa^{3/2} \abs{\lambda_*(\kappa)}^{1/2}
   + \kappa^{3/2} \left(\kappa \abs{\lambda_*(\kappa)}
   - \lambda_\infty\abs{\kappa - 1}
   \right)^{1/2}}{\lambda_\infty^{1/2} \abs{\kappa - 1}} \right) = 0.
  \end{split}
 \end{equation}
 We claim that this condition, combined with
 assumptions~\eqref{hp:lambda0_geometry}
 and~\eqref{hp:lambda0}, implies
 \begin{equation} \label{lambda02}
  \limsup_{\kappa\to 0} \limsup_{\eps\to 0}
   \kappa^{1/2} \sigma_{\eps, \, \kappa} = 0.
 \end{equation}
 Indeed, suppose~\eqref{lambda02} is not satisfied.
 Then, there exist a number~$\eta > 0$, a
 positive sequence~$\kappa_j\to 0$
 and, for each~$j$, a positive sequence~$\eps_{i,j}\to 0$
 and points~$x_{i,j}\in\partial H_{\eps_{i,j}}$ such that
 \begin{equation*}
   \kappa_j^{1/2} v_{\eps_{i,j}, \, \kappa_j}(x_{i,j})
   \geq \eta.
 \end{equation*}
 Thanks to assumption~\eqref{hp:lambda0},
 it is possible to find a number~$\rho > 0$
 (independent of~$i$, $j$) such that
 \begin{equation*}
   \kappa_j^{1/2} v_{\eps_{i,j}, \, \kappa_j}(x)
   \geq \frac{\eta}{2} \qquad
   \textrm{for all } x\in\partial H_{\eps_{i,j}}
   \textrm{ satisfying } \abs{x - x_{i,j}}
   \leq \rho \, \sigma(\partial H_\eps)^{\frac{1}{n-1}} \! .
 \end{equation*}
 Combining this inequality with assumption~\eqref{hp:lambda0_geometry}, we deduce
 \begin{equation*}
   \kappa_j^{1/2} \fint_{\partial H_\eps}
    v_{\eps_{i,j}, \, \kappa_j} \, \d\sigma \geq C\rho^{n-1}\eta
 \end{equation*}
 for some constant~$C$ that does not depend on~$\rho$, $i$, $j$.
 This contradicts~\eqref{lambda01} and completes the proof of~\eqref{lambda02}.

 Next, we proceed to estimate~$c_{\eps,\kappa}$.
 By applying the inequality~$x^2 - (x - a)_+^2 \leq 2a x$,
 exactly as in~\eqref{lambda41.5}--\eqref{lambda42},
 we obtain
 \[
  \begin{split}
   c_{\eps,\kappa} \geq
    \kappa\int_{\T^n\setminus H_\eps}
    v_{\eps,\kappa}^2 \, \d x - 2\kappa \,\sigma_{\eps,\kappa}
    \int_{\T^n\setminus H_\eps} v_{\eps,\kappa} \, \d x
  \end{split}
 \]
 and hence, from~\eqref{lambda02} and Lemma~\ref{lemma:constantlimit-},
 \begin{equation} \label{lambda03}
  \begin{split}
   \liminf_{\kappa\to 0} \liminf_{\eps\to 0} c_{\eps,\kappa}
   &\geq \liminf_{\kappa\to 0} \frac{\kappa}{\kappa - \kappa^2} = 1.
  \end{split}
 \end{equation}
 Now, taking~$w_{\eps,\kappa}$ as a competitor
 for Problem~\eqref{Dir_eps}, we deduce
 \[
  \Dir_* \leq \frac{\kappa\lambda(\eps,\kappa)}{c_{\eps,\kappa}}.
 \]
 We divide both sides of this inequality by~$\eps^2$
 and pass to the limit, first as~$\eps\to 0$, then as~$\kappa\to 0$.
 The inequality~\eqref{lambda00} now follows because of~\eqref{lambda03}.
\end{proof}

\section{The ``strange term'': proof of Theorem~\ref{th:strange_term}}
\label{sect:homog}

\label{main_proof}

\subsection{Trace inequalities}
\label{sect:trace}

The minimisation problem~\eqref{lambda_eps+} for the case~$\kappa > 1$
allows us to obtain trace inequalities, which will be useful in the
proof of Theorem~\ref{th:strange_term}.
Thanks to the uniform Poincar\'e inequality (Lemma~\ref{lemma:Poincare}),
we can prove a trace inequality similar to~\eqref{singletrace+}
for functions that do \emph{not} satisfy
periodic boundary conditions at the boundary of the unit cell.

\begin{lemma} \label{lemma:singletrace-N}
 For any~$\kappa > 1$, any~$\eps\in (0, \, 1)$
 such that~$\kappa\abs{Y\setminus H_\eps} > 1$,
 and any (not necessarily periodic)
 function~$v\in H^1(Y\setminus H_\eps)$, there holds
 \[
  \fint_{\partial H_\eps} v^2 \, \d\sigma
  \leq \frac{C}{\lambda(\eps, \, \kappa)}
   \int_{Y\setminus H_\eps} \abs{\nabla v}^2 \, \d x
  + 4\kappa \int_{Y\setminus H_\eps} v^2 \, \d x,
 \]
 where
 \begin{equation} \label{singletrace-C}
  C := 2\left(1 + C_P \sup_{0 < \eps < 1}\dist(H_\eps, \, \partial Y)^{-2}\right)
 \end{equation}
 and~$C_P$ is the constant given by Lemma~\ref{lemma:Poincare}.
\end{lemma}

Note that~$C < +\infty$, by virtue of
assumptions~\eqref{hp:H}, \eqref{hp:decreas}
and Lemma~\ref{lemma:Poincare}.

\begin{proof}[Proof of Lemma~\ref{lemma:singletrace-N}]
 Given~$\eps$ and~$\kappa$ as above, we define
 \[
  \xi(y) := \min\left\{1, \,
   \frac{\dist(y, \, \partial Y)}{\dist(H_\eps, \, \partial Y)} \right\}\!,
  \qquad y\in Y.
 \]
 The function~$\xi$ is well-defined,
 Lipschitz-continuous, equal to zero on~$\partial Y$
 and equal to~$1$ in~$H_\eps$.
 For any~$v\in H^1(Y\setminus H_\eps)$,
 let~$\bar{v}: = \fint_{Y\setminus H_\eps} v$
 and consider~$w := \xi v + (1 - \xi)\bar{v}$.
 Then, $w$ is constant on~$\partial Y$,
 so it can be identified with an element
 of~$H^1(\T^n\setminus H_\eps)$.
 By Lemma~\ref{lemma:singletrace}, we have
 \begin{equation} \label{comparelambda1}
  \begin{split}
   \fint_{\partial H_\eps} v^2\, \d\sigma
   &= \fint_{\partial H_\eps} w^2\, \d\sigma
   \leq \frac{1}{\lambda(\eps, \, \kappa)}
    \int_{Y\setminus H_\eps} \abs{\nabla w}^2 \, \d x
    + \kappa \int_{Y\setminus H_\eps} w^2 \, \d x.
  \end{split}
 \end{equation}
 The right-hand side of~\eqref{comparelambda1} can be further
 estimated by observing that
 \[
  \abs{\nabla w}^2 = \abs{\xi\nabla v + (v - \bar{v})\nabla\xi}^2
  \leq 2\abs{\nabla v}^2 + \frac{2 \abs{v - \bar{v}}^2}{\dist(H_\eps, \, \partial Y)^2}
 \]
 and
 \[
  w^2 \leq \left(\abs{v} + \abs{\bar{v}}\right)^2
  \leq 2v^2 + 2\bar{v}^2.
 \]
 Then, the estimate~\eqref{comparelambda1}, combined with
%  the Poincar\'e inequality given by
 Lemma~\ref{lemma:Poincare} and the H\"older inequality, implies
 \[
  \begin{split}
   \fint_{\partial H_\eps} v^2\, \d\sigma
   \leq \frac{2}{\lambda(\eps, \, \kappa)}
    \left(1 + \frac{C_P}{\dist(H_\eps, \, \partial Y)^2}\right)
    \int_{Y\setminus H_\eps} \abs{\nabla v}^2 \, \d x
    + 2\kappa \left(1+ \abs{Y\setminus H_\eps}^2\right)
    \int_{Y\setminus H_\eps} v^2 \, \d x,
  \end{split}
 \]
 and the lemma follows.
\end{proof}

As a consequence of Lemma~\ref{lemma:singletrace-N}, we obtain
a trace inequality on the perforated domain~$\Omega_\eps$.

\begin{lemma}
\label{lemma:trace}
 Let~$(\mu_\eps)_{0 < \eps < 1}$ be any family
 of positive numbers satisfying~\eqref{mu_eps}.
 For any~$\kappa > 1$, any small enough~$\eps \in (0, \, 1)$
 small enough and any~$u\in H^1(\Omega_\eps)$
 such that~$u = 0$ on~$\partial\Omega_\eps\cap\partial\Omega$
 (in the sense of traces), there holds
 \begin{equation} \label{traceeps}
  \frac{1}{\mu_\eps}
  \int_{\Gamma_\eps} u^2 \, \d\sigma
  \leq \frac{C\eps^2}{\lambda(\eps, \, \kappa)}
   \int_{\Omega_\eps} \abs{\nabla u}^2 \, \d x
  + C\kappa \int_{\Omega_\eps} u^2 \, \d x,
 \end{equation}
 where the constant~$C$ does not depend on~$\eps$, $\kappa$, $u$.
 %is given by~\eqref{singletrace-C}.
 In particular, there exists a constant~$M > 0$ such that,
 for all small enough~$\eps \in (0, \, 1)$
 and all~$u\in H^1(\Omega_\eps)$
 with~$u = 0$ on~$\partial\Omega_\eps\cap\partial\Omega$,
 there holds
 \begin{equation} \label{trace}
  \frac{1}{\mu_\eps}
  \int_{\Gamma_\eps} u^2 \, \d\sigma
  \leq M\int_{\Omega_\eps} \left(\abs{\nabla u}^2 + u^2\right) \d x .
 \end{equation}
\end{lemma}
\begin{proof}
 Let~$\overline{\mu}_\eps := \sigma(\partial H_\eps)/\eps$,
 as in~\eqref{mu_eps_bar}.
 By Remark~\ref{rk:mu_eps_bar}, we know~$\overline{\mu}_\eps$
 satisfies~\eqref{mu_eps}. Therefore, it suffices to prove
 the lemma with~$\mu_\eps = \overline{\mu}_\eps$, since any other
 family~$(\mu_\eps^\prime)_{0 < \eps < 1}$
 satisfying~\eqref{mu_eps} also satisfies
 $\mu^\prime_\eps/\overline{\mu}_\eps\to 1$ as~$\eps\to 0$
 and hence $1/\mu_\eps^\prime
 \leq 2/\overline{\mu}_\eps$ for small enough~$\eps$.
 We can identify~$u$ with a (compactly supported) function
 in~$H^1(\R^n\setminus\mcP_\eps)$, by setting~$u := 0$
 on~$\R^n\setminus(\mcP_\eps\cup\Omega)$.
 Then \eqref{traceeps} with~$\mu_\eps = \overline{\mu}_\eps$
 follows by Lemma~\ref{lemma:singletrace-N}
 combined with a scaling argument.
 If we fix~$\kappa$ to be large enough and take~$\eps$ to be small enough
 (depending on~$\kappa$), then~\eqref{hp:limit},
 Theorem~\ref{th:lambda*} and~\eqref{hp:critical} imply
 \[
  \frac{\lambda(\eps, \, \kappa)}{\eps^2}
  \geq \frac{\lambda_*(\kappa)}{2} \geq \frac{\St_*}{4} > 0,
 \]
 so~\eqref{trace} follows from~\eqref{traceeps},
 with~$M := 4\max\{C/\St_*, \, \kappa\}$.
\end{proof}

\begin{remark} \label{rk:tracenoncritical}
 By examining the proof of Lemma~\ref{lemma:trace},
 we see that the inequality~\eqref{traceeps} is completely independent
 of assumptions~\eqref{hp:critical}--\eqref{hp:limit}.
 For~\eqref{trace} to hold, it suffices to assume
 that~$(H_\eps)_{0 < \eps < 1}$
 satisfies~\eqref{hp:first}--\eqref{hp:extension} and that
 \begin{equation} \label{hp:otherregimes}
  \textrm{there exists } \kappa> 1 \textrm{ such that }
  \liminf_{\eps\to 0} \frac{\lambda(\eps, \, \kappa)}{\eps^2} > 0.
 \end{equation}
 This observation will be useful in the analysis
 of non-critical regimes, where assumption~\eqref{hp:critical} fails
 (Section~\ref{sect:otherregimes} below).
\end{remark}

\subsection{Proof of Theorem~\ref{th:strange_term}}
\label{sect:strange_term}

Henceforth, we denote by~$C$ several constants that may depend on
the domain~$\Omega$, the family of sets $(H_{\eps})_{0 < \eps < 1}$
and the parameters~$\alpha$, $\beta$, but not on~$\eps$.
Let~$H^1_{\mathrm{loc}}(\R^n\setminus\mcP_\eps)$,
respectively~$H^1_{\mathrm{loc}}(\R^n)$, denote the set
of measurable functions~$\R^n\setminus\mcP_\eps\to\R$,
respectively~$\R^n\to\R$, such that~$u\in H^1(U\setminus\mcP_\eps)$,
respectively~$u\in H^1(U)$, for all open and bounded set~$U\subseteq\R^n$.
Let~$E_\eps\colon H^1(Y\setminus H_{\eps})\to H^1(Y)$
be the extension operator provided by assumption~\eqref{hp:extension}.
By periodicity and scaling, we can construct an extension
operator
\begin{equation} \label{extension}
 \begin{split}
  &\widetilde{E}_\eps\colon
   H^1_{\mathrm{loc}}(\R^n\setminus\mcP_\eps)
   \to H^1_{\mathrm{loc}}(\R^n)
 \end{split}
\end{equation}
that satisfies the following estimate.
Given~$U\subseteq\R^n$, let~$U_\eps$ be the union of the
cubes~$\eps z + \eps Y$ with~$z\in\Z^n$
such that~$(\eps z + \eps Y)\cap U \neq \varnothing$.
Then, $\widetilde{E}_\eps$ satisfies
\begin{equation*}
 \|\nabla(\widetilde{E}_\eps u)\|_{L^2(U)}
 \leq C \norm{\nabla u}_{L^2(U_\eps)}
\end{equation*}
for all bounded, open~$U\subseteq\R^n$ and any
$u\in H^1_{\mathrm{loc}}(\R^n\setminus\mcP_\eps)$,
where the constant~$C > 0$ does not depend on~$\eps$, $u$.
Moreover, the operator~$\widetilde{E}_\eps$
``preserves the support up to a small error'' --- that is,
for all~$\eps\in (0, \, 1)$ and all
$u\in H^1_{\mathrm{loc}}(\R^n\setminus\mcP_\eps)$, there holds
\begin{equation} \label{extensionsupports}
 \spt(\widetilde{E}_\eps u)
  \subseteq %\left(\spt u\right)_\eps.
  \left\{x\in\R^n\colon \dist(x, \, \spt u)\leq \sqrt{n} \eps \right\} \!.
\end{equation}
Here~$\spt(\widetilde{E}_\eps u)$, $\spt u$
denote the supports of~$\widetilde{E}_\eps u$, $u$ respectively.

Now, let~$u_\eps$ be the unique solution to Problem~\eqref{main_eq}.
We identify~$u_\eps$ with an element of~$H^1_{\mathrm{loc}}(\R^n\setminus\mcP_\eps)$ by setting~$u_\eps := 0$
in~$\R^n\setminus(\Omega\cup\mcP_\eps)$.
By testing Equation~\eqref{main_eq} against~$u_\eps$ itself,
we see that~$u_\eps$ satisfies
\begin{equation*}
 \norm{\nabla u_\eps}_{L^2(\Omega_\eps)}
  + \norm{u_\eps}_{L^2(\Omega_\eps)}
  \leq C\norm{f}_{L^2(\Omega)}
    + C\mu_\eps^{-1/2} \norm{g}_{L^2(\Gamma_\eps)} \! .
\end{equation*}
Furthermore, H\"older inequality implies
$\norm{g}_{L^2(\Gamma_\eps)}\leq \sigma(\Gamma_\eps)^{1/2}
\norm{g}_{L^\infty(\Omega)}$ and assumption~\eqref{mu_eps}
gives~$\sigma(\Gamma_\eps) \leq C \mu_\eps$, so we have
\begin{equation} \label{bounds}
 \norm{\nabla u_\eps}_{L^2(\Omega_\eps)}
  + \norm{u_\eps}_{L^2(\Omega_\eps)}
  \leq C\norm{f}_{L^2(\Omega)}
     + C \norm{g}_{C^0(\overline{\Omega})} \! .
\end{equation}
Note that~\eqref{bounds} holds even when~$\alpha = 0$,
since we can apply the Poincar\'e inequality to~$\widetilde{E}_\eps u_\eps$,
which is supported in a~$(\sqrt{n}\eps)$-neighbourhood of~$\Omega$.
Thanks to~\eqref{bounds}, %and~\eqref{extensionsupports},
we can extract a (non-relabelled) subsequence so that
\begin{equation} \label{weakconv}
 \widetilde{E}_\eps u_\eps\rightharpoonup u_0 \qquad
 \textrm{weakly in } H^1(\R^n) \textrm{ as } \eps\to 0.
\end{equation}
The limit function~$u_0$ is supported in~$\overline{\Omega}$,
because of~\eqref{extensionsupports}.
Our task is to identify the equation satisfied by~$u_0$.
We will need the following results.

\begin{lemma} \label{lemma:surfacetobulk}
 Let~$(\mu_\eps)_{0 < \eps < 1}$ be any family of
 positive numbers satisfying~\eqref{mu_eps}.
 For any~$\eps\in (0, \, 1)$, let~$\tau_\eps\in H^1_{\mathrm{loc}}(\R^n\setminus\mcP_\eps)$
 be an~$\eps\Z^n$-periodic function. Suppose that
 \begin{equation} \label{hp:stb1}
  \sup_{0 < \eps < 1} \left(\norm{\nabla\tau_\eps}_{L^2(\Omega_\eps)}
   + \norm{\tau_\eps}_{L^2(\Omega_\eps)}\right) < +\infty
 \end{equation}
 and that the limit
 \begin{equation} \label{hp:stb2}
  \tau_* := \lim_{\eps\to 0} \fint_{\partial H_\eps}
   \tau_\eps(\eps y) \, \d\sigma(y)
 \end{equation}
 exists in~$\R$. Then, for any continuous
 function~$h\colon\Omega\to\R$ with compact
 support in~$\Omega$, there holds
 \begin{equation} \label{stb}
  \lim_{\eps\to 0} \frac{1}{\mu_\eps}
   \int_{\Gamma_\eps} \tau_\eps h  \, \d\sigma
  = \tau_* \int_{\Omega} h \, \d x.
 \end{equation}
\end{lemma}
\begin{proof}
 We can assume with no loss of generality that
 $\mu_\eps = \overline{\mu}_\eps := \sigma(\partial H_\eps)/\eps$,
 because~$\overline{\mu}_\eps$ satisfies~\eqref{mu_eps}
 (see Remark~\ref{rk:mu_eps_bar}) and any other family
 $(\mu_\eps^\prime)_{0 < \eps < 1}$ satisfying~\eqref{mu_eps}
 is such that~$\mu_\eps^\prime/\overline{\mu}_\eps\to 1$
 as~$\eps\to 0$.
% Therefore, the limit
% on the left-hand side of~\eqref{stb} does not change if we
% replace~$\mu_\eps$ by~$\overline{\mu}_\eps$.
 For all~$z\in\eps\Z^n$, let~$Y_z^\eps := z + \eps Y$.
 Let~$\spt h$ be the support of~$h$, and let~$Z_\eps$
 be the set of points~$z\in\eps\Z^n$
 such that~$Y_z^\eps\cap (\spt h) \neq \varnothing$.
 Since~$\spt h$ is a compact subset of~$\Omega$,
 for~$\eps$ small enough we have~$Y_z^\eps\subseteq\Omega$
 for all~$z\in Z_\eps$.
 Using the continuity of~$h$, we approximate~$h$
 with a constant on each~$Y_z^\eps$ and take advantage
 of the fact that~$\tau_\eps$ is periodic. More precisely,
 we have
 \begin{equation*}
  \begin{split}
   \frac{1}{\overline{\mu}_\eps}
    \int_{\Gamma_\eps} \tau_\eps h \, \d\sigma
   &= \sum_{z\in Z_\eps} \frac{1}{\overline{\mu}_\eps}
    \int_{z + \eps \partial H_\eps} h(z) \tau_\eps(x) \, \d\sigma(x)
    + \sum_{z\in Z_\eps} \frac{1}{\overline{\mu}_\eps}
    \int_{z + \eps \partial H_\eps}
    \left(h(x) - h(z)\right) \tau_\eps(x) \, \d\sigma(x) \\
   &=: \mathrm{I}_\eps + \mathrm{II}_\eps.
  \end{split}
 \end{equation*}
 We analyse the terms on the right-hand side separately.
 The first one can be evaluated by making the
 change of variable~$x = z + \eps y$ and recalling that~$\overline{\mu}_\eps = \sigma(\partial H_\eps)/\eps$:
 \begin{equation*}
  \begin{split}
   \mathrm{I}_\eps
%   = \frac{\eps}{\sigma(\partial H_\eps)} \sum_{z\in Z_\eps}
%     h(z) \int_{z + \eps \partial H_\eps} \tau_\eps(x) \, \d\sigma(x)
   = \eps^n \sum_{z\in Z_\eps} h(z)
    \fint_{\partial H_\eps} \tau_\eps(\eps y) \, \d\sigma(y).
  \end{split}
 \end{equation*}
 Now a Riemann sum for the function~$h$
 has appeared on the right-hand side. Using the continuity of~$h$
 and assumption~\eqref{hp:stb2}, we deduce
 \begin{equation} \label{stb1}
  \begin{split}
   \lim_{\eps\to 0} \mathrm{I}_\eps
     = \tau_* \int_\Omega h \, \d x.
  \end{split}
 \end{equation}
 For the second term, we have
 \begin{equation*}
  \begin{split}
   \abs{\mathrm{II}_\eps}
    &\leq \frac{1}{\overline{\mu}_\eps}
     \left(\int_{\Gamma_\eps} \abs{\tau_\eps} \d\sigma\right)
     \sup_{z\in Z_\eps} \norm{h - h(z)}_{L^\infty(Y_z^\eps)}  \\
    &\leq \left(\frac{1}{\overline{\mu}_\eps}
     \int_{\Gamma_\eps} \tau_\eps^2 \, \d\sigma\right)^{1/2}
     \left(\frac{\sigma(\Gamma_\eps)}{\overline{\mu}_\eps}\right)^{1/2}
     \sup_{z\in Z_\eps} \norm{h - h(z)}_{L^\infty(Y_z^\eps)} \! ,
  \end{split}
 \end{equation*}
 by the H\"older inequality.
 The first factor on the right-hand side is bounded uniformly
 with respect to~$\eps$ because of the trace inequality~\eqref{trace}
 and assumption~\eqref{hp:stb1}. The second is also bounded,
 because~$\overline{\mu}_\eps$ satisfies~\eqref{mu_eps}
 (see Remark~\ref{rk:mu_eps_bar}). The third converges to zero
 as~$\eps\to 0$, by the continuity of~$h$. Therefore, we have
 \begin{equation} \label{stb2}
  \begin{split}
   \lim_{\eps\to 0} \mathrm{II}_\eps = 0.
  \end{split}
 \end{equation}
 Combining~\eqref{stb1} with~\eqref{stb2}, we obtain~\eqref{stb}.
\end{proof}

\begin{remark} \label{rk:stb-otherregimes}
 Lemma~\ref{lemma:surfacetobulk} remains valid
 even if conditions~\eqref{hp:critical}--\eqref{hp:limit}
 fail, so long as~\eqref{hp:first}--\eqref{hp:extension}
 and the weaker condition~\eqref{hp:otherregimes} are satisfied.
 Indeed, these conditions are sufficient for
 the trace inequality~\eqref{trace} (see Remark~\ref{rk:tracenoncritical}).
\end{remark}

The next lemma is inspired by~\cite[Lemma~3.22, p.~125]{JikovKozlovOleinik}.
For each~$h\in L^2(\Omega_\eps)$, we denote by~$\overline{h}$
the extension by zero inside~$\mcP_\eps$.

\begin{lemma} \label{lemma:test}
 Let~$(\mu_\eps)_{0 < \eps < 1}$ be a family
 of positive numbers satisfying~\eqref{mu_eps}.
 Consider~$\alpha\geq 0$, $\beta > 0$,
 $f\in L^2(\Omega)$, and a continuous
 function~$g\colon\overline{\Omega}\to\R$. For all~$\eps\in (0, \, 1)$,
 let~$w_\eps\in H^1_{\mathrm{loc}}(\R^n\setminus\mcP_\eps)$
 be an~$\eps\Z^n$-periodic function
 and~$\xi_\eps\in H^1(\Omega_\eps)$.
 Suppose that the following conditions are satisfied:
 \begin{align}
   \overline{w_\eps} \to 1
    \qquad &\textrm{strongly in } L^2(\Omega)
    \textrm{ as } \eps\to 0, \label{test-L2} \\
   \overline{\nabla w_\eps}\rightharpoonup 0
   \qquad &\textrm{weakly in } L^2(\Omega)
    \textrm{ as } \eps\to 0, \label{test-gradient}\\
   -\overline{\Delta w_\eps} \to \Phi
    \qquad &\textrm{strongly in } H^{-1}(\Omega)
    \textrm{ as } \eps\to 0, \label{test-Laplacian}\\
   \overline{\xi_\eps} \to 0, \quad \overline{\nabla\xi_\eps}\to 0
    \qquad &\textrm{strongly in } L^2(\Omega)
    \textrm{ as } \eps\to 0, \label{test-xi} \\
   \fint_{\partial H_\eps} w_\eps(\eps y) \,\d\sigma(y)
    \to \tau_* \qquad &\textrm{as } \eps\to 0 \label{test-average} \\
   \partial_\nu w_\eps + \frac{\beta \, w_\eps}{\mu_\eps}
    &= \frac{\xi_\eps}{\mu_\eps}
    \qquad \textrm{on } \Gamma_\eps, \label{test-BC}
 \end{align}
 for some~$\Phi\in H^{-1}(\Omega)$ and~$\tau_*\in\R$.
 Then, the limit function~$u_0$ in~\eqref{weakconv} satisfies
 \[
  -\Delta u_0 + \alpha u_0 + \Phi u_0 = f + \tau_* g
 \]
 in the sense of distributions in~$\Omega$.
\end{lemma}
\begin{proof}%[Proof of Lemma~\ref{lemma:test}]
 For~$\varphi\in C^\infty_{\mathrm{c}}(\Omega)$,
 we test Problem~\eqref{main_eq}
 against~$w_\eps\varphi$ and integrate by parts:
 \begin{equation} \label{test0}
  \begin{split}
   &\int_{\Omega_\eps}\left(w_\eps\nabla u_\eps\cdot\nabla\varphi
    + \varphi\nabla u_\eps\cdot\nabla w_\eps
    + \alpha \, u_\eps w_\eps \varphi\right) \d x
    + \frac{\beta}{\mu_\eps} \int_{\Gamma_\eps}
    u_\eps w_\eps \varphi \, \d\sigma \\
   &\hspace{3cm} = \int_{\Omega_\eps} f w_\eps \varphi \, \d x
    + \frac{1}{\mu_\eps} \int_{\Gamma_\eps} g w_\eps \varphi \, \d\sigma.
  \end{split}
 \end{equation}
 Integrating by parts again in the term with~$\nabla w_\eps$
 and recalling~\eqref{test-BC}, we obtain
 \begin{equation} \label{test1}
  \begin{split}
   &\int_{\Omega_\eps} \left(w_\eps\nabla u_\eps\cdot\nabla\varphi
    - \varphi u_\eps\,\Delta w_\eps
    - u_\eps\nabla \varphi\cdot\nabla w_\eps
    + \alpha \, u_\eps w_\eps \varphi\right) \d x\\
   &\hspace{4cm} + \frac{\beta}{\mu_\eps}
    \int_{\Gamma_\eps} u_\eps \, \xi_\eps \varphi  \,\d\sigma
   = \int_{\Omega_\eps} f w_\eps \varphi \, \d x
    + \frac{1}{\mu_\eps} \int_{\Gamma_\eps} g w_\eps \varphi \, \d\sigma.
  \end{split}
 \end{equation}
 We first estimate the surface integral on the left-hand side.
 Applying the H\"older inequality and
 the trace inequality~\eqref{trace} we have
 \begin{equation*}
  \begin{split}
   \frac{1}{\mu_\eps}
    \int_{\Gamma_\eps} u_\eps \, \xi_\eps \varphi  \,\d\sigma
   &\leq \left(\frac{1}{\mu_\eps} \int_{\Gamma_\eps}
    u_\eps^2 \,\d\sigma\right)^{1/2}
    \left(\frac{1}{\mu_\eps} \int_{\Gamma_\eps}
    \xi_\eps^2 \, \varphi^2 \,\d\sigma\right)^{1/2} \\
   &\leq M \left(\norm{\nabla u_\eps}_{L^2(\Omega_\eps)}
    + \norm{u_\eps}_{L^2(\Omega_\eps)}\right)
    \left(\norm{\nabla(\xi_\eps\varphi)}_{L^2(\Omega_\eps)}
    + \norm{\xi_\eps\varphi}_{L^2(\Omega_\eps)}\right)
  \end{split}
 \end{equation*}
 where the constant~$M$ does not depend on~$\eps$.
 Then, recalling~\eqref{bounds} and~\eqref{test-xi},
 we obtain
 \begin{equation} \label{test2}
  \begin{split}
   \frac{1}{\mu_\eps}
    \int_{\Gamma_\eps} u_\eps \, \xi_\eps \varphi  \,\d\sigma
   \to 0 \qquad \textrm{as } \eps\to 0.
  \end{split}
 \end{equation}
 For the surface integral on the right-hand side,
 we apply Lemma~\ref{lemma:surfacetobulk}
 with~$\tau_\eps = w_\eps$ and~$h = g\varphi$.
 (The assumptions of Lemma~\ref{lemma:surfacetobulk}
 are satisfied, due to~\eqref{test-L2}, \eqref{test-gradient},
 and~\eqref{test-average}.) We therefore obtain
 \begin{equation} \label{test3}
  \begin{split}
   \frac{1}{\mu_\eps}
    \int_{\Gamma_\eps} g w_\eps \varphi  \,\d\sigma
   \to \tau_* \int_{\Omega} g \, \varphi \,\d x
    \qquad \textrm{as } \eps\to 0.
  \end{split}
 \end{equation}
 Using~\eqref{weakconv}, \eqref{test-L2},
 \eqref{test-gradient}, \eqref{test-Laplacian} and~\eqref{test2},
 we can now pass to the limit as~$\eps\to 0$
 on both sides of~\eqref{test1}. This yields
 \begin{equation*}
  \begin{split}
   &\int_{\Omega} \left(\nabla u_0\cdot\nabla\varphi
    + \alpha \, u_0 \varphi\right) \d x + \langle \Phi, \, u_0 \varphi\rangle
   = \int_{\Omega} f \varphi \, \d x,
  \end{split}
 \end{equation*}
 and the lemma follows.
\end{proof}

Our next task is to construct a family of
correctors~$w_\eps$ satisfying the assumptions
of Lemma~\ref{lemma:test}. We achieve this by considering
solutions to the ``not so strange'' eigenvalue
problem~\eqref{eigenvalue}. More precisely,
given~$\kappa \in (0, \, 1)$ and~$\eps \in (0, \, 1)$,
we consider the positive minimiser~$v_\eps := v_{\eps,\kappa}$ of~\eqref{lambda_eps-},
which we identify with a~$\Z^n$-periodic function
on~$\R^n\setminus(\mcP_\eps/\eps)$, and define
$w_\eps\colon\R^n\setminus\mcP_\eps\to\R$ as follows:
\begin{equation} \label{corrector}
 w_{\eps, \kappa}(x) :=
  \frac{1}{M_{\eps,\kappa}}
  v_{\eps,\kappa}\!\left(\frac{x}{\eps}\right)
  \qquad \textrm{where } \ M_{\eps,\kappa}
  := \int_{\T^n\setminus H_\eps} v_{\eps,\kappa}(x) \, \d x,
  \quad x\in\R^n\setminus\mcP_\eps.
\end{equation}
We also define
\begin{equation} \label{eta_eps}
 \eta_{\eps,\kappa} := \abs{\frac{\lambda(\eps, \, \kappa)}
    {\eps^2} - \lambda_*(\kappa)}
    + \abs{\frac{\mu_\eps}{\overline{\mu}_\eps} - 1} \! ,
\end{equation}
where~$\overline{\mu}_\eps := \eps^{-1} \sigma(\partial H_\eps)$.
% We are going to apply Lemma~\ref{lemma:test} to~$w_{\eps,\kappa}$,
% for a suitable choice of~$\kappa$.
In the next statement, we show that~$w_{\eps,\kappa}$
(for a suitable~$\kappa$)
satisfies conditions~\eqref{test-L2}--\eqref{test-BC}
in Lemma~\ref{lemma:test}, with quantitative
estimates on the convergence rates,
which will be useful in the proof of Proposition~\ref{prop:abstractrate} later on.

\begin{lemma} \label{lemma:corrector}
 Let~$(\mu_\eps)_{0 < \eps < 1}$ be any family
 of positive numbers satisfying~\eqref{mu_eps}.
 For any~$\kappa\in (0, \, 1)$ and~$\eps \in (0, \, 1)$,
 the function~$w_{\eps,\kappa}$ satisfies
 \begin{align}
  \norm{\overline{w_{\eps, \kappa}} - 1}_{L^2(\Omega)}
  &\leq C\eps\kappa^{-1/2},
%  C \abs{\lambda(\eps, \, \kappa)}^{1/2} + C\eps,
  \label{correctorL2} \\
  \norm{-\overline{\Delta w_{\eps, \kappa}} + \kappa \, \lambda_*(\kappa) \overline{w_{\eps, \kappa}}}_{L^2(\Omega)}
  &\leq C \kappa \, \eta_{\eps,\kappa}
   + C\eps\kappa^{-1/2}, \label{correctorDelta} \\
  \frac{1}{\mu_\eps} \int_{\Gamma_\eps} w_{\eps,\kappa}^2 \, \d\sigma
  &\leq C, \label{corrector_L2Gamma}
 \end{align}
 and
 \begin{equation} \label{corrector_Gamma}
   \partial_\nu w_{\eps,\kappa}
    = \left(\lambda_*(\kappa) + \chi_{\eps,\kappa}
     \right) \mu_\eps^{-1} w_{\eps,\kappa}
 \end{equation}
 in the sense of (normal) traces on~$\Gamma_\eps$,
 where~$\chi_{\eps,\kappa}$ is a constant such that
 \begin{equation} \label{corrector_chi}
   \abs{\chi_{\eps,\kappa}} \leq
    C \kappa^{-1} \eta_{\eps,\kappa}
 \end{equation}
 for all small enough~$\eps$. In all of the conditions above,
 $C$ denotes a constant independent of~$\eps$, $\kappa$.
 Moreover, $\overline{\nabla w_{\eps,\kappa}}\rightharpoonup 0$
 weakly in~$L^2(\Omega)$ as~$\eps\to 0$.
\end{lemma}
\begin{proof}%[Proof of Lemma~\ref{lemma:corrector}]
 We split the proof into several steps.
 We will denote by~$C$ several constants
 independent of~$\eps$, $\kappa$.

 \medskip
 \noindent
 \textit{Proof of~\eqref{correctorL2}.}
 By the estimate~\eqref{secondogrado5} in the proof
 of Lemma~\ref{lemma:constantlimit-}, we know that
 \begin{equation} \label{corrector-M}
  M_{\eps,\kappa}^2 \geq 4 + \mathrm{O}(\kappa\eps^2) .
 \end{equation}
 Therefore, the uniform Poincar\'e inequality
 provided by Lemma~\ref{lemma:Poincare} and~\eqref{lambda_eps-} implies
 \begin{equation*} %\label{corrector-1}
  \int_{\T^n\setminus H_\eps}
   \left(\frac{v_{\eps, \kappa}}{M_{\eps,\kappa}} - 1\right)^2 \d x
   = \frac{1}{M_{\eps,\kappa}^2} \int_{\T^n\setminus H_\eps}
   \left(v_{\eps, \kappa} - M_{\eps,\kappa}\right)^2 \d x
   \leq C\abs{\lambda(\eps, \, \kappa)}
 \end{equation*}
 for some~$\eps$-independent constant~$C$.
 An explicit computation based on the periodicity of~$w_\eps$
 and a change of variable shows that
 \begin{equation*}
  \norm{w_{\eps, \kappa} - 1}_{L^2(\Omega_\eps)}
  \leq C\norm{\frac{v_{\eps,\kappa}}{M_{\eps,\kappa}} - 1}_{L^2(\T^n\setminus H_\eps)}
  \leq C \abs{\lambda(\eps, \, \kappa)}^{1/2}.
 \end{equation*}
 The estimate~\eqref{lambda_bound-0}
 and assumption~\eqref{hp:critical} imply
 \begin{equation} \label{corrector1}
  \norm{w_{\eps, \kappa} - 1}_{L^2(\Omega_\eps)}
  \leq \left(\kappa^{-1} \, \Dir(\eps)\right)^{1/2}
  \leq C\eps \, \kappa^{-1/2} .
 \end{equation}
 Since~$\Omega\cap\mcP_\eps$ is the union
 of~$\mathrm{O}(\eps^{-n})$ copies of~$\eps H_\eps$,
 we have~$\abs{\Omega\cap\mcP_\eps} \leq C \abs{H_\eps}$
 and the right-hand side can be further estimated
 as in Remark~\ref{rk:zerocapacity}:
 \begin{equation} \label{corrector-meas}
  \abs{\Omega\cap\mcP_\eps} \leq C\eps^2.
 \end{equation}
%  The volume of~$H_\eps$ can be further estimated
%  by applying the Poincar\'e inequality to the function
%  $1 - \varphi_\eps$, where~$\varphi_\eps$ is a minimiser for
%  the right-hand side of~\eqref{Dir_eps}. We obtain
%  \begin{equation} \label{corrector-meas}
%   \abs{\Omega\cap\mcP_\eps} \leq C\abs{\mcP_\eps}
%   \leq C\int_{Y} (1 - \varphi_\eps)^2 \d x
% %   \leq C\int_{Y} \abs{\nabla \varphi_\eps}^2 \d x
%   \leq C \Dir(\eps) \leq C\eps^2,
%  \end{equation}
%  because of the assumption~\eqref{hp:critical}.
 Combining~\eqref{corrector1} with~\eqref{corrector-meas},
 we obtain~\eqref{correctorL2}.

 \medskip
 \noindent
 \textit{Proof of~\eqref{correctorDelta}.}
 Equation~\eqref{eigenvalue} and the definition~\eqref{eta_eps}
 of~$\eta_{\eps,\kappa}$ imply that
 \[
  \begin{split}
   -\Delta w_{\eps, \kappa}(x)
%   = -\frac{\kappa \, \lambda(\eps, \, \kappa)}{\eps^2 M_{\eps,\kappa}}
%    v_{\eps, \kappa}\left(\frac{x}{\eps}\right)
   = -\kappa \, \eps^{-2} \, \lambda(\eps, \, \kappa) \, w_{\eps, \kappa}(x)
   = -\kappa\left(\lambda_*(\kappa)
   + \mathrm{O}(\eta_{\eps, \kappa})\right)
   w_{\eps, \kappa}(x)
  \end{split}
 \]
 for all~$x\in\R^n\setminus\mcP_{\eps}$. Therefore, we have
 \[
  \norm{-\Delta w_{\eps,\kappa} + \kappa\lambda_*(\kappa)}_{L^2(\Omega_\eps)}
  \leq
  \kappa \abs{\lambda_*(\kappa)} \, \norm{w_{\eps,\kappa} - 1}_{L^2(\Omega_\eps)}
  + \mathrm{O}(\kappa \eta_{\eps,\kappa})
   \, \norm{w_{\eps,\kappa}}_{L^2(\Omega_\eps)}
 \]
 and, taking~\eqref{lambda_bound-}
 and~\eqref{correctorL2} into account,
 \eqref{correctorDelta} follows.

 \medskip
 \noindent
 \textit{Proof of~\eqref{corrector_L2Gamma}.}
 We assume first that~$\mu_\eps =
 \overline{\mu}_\eps := \sigma(\partial H_\eps)/\eps$.
 Then, recalling that~$w_{\eps,\kappa}$ is periodic
 and applying a change of variable, we obtain
 \[
  \frac{1}{\overline{\mu}_\eps}
   \int_{\Gamma_\eps} w_{\eps,\kappa}^2 \,\ d\sigma
  \leq \frac{C}{M_{\eps,\kappa}^2}
   \fint_{\partial H_\eps} v_{\eps,\kappa}^2 \, \d\sigma
  \leq \frac{C\kappa}{M_{\eps,\kappa}^2}
   \int_{\T^n\setminus H_\eps} v_{\eps,\kappa}^2 \, \d\sigma.
 \]
 The last inequality follows by the constraint
 $\mcG_{\eps,\kappa}(v_{\eps,\kappa}) = -1$.
 Applying the inequality~$(a + b)^2 \leq 2a^2 + 2b^2$
 and %a uniform Poincar\'e inequality
 Lemma~\ref{lemma:Poincare}, %as in~\eqref{corrector-1},
 we obtain
 \[
  \begin{split}
   \frac{1}{\overline{\mu}_\eps}
    \int_{\Gamma_\eps} w_{\eps,\kappa}^2 \,\ d\sigma
   \leq \frac{C\kappa}{M_{\eps,\kappa}^2}
    \left(\int_{\T^n\setminus H_\eps}
    (v_{\eps,\kappa} - M_{\eps,\kappa})^2 \, \d\sigma
    + M_{\eps,\kappa}^2 \abs{\T^n\setminus H_\eps}\right)
   \leq \frac{C\kappa\abs{\lambda(\eps, \, \kappa)}}
    {M_{\eps,\kappa}^2} + C\kappa.
  \end{split}
 \]
 We know that~$\kappa\abs{\lambda(\eps,\kappa)}$
 is bounded above by~\eqref{lambda_bound-0} and we
 have observed already that~$M_{\eps,\kappa}$
 is bounded below by~\eqref{corrector-M}
 (taking into account that~$M_{\eps,\kappa} > 0$),
 so we have proved that~\eqref{corrector_L2Gamma}
 holds when~$\mu_\eps = \overline{\mu}_\eps$.
 Now, \eqref{corrector_L2Gamma} remains true for any
 choice of~$\mu_\eps$ satisfying~\eqref{mu_eps},
 since for any such~$\mu_\eps$ one
 has~$\mu_\eps/\overline{\mu}_\eps\to 1$ as~$\eps\to 0$.

 \medskip
 \noindent
 \textit{Proof of~\eqref{corrector_Gamma}--\eqref{corrector_chi}.}
 Equation~\eqref{eigenvalue} implies
 \[
  \begin{split}
   \partial_\nu w_{\eps,\kappa}(x)
   = \frac{\lambda(\eps, \, \kappa)}{\eps \sigma(\partial H_\eps)}
     w_{\eps,\kappa}(x)
   = \frac{\lambda(\eps, \, \kappa)}{\eps^2 \overline{\mu}_\eps}
     w_{\eps,\kappa}(x)
   = \frac{\lambda(\eps, \, \kappa) \, \mu_\eps}
     {\eps^2 \overline{\mu}_\eps}
   \frac{w_{\eps,\kappa}(x)}{\mu_\eps}
  \end{split}
 \]
 for~$x\in\partial\mcP_\eps$, in the sense of
 (normal) traces. Then, \eqref{corrector_Gamma}
 is satisfied with
 \[
  \chi_{\eps,\kappa} := \frac{\lambda(\eps,\kappa) \mu_\eps}
   {\eps^2 \overline{\mu}_\eps} - \lambda_*(\kappa),
 \]
 for which we have
 \[
  \abs{\chi_{\eps,\kappa}}
  \leq \frac{\mu_\eps}{\overline{\mu}_\eps}
  \abs{\frac{\lambda(\eps,\kappa)}{\eps^2} - \lambda_*(\kappa)}
   + \abs{\lambda_*(\kappa)} \abs{\frac{\mu_\eps}{\overline{\mu}_\eps}- 1}
  \leq \left(\frac{\mu_\eps}{\overline{\mu}_\eps}
   + \abs{\lambda_*(\kappa)}\right) \eta_{\eps,\kappa}  .
 \]
 The convergence~$\mu_\eps/\overline{\mu}_\eps\to 1$ as~$\eps\to 0$
 implies~$\mu_\eps/\overline{\mu}_\eps\leq 2$ for small enough~$\eps$,
 and~$\kappa\abs{\lambda_*(\kappa)}\leq \Dir_*$ by~\eqref{lambda_bound-}.
 The bound~\eqref{corrector_chi} now follows.

 \medskip
 \noindent
 \textit{Proof that $\overline{\nabla w_{\eps,\kappa}}\rightharpoonup 0$ weakly in~$L^2(\Omega)$ as~$\eps\to 0$.}
 The family~$(\overline{\nabla w_{\eps, \kappa}})_{0 < \eps < 1}$
 is bounded in~$L^2(\Omega)$, by periodicity
 and assumption~\eqref{hp:limit}.
 Let~$\widetilde{E}_\eps$ be the extension
 operator given by~\eqref{extension}, and let~$\chi_\eps$
 be the indicator function of the set~$\mcP_\eps$
 (i.e., $\chi_\eps := 1$ on~$\mcP_\eps$ and~$\chi_\eps := 0$ otherwise).
 The family $(\widetilde{E}_\eps w_{\eps,\kappa})_{0 < \eps < 1}$
 is bounded in~$H^1(\Omega)$ and hence, upon extracting
 a (non-relabelled, countable) subsequence,
 we can assume that~$\widetilde{E}_\eps w_{\eps,\kappa}$
 converges weakly in~$H^1(\Omega)$ to some limit~$\xi$.
 Since~$\chi_\eps \to 0$ strongly in~$L^2(\Omega)$
 by~\eqref{corrector-meas}, $\widetilde{E}_\eps w_{\eps,\kappa}
 - \overline{w_{\eps,\kappa}} = \chi_\eps\,
 \widetilde{E}_\eps w_{\eps,\kappa} \to 0$
 strongly in~$L^1(\Omega)$, and $\overline{w_{\eps,\kappa}}\to 1$
 strongly in~$L^2(\Omega)$, it follows that~$\xi = 1$.
 This implies that
 $\nabla(\widetilde{E}_\eps w_{\eps\,\kappa}) \rightharpoonup 0$
 weakly in~$L^2(\Omega)$ and, hence, $\overline{\nabla w_\eps}
 = \nabla(\widetilde{E}_\eps w_\eps) \, \chi_\eps \rightharpoonup 0$
 weakly in~$L^2(\Omega)$.
 Since the limit is uniquely identified (and the weak topology on
 bounded subsets of~$L^2(\Omega)$ is metrisable), we have
 weak convergence $\overline{\nabla w_\eps}\rightharpoonup 0$
 not only along a subsequence, but for the whole sequence
 as~$\eps\to 0$.
\end{proof}

Now we have all the tools we need to prove Theorem~\ref{th:strange_term} and Proposition~\ref{prop:abstractrate}.

\begin{proof}[Proof of Theorem~\ref{th:strange_term}]
 Theorem~\ref{th:lambda*} implies that for all~$\beta > 0$
 there exists a unique~$\kappa_*(\beta)\in (0, \, 1)$
 such that~$-\lambda_*(\kappa_*(\beta)) = \beta$.
 Consider the function~$w_\eps := w_{\eps, \, \kappa_*(\beta)}$
 given by~\eqref{corrector}.
 By Remark~\ref{rk:bdav}, we have
 \[
  \fint_{\partial H_\eps} w_{\eps,\kappa_*(\beta)}(\eps y) \, \d\sigma(y)
  = \frac{1}{M_{\eps,\kappa_*(\beta)}} \fint_{\partial H_\eps} v_{\eps,\kappa_*(\beta)}(y) \, \d\sigma(y)
  = \kappa_*(\beta)
 \]
 for all~$\eps\in (0, \, 1)$.
 Lemma~\ref{lemma:corrector} implies
 that~$w_\eps$ satisfies all assumptions~\eqref{test-L2}--\eqref{test-BC}
 with~$\Phi = \beta\kappa_*(\beta)$,
 $\tau_* = \kappa_*(\beta)$ and
 $\xi_\eps := \chi_{\eps, \, \kappa_*(\beta)} w_\eps$,
 where~$\chi_{\eps, \, \kappa}$ is given in~\eqref{corrector_Gamma}.
 Therefore, by Lemma~\ref{lemma:test} $u_0$
 is a solution of Problem~\eqref{hom_eq}.
 Since the limit problem is uniquely identified and
 has a unique solution, the convergence~$\overline{u_\eps}\to u_0$
 holds not only along a subsequence, but also for the whole
 family~$(\overline{u_\eps})_{0 < \eps < 1}$.
\end{proof}

\begin{remark} \label{rk:f_eps}
 Suppose that~$(f_\eps)_{0 < \eps < 1}$, $(g_\eps)_{0 < \eps < 1}$
 are families of functions in~$L^2(\Omega)$, $C^0(\overline{\Omega})$ respectively,
 such that~$f_\eps\rightharpoonup f$ weakly in~$L^2(\Omega)$,
 $g_\eps\to g$ uniformly. Then, the solutions~$u_\eps$ to
 \begin{equation} \label{main_eq_eps}
  \begin{cases}
   -\Delta u_\eps + \alpha u_\eps = f_\eps & \textrm{in } \Omega_\eps, \\
   \hspace{0.8mm} \partial_\nu u_\eps
    + \dfrac{\beta\,u_\eps}{\mu_\eps}
    = \dfrac{g_\eps}{\mu_\eps}
    & \textrm{on } \Gamma_\eps, \\
   u_\eps = 0 &\textrm{on } \partial\Omega_\eps\cap\partial\Omega
  \end{cases}
 \end{equation}
 still satisfy~$\norm{u_\eps - u_0}_{L^2(\Omega_\eps)} \to 0$
 as~$\eps\to 0$, where~$u_0$ is the unique solution to~\eqref{hom_eq}.
 The proof of this claim is analogous to that of Theorem~\ref{th:strange_term}, the only difference
 being the source terms. However, if~$w_{\eps,\kappa}$ are
 the correctors defined in~\eqref{corrector}
 and~$\varphi\in C^\infty_{\mathrm{c}}(\Omega)$
 is a cut-off function, we have
 \begin{equation*} %\label{f_eps}
  \int_{\Omega_\eps} \left(f_\eps - f\right)
   w_{\eps,\kappa} \varphi \, \d x
  = \int_{\Omega} (f_\eps - f) \,
   \overline{w_{\eps,\kappa}} \, \varphi \to 0
   \qquad \textrm{as } \eps\to 0,
 \end{equation*}
 by virtue of the weak convergence
 $f_\eps\rightharpoonup f$ in~$L^2(\Omega)$
 and~\eqref{correctorL2}. Moreover, applying the
 H\"older inequality, we obtain
 \begin{equation*} %\label{g_eps}
  \abs{\frac{1}{\mu_\eps} \int_{\Gamma_\eps}
   \left(g_\eps - g\right) w_{\eps,\kappa} \varphi \, \d \sigma}
  \leq \norm{\varphi}_{L^\infty(\Omega)}
   \norm{g_\eps - g}_{C^0(\overline{\Omega})}
   \left(\frac{\sigma(\Gamma_\eps)}{\mu_\eps}\right)^{1/2}
   \frac{\norm{w_{\eps,\kappa}}_{L^2(\Gamma_\eps)}}{\mu_\eps^{1/2}}
    \to 0 \qquad \textrm{as } \eps\to 0,
 \end{equation*}
 in view of~\eqref{mu_eps}, \eqref{corrector_L2Gamma}
 and the uniform convergence~$g_\eps\to g$. Therefore,
 the arguments in Lemma~\ref{lemma:test}
 easily extend to case of $\eps$-dependent source terms.
 The rest of the proof requires no modifications.
\end{remark}

\begin{proof}[Proof of Proposition~\ref{prop:abstractrate}]
 We proceed along the lines of~\cite[Th\'eor\`eme~1.1]{KacimiMurat}.
 Assume that the domain~$\Omega$ is of class~$C^2$,
 that~$f\in L^\infty(\Omega)$ and that~$g = 0$.
 By elliptic regularity and the Sobolev embedding
 $W^{2,p}(\Omega) \hookrightarrow W^{1,\infty}(\Omega)$,
 with~$p > n$, the unique solution~$u_0$
 to Problem~\eqref{hom_eq} satisfies
 \begin{equation} \label{rate0}
  \norm{u_0}_{L^\infty(\Omega)} + \norm{\nabla u_0}_{L^\infty(\Omega)}
   + \norm{\Delta u_0}_{L^\infty(\Omega)} \leq C \norm{f}_{L^\infty(\Omega)},
 \end{equation}
 for some constant~$C$ that does not depend on~$f$.
 Next, we consider the corrector~$w_\eps := w_{\eps,\kappa_*(\beta)}$
 again and write an equation for $U_\eps := u_\eps - w_\eps u_0$.
 Taking Equation~\eqref{main_eq} and~\eqref{hom_eq} into account,
 we obtain
 \[
  \begin{split}
   -\Delta U_\eps + \alpha U_\eps
   &= -\Delta u_\eps + w_\eps \Delta u_0
    + 2 \nabla w_\eps \cdot \nabla u_0 + u_0 \Delta w_\eps
    + \alpha u_\eps - \alpha w_\eps u_0  \\
   &= (1 - w_\eps) f
    + (\beta \kappa_*(\beta) w_\eps + \Delta w_\eps) u_0 %\\
%    &\hspace{1.95cm}
   + 2\,\mathrm{div}\left((w_\eps - 1)\nabla u_0\right)
    - 2(w_\eps - 1) \Delta u_0
  \end{split}
 \]
 in~$\Omega_\eps$. By testing this equation against~$U_\eps$
 and integrating by parts, we deduce
 \[
  \begin{split}
   &\int_{\Omega_\eps} \left(\abs{\nabla U_\eps}^2
    + \alpha U_\eps^2\right) \d x
    - \int_{\Gamma_\eps} U_\eps \, \partial_\nu U_\eps \, \d\sigma
    = \int_{\Omega_\eps} \left((1 - w_\eps) f
    + (\beta \kappa_*(\beta) w_\eps + \Delta w_\eps)
     u_0 \right) U_\eps \, \d x \\
   &\hspace{1.5cm} - 2\int_{\Omega_\eps}
    (w_\eps - 1) \left(\nabla u_0\cdot\nabla U_\eps
    + U_\eps \, \Delta u_0\right) \d x
    + 2\int_{\Gamma_\eps} (w_\eps - 1) U_\eps
    \, \partial_\nu u_0 \, \d \sigma.
  \end{split}
 \]
 Equation~\eqref{corrector_Gamma} in Lemma~\ref{lemma:corrector}
 together with our choice of~$\kappa = \kappa_*(\beta)$
 implies that
 \[
  \partial_\nu w_{\eps} + \frac{\beta \, w_{\eps}}{\mu_\eps}
   = \frac{\chi_{\eps} w_{\eps}}{\mu_\eps}
 \]
 in the sense of (normal) traces on~$\Gamma_\eps$,
 where~$\chi_\eps := \chi_{\eps, \, \kappa_*(\beta)}$
 is as in~\eqref{corrector_Gamma}.
 This equality, together with the boundary conditions
 for~$u_\eps$ in~\eqref{main_eq} (with~$g = 0$), implies
 \[
   \partial_\nu U_{\eps}
    + \frac{\beta \, U_{\eps}}{\mu_\eps}
   = - w_\eps \, \partial_\nu u_0
     - \frac{\chi_{\eps} w_{\eps} u_0}{\mu_\eps}
   \qquad \textrm{on } \Gamma_\eps,
 \]
 and hence
 \[
  \begin{split}
   &\int_{\Omega_\eps} \left(\abs{\nabla U_\eps}^2
    + \alpha U_\eps^2\right) \d x
    + \frac{\beta}{\mu_\eps} \int_{\Gamma_\eps} U_\eps^2 \, \d\sigma
    = \int_{\Omega_\eps} \left((1 - w_\eps) f
    + (\beta \kappa_*(\beta) w_\eps + \Delta w_\eps)
     u_0 \right) U_\eps \, \d x \\
   &\hspace{1.5cm} - 2\int_{\Omega_\eps}
    (w_\eps - 1) \left(\nabla u_0\cdot\nabla U_\eps
    + U_\eps \, \Delta u_0\right) \d x
    + \int_{\Gamma_\eps} (w_\eps - 2) U_\eps
    \, \partial_\nu u_0 \, \d \sigma
    - \frac{\chi_\eps}{\mu_\eps}
    \int_{\Gamma_\eps} U_\eps w_\eps u_0 \, \d\sigma.
  \end{split}
 \]
 Let~$\tau > 0$ be a small parameter, to be chosen later.
 By making repeated use of the inequality
 $\abs{ab} \leq \theta a^2/2 + b^2/(2\theta)$,
 valid for all~$a$, $b$ and all~$\theta > 0$, we obtain
 \begin{equation} \label{rate1}
   \begin{split}
    &\int_{\Omega_\eps} \left(\abs{\nabla U_\eps}^2
     + \alpha U_\eps^2\right) \d x
     + \frac{\beta}{\mu_\eps} \int_{\Gamma_\eps} U_\eps^2 \, \d\sigma \\
    &\hspace{0.75cm} \leq \tau\int_{\Omega_\eps}
      \left(\abs{\nabla U_\eps}^2 + U_\eps^2\right) \d x
      + \frac{\tau}{\mu_\eps} \int_{\Gamma_\eps} U_\eps^2 \, \d\sigma \\
    &\hspace{1.5cm} + \frac{C}{\tau} \int_{\Omega_\eps}
     \left((1 - w_\eps)^2 \left(f^2
     + \abs{\nabla u_0}^2 + (\Delta u_0)^2\right)
     + (\beta \kappa_*(\beta) w_\eps + \Delta w_\eps)^2 u_0^2
      \right) \d x \\
    &\hspace{1.5cm}
     + \frac{C\mu_\eps}{\tau} \int_{\Gamma_\eps} (w_\eps - 2)^2
     \, (\partial_\nu u_0)^2 \, \d \sigma
     + \frac{C\chi_\eps^2}{\tau\mu_\eps}
     \int_{\Gamma_\eps} w_\eps^2 \,  u_0^2 \, \d\sigma
   \end{split}
 \end{equation}
 for some constant~$C$ independent of~$\eps$, $\kappa$, $\tau$.
 We can bound several terms on the right-hand side
 of~\eqref{rate1} using the estimate~\eqref{rate0} and
 the inequality~$\abs{\chi_\eps} \leq C\eta_\eps$,
 which follows from~\eqref{corrector_chi} and which is valid
 for some constant~$C$ depending on~$\kappa_*(\beta)$
 but not on~$\eps$. We thus obtain
 \begin{equation} \label{rate2}
    \begin{split}
     &\int_{\Omega_\eps} \left(\abs{\nabla U_\eps}^2
      + \alpha U_\eps^2\right) \d x
      + \frac{\beta}{\mu_\eps} \int_{\Gamma_\eps} U_\eps^2 \, \d\sigma
     \leq \tau\int_{\Omega_\eps}
       \left(\abs{\nabla U_\eps}^2 + U_\eps^2\right) \d x
       + \frac{\tau}{\mu_\eps} \int_{\Gamma_\eps} U_\eps^2 \, \d\sigma \\
     &\hspace{1.5cm} + \frac{C}{\tau}
      \norm{f}^2_{L^\infty(\Omega)} \int_{\Omega_\eps}
      \left((1 - w_\eps)^2
      + (\beta \kappa_*(\beta) w_\eps + \Delta w_\eps)^2
       \right) \d x \\
     &\hspace{1.5cm}
      + \frac{C\mu_\eps}{\tau} \norm{f}^2_{L^\infty(\Omega)}
      \int_{\Gamma_\eps} (w_\eps^2 + 1)^2 \, \d \sigma
      + \frac{C\eta_\eps^2}{\tau\mu_\eps} \norm{f}^2_{L^\infty(\Omega)}
      \int_{\Gamma_\eps} w_\eps^2 \, \d\sigma.
    \end{split}
 \end{equation}
 If~$\alpha > 0$, we choose
 $\tau := \min(1/2, \, \alpha/2, \, \beta/2)$
 and absorb the first two terms on the right-hand side
 into the left-hand side. If~$\alpha = 0$,
 we accomplish the same task by applying the inequality
 \begin{equation} \label{rate-Poincare}
  \int_{\Omega} U_\eps^2 \, \d x
   \leq C \int_{\Omega_\eps} \abs{\nabla U_\eps}^2 \, \d x,
 \end{equation}
 which is valid for an~$\eps$-independent constant~$C$.
 To see why~\eqref{rate-Poincare} holds, we recall that~$u_\eps$,
 $u$ are equal to zero (in the sense of traces)
 on~$\partial\Omega_\eps \cap\partial\Omega$, $\partial\Omega$
 respectively, and we extend~$U_\eps$ to a function
 in~$H^1_{\mathrm{loc}}(\R^n\setminus\mcP_\eps)$ by setting
 $U_\eps := 0$ in~$\R^n\setminus(\Omega\cup\mcP_\eps)$.
 Then, by~\eqref{extensionsupports} the
 function~$\widetilde{E}_\eps U_\eps$
 (with~$\widetilde{E}_\eps$ being an extension operator
 as in~\eqref{extension})
 is supported in a small neighbourhood of~$\Omega$,
 and~\eqref{rate-Poincare} follows by
 applying the Poincar\'e inequality
 to~$\widetilde{E}_\eps U_\eps$.
 From~\eqref{rate2} and~\eqref{rate-Poincare},
 for~$\tau$ small enough we deduce
 \[
  \begin{split}
   &\int_{\Omega_\eps} \left(\abs{\nabla U_\eps}^2
    + \alpha U_\eps^2\right) \d x
    + \frac{\beta}{\mu_\eps} \int_{\Gamma_\eps} U_\eps^2 \, \d\sigma
    \leq C \norm{f}_{L^\infty(\Omega)}^2
     \int_{\Omega_\eps} \left((1 - w_\eps)^2
    + (\beta \kappa_*(\beta) w_\eps + \Delta w_\eps)^2 \right) \d x \\
   &\hspace{1.5cm} + \frac{C\mu_\eps}{\beta}
    \norm{f}_{L^\infty(\Omega)}^2
    \int_{\Gamma_\eps} \left(w_\eps^2 + 1\right) \d \sigma
    + \frac{\eta_\eps^2}{\mu_\eps} \norm{f}_{L^\infty(\Omega)}^2
    \int_{\Gamma_\eps} w_\eps^2 \d\sigma .
  \end{split}
 \]
 We now bound the right-hand side with
 using Lemma~\ref{lemma:corrector}
 (in particular, the estimates~\eqref{correctorL2},
 \eqref{correctorDelta}, \eqref{corrector_L2Gamma})
 and of assumption~\eqref{mu_eps}. This yields
 \[
  \begin{split}
   &\int_{\Omega_\eps} \left(\abs{\nabla U_\eps}^2
    + \alpha U_\eps^2\right) \d x
    + \frac{\beta}{\mu_\eps} \int_{\Gamma_\eps} U_\eps^2 \, \d\sigma
    \leq C \left(\eps^2 + \eta_\eps^2 + \sigma(\Gamma_\eps)^2\right)
     \norm{f}_{L^\infty(\Omega)}^2 \!,
  \end{split}
 \]
 where the constant~$C$ on the right-hand side depends
 on~$\beta > 0$ and on~$\kappa_*(\beta)$ but not on~$\eps$,
 thus completing the proof.
\end{proof}

\section{The model example revisited}
\label{sect:example}

\subsection{The model example satisfies~\eqref{hp:first}--\eqref{hp:last}}

The goal of this section is to prove that the family of holes
defined by Equation~\eqref{model_example}
--- namely, $\mcP_\eps = \eps^{\gamma} \overline{D} + \eps\Z^n$,
where~$D\subseteq\R^n$ is a bounded open set
of class~$C^2$ in dimension~$n\geq 3$
such that~$\R^n\setminus\overline{D}$ is connected
and~$\gamma := n/(n-2)$ ---
fits into the abstract theory developed above.
Let~$\eps_0 \in (0, \, 1)$ be chosen
small enough so that~$\eps^{\gamma - 1}\overline{D}\subseteq Y$
for all~$\eps\in (0, \, \eps_0]$. We define
\begin{equation} \label{model_Heps}
 H_\eps := \eps^{\gamma - 1}\overline{D}
  \quad \textrm{if } \eps\in (0, \, \eps_0], \qquad
 H_\eps := H_{\eps_0}
  \quad \textrm{if } \eps\in (\eps_0, \, 1).
\end{equation}

\begin{lemma} \label{lemma:example-extension}
 The family~$(H_\eps)_{0 < \eps < 1}$ defined
 in~\eqref{model_Heps} satisfies
 assumptions~\eqref{hp:H}, \eqref{hp:decreas}, and~\eqref{hp:extension}.
\end{lemma}
\begin{proof}
 Conditions~\eqref{hp:H} and~\eqref{hp:decreas} are
 satisfied immediately by definition of~$H_\eps$.
 To see that~\eqref{hp:extension} is also satisfied,
 we consider the (open) cube~$Y_0 := \eps_0^{1 - \gamma} Y$,
 with~$\eps_0 > 0$ chosen as above,
 so that~$\overline{D}\subseteq Y_0$.
 The open set~$Y_0\setminus\overline{D}$
 is connected and has Lipschitz boundary,
 hence there exists a bounded linear extension operator
 $E\colon H^1(Y_0\setminus\overline{D})\to H^1(Y_0)$.
 We can always assume that $E$ maps constant functions
 to constant functions, for otherwise, we replace
 $E$ with~$\tilde{E}\colon z \mapsto E(z - \bar{z}) + \bar{z}$,
 where~$\bar{z} := \fint_{Y_0\setminus\overline{D}} z \, \d x$.
 Such an operator satisfies
 \begin{equation*}
  \norm{\nabla(Ez)}_{L^2(Y_0)}
  = \norm{\nabla(E(z - \bar{z}))}_{L^2(Y_0)}
  \leq C \norm{\nabla z}_{L^2(Y_0\setminus\overline{D})}
   + C \norm{z - \bar{z}}_{L^2(Y_0\setminus\overline{D})} \! .
 \end{equation*}
 By applying the Poincar\'e inequality to the second term on
 the right-hand side, we deduce that,
 for all~$z\in H^1(Y_0\setminus\overline{D})$,
 \begin{equation*} %\label{example-extension}
  \norm{\nabla(Ez)}_{L^2(Y_0)}
  \leq C \norm{\nabla z}_{L^2(Y_0\setminus\overline{D})} \! ,
 \end{equation*}
 where the constant~$C$ depends on~$Y_0$ and~$D$, but not on~$z$.
 Now, since the inequality in~\eqref{hp:extension} is scale-invariant,
 we can construct an extension operator
 $E_\eps\colon H^1(Y\setminus H_\eps)\to H^1(Y)$
 satisfying~\eqref{hp:extension} by invoking a scaling argument.
\end{proof}

It still remains to prove that~$(H_\eps)_{0 < \eps < 1}$
satisfies~\eqref{hp:critical} and~\eqref{hp:limit}.
To this end, we recall some notation already
introduced in Section~\ref{model_example_fit_sec}.
We define~$\dot{H}^1(\R^{n}\setminus\overline{D})$ as
set of functions~$z\in L^2_{\mathrm{loc}}(\R^n\setminus\overline{D})$
that satisfy~$\nabla z\in L^2(\R^n\setminus\overline{D}, \, \R^n)$
and $\abs{\{x\in\R^n\setminus\overline{D}\colon
\abs{u(x)}\geq c\}} < +\infty$
for all~$c > 0$. The space~$\dot{H}^1(\R^n\setminus\overline{D})$
is a real Hilbert space with the inner product
\begin{equation} \label{E1-innerproduct}
 \left\langle z, \, w \right\rangle_{\dot{H}^1(\R^n\setminus\overline{D})}
 := \int_{\R^n\setminus\overline{D}} \nabla z\cdot \nabla w \, \d x
  + \fint_{\partial D} z w \, \d\sigma
\end{equation}
(see e.g.~\cite[Theorem~3.3]{AuchmutyHan}).
%We recall some equivalent characterisation of this space.
Let~$X_D$ be the set of functions~$\R^n\setminus\overline{D}\to\R$
that admit an extension in~$C^\infty_{\mathrm{c}}(\R^n)$.

\begin{lemma} \label{lemma:densityE1}
 The set~$X_D$ is dense in~$\dot{H}^1(\R^n\setminus\overline{D})$.
\end{lemma}
\begin{proof}
 This is a classical result and we provide a proof for
 the reader's convenience only.
 Let~$z\in \dot{H}^1(\R^n\setminus\overline{D})$.
 By making use of the operator
 $E\colon H^1(U\setminus\overline{D})\to H^1(U)$,
 we can extend any~$z\in \dot{H}^1(\R^n\setminus\overline{D})$
 to an element of~$H^1_{\mathrm{loc}}(\R^n)$,
 still denoted by~$z$ for simplicity.
 For any positive integer~$j$, we define
 \[
  \phi_j (t) :=
  \begin{cases}
   j       &\textrm{for } t \geq 1/j + j, \\
   t - 1/j &\textrm{for } 1/j < t < 1/j + j, \\
   0       &\textrm{for } -1/j \leq t \leq 1/j, \\
   t + 1/j &\textrm{for } - j - 1/j < t < -1/j, \\
   -j      &\textrm{for } t \leq - j - 1/j
  \end{cases}
 \]
 and~$z_j := \phi_j\circ z$. Note that~$\phi_j$
 is Lipschitz continuous, hence~$z_j\in H^1_{\mathrm{loc}}(\R^n)$
 and~$\nabla z_j \in L^2(\R^n, \, \R^n)$.
 Moreover, defining
 $A_j := \{x\in\R^n\colon \abs{z(x)} \leq 1/j \textrm{ or }
   \abs{z(x)} \geq j + 1/j \}$
 and~$A := \{x\in\R^n\colon z(x) = 0\}$, we have
 \begin{equation*} %\label{densityE1}
   \limsup_{j\to +\infty} \norm{\nabla z - \nabla z_j}_{L^2(\R^n)}
   = \lim_{j\to+\infty} \norm{\nabla z}_{L^2(A_j)}
   = \norm{\nabla z}_{L^2(A)} = 0.
 \end{equation*}
 This fact, together with Lebesgue's dominated convergence theorem,
 implies that~$\norm{z - z_j}_{\dot{H}^1(\R^n\setminus\overline{D})}\to 0$
 as~$j\to+\infty$.
 By construction, $z_j$ is bounded in~$\R^n$
 and is equal to zero outside the
 set~$\{x\in\R^n\colon \abs{z(x)}\geq 1/j\}$,
 which has finite measure because~$z\in\dot{H}^1(\R^n\setminus\overline{D})$.
 Therefore, $z_j\in L^2(\R^n)$, and
 hence~$z_j\in H^1(\R^n)$. By standard density results
 (and continuity of the trace operator $H^1(D)\to L^2(\partial D)$)
% by a standard truncation and regularisation (e.g., by convolution)
 we find functions~$\tilde{z}_j\in C^\infty_{\mathrm{c}}(\R^n)$
 such that~$\norm{z_j - \tilde{z}_j}_{H^1(\R^n)}
 + \norm{z_j - \tilde{z}_j}_{L^2(\partial D)} \leq 1/j$ for each~$j$.
 The lemma follows.
\end{proof}

Let~$V := \dot{H}^1(\R^n\setminus\overline{D})\oplus\R$.
Any element~$z\in V$ can be identified uniquely
with a function of the form
$z = z_0 + L$, where~$z\in\dot{H}^1(\R^n\setminus\overline{D})$
and~$L\in\R$ is a constant. In the sequel, we will
systematically identify elements of~$V$
with functions as above and write~$z(\infty) := L$.

\begin{lemma} \label{lemma:V-limit}
 Let~$(U_j)_{j\in\N}$ be an increasing sequence
 of bounded open sets
 such that \mbox{$\cup_{j\in\N} U_j = \R^n\setminus\overline{D}$}.
 Suppose that the sets in~$(U_j)_{j\in\N}$ satisfy
 a uniform Poincar\'e-Sobolev inequality, i.e. there exists
 a constant~$C$ such that
 \begin{equation} \label{hp:Vlimit-Poincare}
  \norm{z}_{L^p(U_j)} \leq C \norm{\nabla z}_{L^2(U_j)}
 \end{equation}
 for~$p := 2n/(n-2)$, all~$j\in\N$ and all~$z\in H^1(U_j)$
 with~$\int_{U_j} z(x) \d x = 0$.
 Let~$z_j\in H^1(U_j)$ be a sequence of functions such that
 \begin{equation} \label{hp:Vlimit}
  \sup_{j\in\N} \norm{\nabla z_j}_{L^2(U_j)} < + \infty,
  \qquad L := \lim_{j\to+\infty} \fint_{U_j} z_j \, \d x \in\R
 \end{equation}
 (in particular, we assume that the limit exists).
 Then, there exists a (non-relabelled) subsequence
 and a function~$z\in V$ such that~$z_j\rightharpoonup z$
 pointwise almost everywhere and weakly
 in~$H^1(U\setminus\overline{D})$ for all open
 bounded set~$U\subseteq\R^n$, $z(\infty) = L$, and
 \begin{equation} \label{Vlimit}
  \norm{\nabla z}_{L^2(\R^n\setminus\overline{D})}
  \leq \liminf_{j\to +\infty} \norm{\nabla z_j}_{L^2(U_j)}.
 \end{equation}
\end{lemma}
\begin{proof}
 Assumption~\eqref{hp:Vlimit}, combined with standard
 compactness results and a diagonal argument, enables us to
 extract a (non-relabelled) subsequence that converges
 pointwise almost everywhere and weakly
 in~$H^1(U\setminus\overline{D})$ for all bounded open
 sets~$U\subseteq\R^n$. The limit function~$z$ must satisfy
 \[
  \norm{\nabla z}_{L^2(U\setminus\overline{D})}
  \leq \liminf_{j\to+\infty} \norm{\nabla z_j}_{L^2(U\setminus\overline{D})}
  \leq \liminf_{j\to+\infty} \norm{\nabla z_j}_{L^2(U_j)}
 \]
 for all bounded~$U$. Taking the limit as~$U\nearrow\R^n$,
 we deduce that the inequality~\eqref{Vlimit} holds
 and, using assumption~\eqref{hp:Vlimit-Poincare}, that
 $\nabla z\in L^2(\R^n\setminus\overline{D}, \, \R^n)$.
 To conclude the proof, it only remains to show that~$z - L$
 decays at infinity, namely, that for any~$c > 0$ the set
 $E_c := \{x\in\R^n\setminus\overline{D}\colon \abs{z(x) - L}
 \geq c\}$ has finite measure. To this end, we define
 $L_j := \fint_{U_j} z_j(x) \, \d x$ and
 $E_c^j := \{x\in U_j\colon \abs{z_j(x) - L_j} \geq c/2\}$
 for each~$j\in\N$.
 The uniform Poicar\'e-Sobolev inequality~\eqref{hp:Vlimit-Poincare}
 and the uniform bound~\eqref{hp:Vlimit} imply
 \begin{equation} \label{Vlimit1}
  \abs{E^j_c}^{1/p}
  \leq 2c^{-1} \norm{z_j - L_j}_{L^p(U_j)}
  \leq 2Cc^{-1} \norm{\nabla z_j}_{L^2(U_j)}
  \leq M_c,
 \end{equation}
 for some constant~$M_c$ that depends on~$c$ but not on~$j$.
 Moreover, since~$L_j\to L$ by~\eqref{hp:Vlimit} and
 $z_j\to z$ pointwise almost everywhere as~$j\to+\infty$,
 we have $E_c\subseteq \liminf_{j\to+\infty} E_c^j
 := \cup_{k\geq 1} \cap_{j\geq k} E_c^j$ up to negligible sets
 (more precisely, the difference $E_c\setminus(\liminf_{j\to+\infty} E_c^j)$ has Lebesgue measure zero).
 Therefore, Fatou's lemma implies
 \[
  \begin{split}
   \abs{E_c} \leq \liminf_{j\to\infty} \abs{E_c^j}
%   \leq \frac{2}{c} \liminf_{j\to+\infty}
%    \norm{\tilde{z}_j}_{L^p(E_c^j)} < +\infty.
   \leq M_c^p,
  \end{split}
 \]
 where the last inequality follows from~\eqref{Vlimit1}.
 This shows that~$E_c$ has finite measure, thus completing the proof.
\end{proof}

\begin{remark} \label{rk:assumePoincare}
 We have seen in Lemma~\ref{lemma:example-extension}
 that the family of sets~$(H_\eps)_{0 < \eps \leq \eps_0}$
 defined in~\eqref{model_Heps}
 satisfies~\eqref{hp:first}--\eqref{hp:extension}.
 Remark~\ref{rk:Poincarep} provides a constant~$C>0$ such that
 \begin{equation} \label{assumePoincare1}
  \norm{u - \bar{u}}_{L^p(Y\setminus H_\eps)}
  \leq C \norm{\nabla u}_{L^2(Y\setminus H_\eps)}
 \end{equation}
 for all~$\eps \in (0, \, 1)$, $u\in H^1(Y\setminus H_\eps)$,
 with~$\bar{u} := \fint_{Y\setminus H_\eps} u(y) \, \d y$
 and~$p := n/(n-2)$. Let~$Y_0 := \eps_0^{1-\gamma} Y$,
 where~$\eps_0 > 0$ is chosen as in~\eqref{model_Heps}.
 Take~$\eps \in (0, \, \eps_0]$ and let
 $r := (\eps/\eps_0)^{1-\gamma} \geq 1$.
 Since both sides of the inequality~\eqref{assumePoincare1}
 are scale-invariant, by a change of variables we obtain
 \[
  \norm{z - \bar{z}}_{L^p(r Y_0 \setminus\overline{D})}
   \leq C \norm{\nabla z}_{L^2(r Y_0 \setminus\overline{D})}
 \]
 for all~$r\geq 1$, $z\in H^1(r Y_0 \setminus\overline{D})$,
 where~$\bar{z} := \fint_{r Y_0 \setminus\overline{D}} z(x) \, \d x$.
 The constant~$C$ on the right-hand side does not depend on~$r$, $z$.
 Therefore, the family~$(r Y_0\setminus\overline{D})_{r\geq 1}$
 satisfies condition~\eqref{hp:Vlimit-Poincare}.
\end{remark}

Having introduced the space~$V$, we can now prove that
the family~$(H_\eps)_{0 < \eps < 1}$ satisfies the
assumptions~\eqref{hp:critical} and~\eqref{hp:limit}.
We consider~\eqref{hp:limit} first.
Given~$\kappa > 1$, we define
\begin{equation} \label{Lambda*+}
 \Lambda_*(\kappa) := \inf\left\{
 \int_{\R^n\setminus\overline{D}}\abs{\nabla z}^2\colon
  z\in V, \ \fint_{\partial D} z^2 \,\d\sigma
 = 1 + \kappa \, z(\infty)^2\right\} \! .
\end{equation}
Similarly, if~$0 < \kappa < 1$, we define
\begin{equation} \label{Lambda*-}
 \Lambda_*(\kappa) := \inf\left\{
 \int_{\R^n\setminus D}\abs{\nabla z}^2\colon
 z\in V, \ \fint_{\partial D} z^2 \,\d\sigma
 = -1 + \kappa \, z(\infty)^2\right\} \!.
\end{equation}

\begin{lemma} \label{lemma:example_limit}
 The family of sets~$(H_\eps)_{0 < \eps < 1}$
 defined in~\eqref{model_Heps} satisfies
 condition~\eqref{hp:limit}, and the equality
 \[
  \lambda_*(\kappa) = \Lambda_*(\kappa)
 \]
 holds for all~$\kappa > 0$ with~$\kappa\neq 1$.
\end{lemma}
\begin{proof}
 We focus on the case~$\kappa > 1$; in case~$0 < \kappa < 1$
 the proof is analogous.

 \medskip
 \noindent
 \textit{First step.}
 First, we claim that
 \begin{equation} \label{example-limsup}
  \limsup_{\eps\to 0} \frac{\lambda(\eps, \, \kappa)}{\eps^2}
   \leq \Lambda_*(\kappa).
 \end{equation}
 By Lemma~\ref{lemma:densityE1} and a density argument,
 the infimum on the right-hand side
 of~\eqref{Lambda*+} does not change if we restrict our
 attention to competitors~$z\colon\R^n\setminus\overline{D}\to\R$
 that admit a smooth extension~$\R^n\to\R$ and are constant
 in a neighbourhood of infinity. Let~$z$ be such a competitor,
%  We fix an open set~$U_0\subseteq\T^n$ that contains~$H_\eps$
%  for each~$\eps$ sufficiently small and a cut-off
%  function~$\varphi\in C^\infty(\T^n)$,
%  such that~$\int_{\T^n} \varphi(x) \, \d x = 1$
%  and~$\varphi = 0$ in~$\overline{U_0}$.
 and let
 \[
  v_\eps(x) := c_\eps^{-1/2}z(\eps^{1 - \gamma} x)
  \qquad \textrm{for } x\in \T^n\setminus H_\eps \textrm{ and } \eps \in (0, \, 1),
 \]
 where~$c_\eps > 0$ is a normalisation constant,
 to be chosen in a moment. Since~$z$ is equal to
 a constant~$z(\infty)$ away from a compact
 set (which we call~$K$, say), for small enough~$\eps \in (0, \, 1)$
 we have~$v_\eps = c_\eps^{-1/2}z(\infty)$ in~$Y\setminus \eps^{\gamma - 1} K$.
 In particular, $v_\eps$ does satisfy periodic boundary conditions on~$\partial Y$.
% in a neighbourhood of~$\partial Y$ and
 Moreover, we have $c_\eps\int_{\T^n\setminus H_\eps} v_\eps^2 \,\d x
 \to z(\infty)^2$ as~$\eps\to 0$ and, as a consequence,
 \[
  c_\eps \mcG_{\eps,\kappa}(v_\eps)
  = \fint_{\partial H_\eps} z^2 \, \d\sigma
   - \kappa\, z(\infty)^2 + \mathrm{o}_{\eps\to 0}(1)
  = 1 + \mathrm{o}_{\eps\to 0}(1) .
 \]
 In particular, $\mcG_{\eps,\kappa}(v_\eps) > 0$
 for small enough~$\eps$, and we can
 take~$c_\eps$ so that~$\mcG_{\eps,\kappa}(v_\eps) = 1$.
 Such a constant~$c_\eps$ then satisfies~$c_\eps\to 1$ as~$\eps\to 0$.
 Then, $v_\eps$ is an admissible competitor for
 Problem~\eqref{lambda_eps+}, and we have
 \[
  \begin{split}
   \lambda(\eps, \, \kappa)
   \leq \int_{\T^n\setminus H_\eps} \abs{\nabla v_\eps}^2 \, \d x
   = \frac{\eps^{(\gamma - 1)(n-2)}}{c_\eps}
    \int_{\eps^{1 - \gamma} Y\setminus\overline{D}}
    \abs{\nabla z}^2 \, \d x
   = \frac{\eps^2}{1 + \mathrm{o}_{\eps\to 0}(1)}
    \int_{ \eps^{1 - \gamma} Y\setminus\overline{D}}
    \abs{\nabla z}^2 \, \d x.
  \end{split}
 \]
 We divide both sides of this inequality by~$\eps^2$,
 pass to the limit superior as~$\eps\to 0$,
 and take the infimum over all possible~$z$.
 The claim~\eqref{example-limsup} follows.

 \medskip
 \noindent
 \textit{Second step.}
 Next, we claim that
 \begin{equation} \label{example-liminf}
  \Lambda_*(\kappa) \leq
   \liminf_{\eps\to 0} \frac{\lambda(\eps, \, \kappa)}{\eps^2}.
 \end{equation}
 Combined with~\eqref{example-limsup},
 this claim implies that the limit~$\lambda_*(\kappa)$
 exists and is equal to~$\Lambda_*(\kappa)$.
 For the proof, we consider the positive minimiser~$v_{\eps,\kappa}$
 for Problem~\eqref{eigenvalue} and define
 $z_{\eps,\kappa}\colon \eps^{1-\gamma} Y\setminus\overline{D}\to\R$ as
 \begin{equation} \label{example-rescale}
  z_{\eps,\kappa}(x) := v_{\eps,\kappa}\left(\eps^{\gamma - 1} x\right) \! ,
  \qquad x\in \eps^{1-\gamma} Y\setminus\overline{D},
 \end{equation}
 for $\eps$ small enough.
 Up to extraction of a (non-relabelled) subsequence~$\eps\to 0$,
 Lemma~\ref{lemma:constantlimit+} implies that
 \begin{equation} \label{example-liminf-average}
  \fint_{\eps^{1-\gamma} Y\setminus\overline{D}} z_{\eps,\kappa} \, \d x =
  \fint_{\T^n\setminus H_\eps} v_{\eps, \, \kappa} \, \d x \to v_*(\kappa)
 \end{equation}
 as~$\eps\to 0$. Moreover, we have
 \[
  \int_{\eps^{1-\gamma} Y\setminus\overline{D}}
   \abs{\nabla z_{\eps,\kappa}}^2 \d x
  = \eps^{(\gamma - 1)(2 - n)}\int_{\T^n\setminus H_\eps}
   \abs{\nabla v_{\eps,\kappa}}^2 \d x
  = \frac{\lambda(\eps, \kappa)}{\eps^2}.
 \]
 The inequality~\eqref{example-limsup} we have proved
 in the first step shows that the right-hand side remains bounded
 as~$\eps\to 0$. By Lemma~\ref{lemma:V-limit}
 and Remark~\ref{rk:assumePoincare}
 (and up to extraction of a non-relabelled subsequence),
 the functions~$z_{\eps,\kappa}$ converge pointwise almost
 everywhere and
%  Therefore, up to a diagonal argument and a further
%  extraction of a subsequence, the functions~$z_{\eps,\kappa}$
%  converge --- pointwise almost everywhere and
 weakly in~$H^1(U\setminus\overline{D})$ for each
 bounded open~$U\subseteq\R^n$ to a function~$z\in V$,
 which satisfies~$z(\infty) = v_*(\kappa)$
 (because of~\eqref{example-liminf-average}) and
 \begin{equation} \label{example-liminf1}
  \int_{\R^n\setminus\overline{D}}
   \abs{\nabla z}^2 \d x
%  \leq \liminf_{\eps\to 0} \int_{\eps^{1-\gamma} Y\setminus\overline{D}}
%   \abs{\nabla z_{\eps,\kappa}}^2 \d x
  \leq \liminf_{\eps\to 0} \frac{\lambda(\eps, \, \kappa)}{\eps^2}.
 \end{equation}
 Since the trace operator~$H^1(U\setminus\overline{D})\to
 H^{1/2}(\partial D)\hookrightarrow L^2(\partial D)$
 is compact, we have
 \[
  \begin{split}
   \fint_{\partial D} z^2 \, \d\sigma
   = \lim_{\eps\to 0} \fint_{\partial D}
    z_{\eps,\kappa}^2 \, \d\sigma
   = \lim_{\eps\to 0} \fint_{\partial H_\eps}
    v_{\eps,\kappa}^2 \, \d\sigma
   = 1 + \kappa v_*(\kappa)^2.
  \end{split}
 \]
 The last equality follows from the constraint
 $\mcG_{\eps,\kappa}(v_{\eps,\kappa}) = 1$.
 Therefore, $z$ is an admissible competitor for
 Problem~\eqref{Lambda*+}.
% By taking the supremum over all bounded open~$U\subseteq\R^n$ in~\eqref{example-liminf1},
 Taking~\eqref{example-liminf1} into account,
 we obtain the inequalities
 \begin{equation} \label{example-liminf2}
  \Lambda_*(\kappa)
  \leq \int_{\R^n\setminus\overline{D}} \abs{\nabla z}^2 \, \d x
  \leq \liminf_{\eps\to 0} \frac{\lambda(\eps, \, \kappa)}{\eps^2},
 \end{equation}
 which prove~\eqref{example-liminf} and hence complete the proof of the lemma.
\end{proof}

\begin{remark} \label{rk:Gamma}
 Let~$z\in \dot{H}^1(\R^n\setminus\overline{D}) + v_*(\kappa)$
 be the function constructed in Step~2 of the previous proof.
 By combining~\eqref{example-liminf2} with~\eqref{example-limsup},
 we obtain
 \[
  \int_{\R^n} \abs{\nabla z}^2 \, \d x
  = \Lambda_*(\kappa) = \lambda_*(\kappa).
 \]
 In particular, when~$\kappa > 1$
 the infimum on the right-hand side of~\eqref{Lambda*+}
 is attained. The same is true for~\eqref{Lambda*-}
 when~$0 < \kappa < 1$, by a similar argument.
\end{remark}

\begin{proof}[Proof of Proposition~\ref{prop:example}]
 We have seen already that~$(H_\eps)_{0 < \eps < 1}$
 satisfies assumptions~\eqref{hp:H}, \eqref{hp:decreas},
 \eqref{hp:extension}, and~\eqref{hp:limit}.
 By reasoning exactly as in Lemma~\ref{lemma:example_limit}
 we can prove that
 \begin{align}
  \Dir_* := \lim_{\eps\to 0} \frac{\Dir(\eps)}{\eps^2}
   &= \inf\left\{\int_{\R^{n}\setminus\overline{D}}
   \abs{\nabla z}^2\colon
   z\in V, \ z = 0 \textrm{ on } \partial D,
   \ z(\infty) = 1 \right\} \! , \label{example_Dir*} \\
  \St_* := \lim_{\eps\to 0} \frac{\St(\eps)}{\eps^2}
   &= \inf\left\{\int_{\R^{n}\setminus\overline{D}}
   \abs{\nabla z}^2\colon z\in V, \
   \fint_{\partial D} z^2\, \d\sigma = 1,
   \ z(\infty) = 0 \right\} \label{example_St*} \! .
 \end{align}
 For instance, in the case of~\eqref{example_Dir*},
 we consider the unique (nonnegative) minimiser~$\varphi_\eps$
 to~\eqref{Dir_eps}, extend it by zero inside~$H_\eps$
 (which preserves the~$H^1$-regularity, since~$\varphi_\eps = 0$
 on~$\partial H_\eps$ by definition) and observe that
 \[
  1 - \left(\int_Y \varphi_\eps \, \d y\right)^2
  = \int_Y \varphi_\eps^2 \, \d x
   - \left(\int_Y \varphi_\eps \, \d y\right)^2
  = \int_Y \left( \varphi_\eps
     - \int_Y \varphi_\eps \, \d y\right)^2 \d x
  \leq C\eps^2
 \]
 for some~$\eps$-independent constant~$C$,
 because of the Poincar\'e inequality and~\eqref{hp:critical}.
 The functions defined
 by~$z_\eps(x) := \varphi_\eps(\eps^{\gamma - 1} x)$
 for~$x\in \eps^{1-\gamma} Y\setminus\overline{D}$ have
 zero trace on~$\partial D$ and satisfy
 \[
  \fint_{\eps^{1-\gamma}\setminus\overline{D}} z_\eps \, \d x
  = \fint_{Y\setminus H_\eps}\varphi_\eps \, \d x\to 1
  \quad \textrm{as } \eps \to 0,
  \qquad \int_{\eps^{1-\gamma}Y\setminus\overline{D}}
   \abs{\nabla z_\eps}^2 \, \d x = \frac{\Dir(\eps)}{\eps^2}.
 \]
 Therefore, we can apply Lemma~\ref{lemma:V-limit}
 and extract a (non-relabelled) subsequence so that~$z_\eps\to z$,
 where~$z\in V$ is zero on~$\partial D$ and satisfies~$z(\infty) = 1$,
 $\int_{\R^n\setminus\overline{D}} \abs{\nabla z}^2 \, \d x \leq \Dir_*$.
 This proves the right-hand side of~\eqref{example_Dir*}
 does not exceed its left-hand side;
 the opposite inequality follows along the lines of~\eqref{example-limsup}.
 The proof of~\eqref{example_St*} is similar
 but requires some modifications. Indeed, in this case
 we are naturally led to consider the functions
 $\tilde{z}_\eps(x) := \psi_\eps(\eps^{\gamma - 1} x)$,
 $x\in \eps^{1 - \gamma} Y\setminus\overline{D}$,
 where~$\psi_\eps$ is the unique (nonnegative) minimiser to~\eqref{St_eps}.
 Compactness for the functions~$\tilde{z}_\eps$
 comes not from controlling their average but from
 the boundary condition~$\tilde{z}_\eps = 0$ on~$\eps^{1-\gamma} \partial Y$.
 The necessary adaptations to Lemma~\ref{lemma:V-limit}
 and Remark~\ref{rk:assumePoincare} are straightforward.

 Since the trace operator~$V\to L^2(\partial D)$
 is compact (see e.g.~\cite[Corollary~3.4]{AuchmutyHan})
 and the constraints $z(\infty) = 0$, $z(\infty) = 1$
 are convex (in fact, affine), the direct method
 in the calculus of variations implies that the right-hand sides of
 both~\eqref{example_Dir*} and~\eqref{example_St*}
 are attained. In particular, both~$\Dir_*$ and~$\St_*$ are finite
 and strictly positive and~\eqref{hp:critical} is satisfied.

 To complete the proof, it remains to show that~$\lambda_0 := -\lim_{\kappa\to 0} \kappa\lambda_*(\kappa) = \Dir_*$.
 We will do so by applying Proposition~\ref{prop:lambda0}.
 By scaling, condition~\eqref{hp:lambda0_geometry}
 reduces to finding a constant~$C> 0$ such that
 \[
  \sigma(\partial D \cap B^n_\theta(x)) \geq C \theta^{n-1}
 \]
 for all~$x\in\partial D$ and~$\theta$
 such that~$0 < \theta < \sigma(\partial D)^{\frac{1}{n - 1}}$.
 Such a constant exists, because~$\partial D$
 is a compact and $C^2$-regular Riemannian manifold.
 Therefore, \eqref{hp:lambda0_geometry} is satisfied.
 We next prove that~\eqref{hp:lambda0} is also satisfied.
 To this end, for~$\eps\in (0, \, 1)$, we consider the functions
 \begin{equation*} %\label{example-rescale}
  z_{\eps,\kappa}(x) := \kappa^{1/2}
   v_{\eps,\kappa}\left(\eps^{\gamma - 1} x\right) \! ,
  \qquad x\in \eps^{1-\gamma} Y\setminus\overline{D}.
 \end{equation*}
 (Compared to~\eqref{example-rescale},
 this definition has an extra factor of~$\kappa^{1/2}$.)
 By rescaling Equation~\eqref{eigenvalue},
 we see that $z_{\eps,\kappa}$ satisfies, for~$0 < \kappa < 1$,
 \[
  \begin{cases}
   -\Delta z_{\eps,\kappa} = \kappa \eps^{2\gamma - 2} \,
    \abs{\lambda(\eps,\kappa)} z_{\eps, \kappa}
    &\textrm{in } \eps^{1 - \gamma}Y\setminus\overline{D} \\[3mm]
   \partial_{\nu} z_{\eps,\kappa}
    + \dfrac{\abs{\lambda(\eps,\kappa)} \, z_{\eps,\kappa}}
    {\eps^2 \, \sigma(\partial D)} = 0
    &\textrm{on } \partial D,
  \end{cases}
 \]
 where~$\nu$ is the unit
 normal pointing inside~$D$. Having already proved
 that~$H_\eps = \eps^{1-\gamma} \overline{D}$ satisfies the
 conditions~\eqref{hp:first}--\eqref{hp:last},
 we know %from Theorem~\ref{th:lambda*}
 that~$\lambda(\eps,\kappa)/\eps^2\to \lambda_*(\kappa) < 0$
 as~$\eps\to 0$. Moreover,
 $\norm{\nabla z_{\eps,\kappa}}_{L^2(\eps^{1-\gamma} Y\setminus\overline{D})}
 = \kappa^{1/2} \eps^{-1} \abs{\lambda(\eps, \, \kappa)}^{1/2}$
 is bounded uniformly with respect to~$\eps$ and~$\kappa$,
 by~\eqref{lambda_bound-} and~\eqref{hp:critical},
  while
 \[
  \limsup_{\eps\to 0} \fint_{\eps^{1 - \gamma}Y\setminus\overline{D}}
   z_{\eps,\kappa} \, \d x
  = \kappa^{1/2} \limsup_{\eps\to 0}
   \fint_{\T^n\setminus H_\eps} v_{\eps,\kappa} \, \d x
  \leq \frac{2\left(\Dir_*\right)^{1/2}}{\lambda_\infty^{1/2}(1 - \kappa)}
 \]
 by~\eqref{secondogrado-futref}.
 Therefore, Remark~\ref{rk:assumePoincare}
% the Poincar\'e-Sobolev inequality (see e.g.~\cite[Theorem~8.3]{LiebLoss})
 implies that the $L^p(\eps^{1-\gamma} Y\setminus\overline{D})$-norm
 of~$z_{\eps,\kappa}$ (where~$p := 2n/(n-2)$) is bounded uniformly with
 respect to~$\eps$ and~$\kappa\in (0, \, 1/2)$.
 Then, elliptic regularity estimates and a bootstrap argument
 imply that the functions~$z_{\eps,\kappa}$
 restricted to~$\partial D$ are Lipschitz continuous,
 with uniform bounds on the Lipschitz constant. After rescaling,
 this shows that condition~\eqref{hp:lambda0} is satisfied.
 Therefore, Proposition~\ref{prop:lambda0} applies,
 and~$\lambda_0 = \Dir_*$.
\end{proof}

When~$D$ is the unit ball~$B^n$, the problems~\eqref{Lambda*+},
\eqref{Lambda*-}, \eqref{example_Dir*} and~\eqref{example_St*}
can be solved explicitly, leading to the
expressions for~$\lambda_*(\kappa)$
and~$\kappa_*(\beta)$ featuring in Proposition~\ref{prop:ball}.

\begin{proof}[Proof of Proposition~\ref{prop:ball}]
 Given a function~$z\in V$,
 we define~$z_r\colon [1, \, +\infty)\to\R$ as
 \[
  z_r(\rho) := \fint_{\partial B^n}
   z(\rho \, \omega) \, \d\sigma(\omega),
   \qquad \textrm{a.e. } \rho\in [1, \, +\infty).
 \]
 By abuse of notation, we identify~$z_r$ with a radial function
 defined in~$\R^n\setminus B^n$, i.e.~$z_r(x) = z_r(\abs{x})$.
 The~$L^2$-norm of the gradient of~$z_r$ does not exceed that of~$z$:
%  the H\"older inequality implies
 \[
  \begin{split}
   \frac{1}{\sigma(\partial B^n)}
    \norm{\nabla z_r}_{L^2(\R^n\setminus B_1)}^2
   = \int_1^{+\infty} z_r^\prime(\rho)^2 \rho^{n-1} \, \d \rho
   &= \int_1^{+\infty} \left(\fint_{\partial B^n}
    \partial_\rho z(\rho\, \omega) \, \d\sigma(\omega) \right)^2
    \rho^{n-1} \, \d \rho \\
   &\leq \int_1^{+\infty} \left(\fint_{\partial B^n}
    \abs{\partial_\rho z(\rho\, \omega)}^2  \d\sigma(\omega) \right)
    \rho^{n-1} \, \d \rho \\
   &\leq \frac{1}{\sigma(\partial B^n)}
    \norm{\nabla z}_{L^2(\R^n\setminus B_1)}^2 \! .
  \end{split}
 \]
 Therefore, upon replacing~$z$ with~$z_r$,
 the infima in Equations~\eqref{Lambda*+},
 \eqref{Lambda*-}, \eqref{example_Dir*} and~\eqref{example_St*}
 can be restricted to radial functions,
 i.e. those of the form~$z(x) = z(\abs{x})$. Assuming
 for instance~$\kappa > 1$, Lemma~\ref{lemma:example_limit}
 implies
 \begin{equation} \label{Lambda*+ball}
  \lambda_*(\kappa) = \inf\left\{
  \sigma_n \int_1^{+\infty} z^\prime(\rho)^2 \rho^{n-1} \, \d\rho\colon
  z\colon [1, \, +\infty)\to\R, \
  z(1)^2 = 1 + \kappa z(\infty)^2\right\} \! ,
 \end{equation}
 where~$\sigma_n :=\sigma(\partial B^n)$.
 More precisely, the infimum is taken among all
 functions~$z\in H^1_{\mathrm{loc}}(1, \, +\infty)$
 such that the integral on the right-hand side of~\eqref{Lambda*+ball}
 is finite and there exists a constant~$z(\infty)\in\R$ for which
 the difference~$z - z(\infty)$ decays at infinity, i.e.
 $\{\rho\in [1, \, +\infty)\colon \abs{z(\rho) - z(\infty)} > c\}$
 has finite measure for all~$c> 0$.
% We have used the notation~$\sigma_n := \sigma(\partial B^n)$.
 By Remark~\ref{rk:Gamma},
 we know that the infimum is actually a minimum.
 Any minimiser for~\eqref{Lambda*+ball} must satisfy
 the Euler-Lagrange equation, %necessary condition for optimality,
 which implies
%  \[
%   \begin{cases}
%    -\dfrac{\d}{\d\rho}\left(\rho^{n-1} \, z^\prime(\rho)\right) = 0
%      \qquad \textrm{for a.e. } \rho\geq 1. %\\[3mm]
%    z(1)^2 = 1 + \kappa z(+\infty)^2
%   \end{cases}
%  \]
 \[
  z^\prime(\rho) = \frac{c}{\rho^{n-1}}
  \qquad \textrm{a.e. } \rho\geq 1
 \]
 for some constant~$c$.
 By integrating this equality over~$\rho\in [1, \, +\infty)$,
 we can express~$z(\infty)$
 as a function of~$z_1 := z(1)$ and~$c$, so that
 the constraint~$z(1)^2 = 1 + \kappa z(\infty)^2$
 takes the form
 \[
  z_1^2 = 1 + \kappa \left(z_1 + \frac{c}{n-2}\right)^2 \!.
 \]
 This is a quadratic equation for~$c$, whose solutions are
 \begin{equation} \label{ballc}
  c_\pm(z_1) := (n-2)
   \left(- z_1 \pm \sqrt{\frac{z_1^2 - 1}{\kappa}}\right) \! .
 \end{equation}
 The functional to be minimised in~\eqref{Lambda*+ball}
 is proportional to~$c^2$, so all
 that remains to do is minimising the
 function~$f(z_1) := \min(c_-(z_1)^2, \, c_+(z_1)^2)$
 over all~$z_1$ with~$\abs{z_1}\geq 1$.
 Note that~$c_+(-z_1) = -c_-(z_1)$ and
 $c_-(-z_1) = -c_+(z_1)$ for all~$z_1$ with~$\abs{z_1}\geq 1$,
 so the function~$f$ is even. A standard computation shows
 that~$f$ attains its the minimum value when
 \[
  \abs{z_1} = z_\kappa := \sqrt{\frac{\kappa}{\kappa - 1}}.
 \]
 Now, Equation~\eqref{Lambda*+ball}
 reads
 \[
  \lambda_*(\kappa) = \frac{\sigma_n \, f(z_\kappa)}{n-2},
 \]
 and the lemma follows by simple computations.
 The case~$0 < \kappa < 1$ is analogous.
\end{proof}

\subsection{Other asymptotic regimes}
\label{sect:otherregimes}

Finally, in this section we consider other regimes
for our model example, i.e.~we assume the holes have the form
\begin{equation} \label{model_example_bis}
 \mcP_\eps := \eps\Z^n + r_\eps \overline{D},
\end{equation}
where~$r_\eps > 0$ satisfies either~$\eps^{-\gamma} r_\eps \to 0$
or~$\eps^{-\gamma} r_\eps\to + \infty$ as~$\eps\to 0$.
We will deduce the asymptotic behaviour of solutions to~\eqref{main_eq} from abstract results, along
the lines of Theorem~\ref{th:strange_term}.

\begin{prop} \label{prop:strangeterm-0}
 Consider a family of sets~$(H_\eps)_{0 < \eps < 1}$ in~$Y$
 that satisfies~\eqref{hp:first}--\eqref{hp:extension}
 and a family~$(\mu_\eps)_{0 < \eps < 1}$
 of positive numbers satisfying~\eqref{mu_eps}.
 Suppose also that the numbers~$\Dir(\eps)$
 defined in~\eqref{Dir_eps} satisfy
 \begin{equation} \label{hp:strangeterm0}
  \lim_{\eps\to 0} \frac{\Dir(\eps)}{\eps^2} = 0.
 \end{equation}
 Then, the unique solution~$u_\eps$ to~\eqref{main_eq}
 satisfies~$\norm{u_\eps - u_0}_{L^2(\Omega_\eps)} \to 0$
 as~$\eps\to 0$, where~$u_0$ is the unique solution to
 \begin{equation} \label{hom_eq_0}
  \begin{cases}
   -\Delta u_0 + \alpha u_0 = f & \textrm{in } \Omega, \\
   u_0 = 0 &\textrm{on } \partial\Omega.
  \end{cases}
 \end{equation}
\end{prop}
\begin{proof}
 Let~$\varphi_\eps\in H^1(\T^n)$ be the unique
 (positive) minimiser of Problem~\eqref{Dir_eps},
 which we identify with a $\Z^n$-periodic function in~$\R^n$.
 Let
 \[
  w_\eps(x) := \varphi_\eps\!\left(\frac{x}{\eps}\right)
  \qquad \textrm{for } x\in\R^n.
 \]
 Let~$N_\eps$ be the number of cubes of the form~$\eps z+ \eps Y$,
 with~$z\in\Z^n$, that intersect~$\Omega$.
 Since~$\eps^n N_\eps \to \abs{\Omega}$ as~$\eps\to 0$,
 assumption~\eqref{hp:strangeterm0}
 combined with a change of variable yields
 \[
  \int_{\Omega} \abs{\nabla w_\eps(x)}^2 \, \d x
  \leq N_\eps \eps^{n-2} \int_{Y} \abs{\nabla\varphi_\eps(y)}^2 \, \d y
  \leq \frac{C\Dir(\eps)}{\eps^2} \to 0
  \qquad \textrm{as } \eps\to 0.
 \]
 Moreover,
 \[
  \int_{\Omega} w_\eps(x)^2 \, \d x
  = \eps^{n}N_\eps \int_{Y} \varphi_\eps(y)^2 \, \d y
   + \mathrm{o}_{\eps\to 0}(1)
  \to \abs{\Omega} \qquad \textrm{as } \eps\to 0,
 \]
 so~$w_\eps\to 1$ strongly
 in~$H^1(\Omega)$ as~$\eps\to 0$. Now,
 we consider the solutions~$u_\eps$ to~\eqref{main_eq}.
 The uniform estimate~\eqref{bounds} remains true,
 and it is still possible to extract a subsequence~$\eps\to 0$
 such that~$u_\eps$ converges in the sense of~\eqref{weakconv}.
 Given a function~$\varphi\in C^\infty_{\mathrm{c}}(\Omega)$,
 we test Problem~\eqref{main_eq}
 against~$w_\eps\varphi$ and integrate by parts:
 \begin{equation*}
  \begin{split}
   &\int_{\Omega_\eps}\left(w_\eps\nabla u_\eps\cdot\nabla\varphi
    + \varphi\nabla u_\eps\cdot\nabla w_\eps
    + \alpha \, u_\eps w_\eps \varphi\right) \d x
   = \int_{\Omega_\eps} f w_\eps \varphi \, \d x.
  \end{split}
 \end{equation*}
 All the surface integrals on~$\Gamma_\eps$
 vanish because~$w_\eps = 0$
 on~$\partial H_\eps$. Using the strong
 convergence~$w_\eps\to 1$, it is immediate
 to pass to the limit as~$\eps\to 0$ and check that
 the limit~$u_0$ is a solution of~\eqref{hom_eq_0}.
 Since the limit equation
 has a unique solution, we have convergence not
 only along a subsequence, but for the whole family~$\eps\to 0$.
\end{proof}

\begin{prop} \label{prop:strangeterm-beta}
 Consider a family~$(H_\eps)_{0 < \eps < 1}$ of subsets of~$Y$
 that satisfies~\eqref{hp:first}--\eqref{hp:extension}
 and a family~$(\mu_\eps)_{0 < \eps < 1}$ of positive
 numbers that satisfies~\eqref{mu_eps}.
 Moreover, suppose that there exists~$\kappa > 1$ such that
 \begin{equation} \label{hp:strangetermbeta}
  \lim_{\eps\to 0} \frac{\lambda(\eps, \, \kappa)}{\eps^2}
  = + \infty.
 \end{equation}
 Then the unique solution~$u_\eps$ to~\eqref{main_eq}
 satisfies~$\norm{u_\eps - u_0}_{L^2(\Omega_\eps)} \to 0$
 as~$\eps\to 0$, where~$u_0$ is the unique solution to
 \begin{equation} \label{hom_eq_beta}
  \begin{cases}
   -\Delta u_0 + \alpha u_0 + \beta u_0
     = f + g & \textrm{in } \Omega, \\
   u_0 = 0 &\textrm{on } \partial\Omega.
  \end{cases}
 \end{equation}
\end{prop}
\begin{proof}
 The uniform estimate~\eqref{bounds} remains valid and allows us
 to extract a (non-relabelled) subsequence such that~$u_\eps$
 converges to a limit~$u_0$ in the sense of~\eqref{weakconv}.
 Now, let~$\varphi\in C^\infty_{\mathrm{c}}(\Omega)$ be
 a test function. We claim that
 \begin{equation} \label{strangetermbetag}
  \lim_{\eps\to 0} \frac{1}{\mu_\eps} \int_{\Gamma_\eps}
   g \varphi \, \d\sigma
  = \int_{\Omega} g \varphi \, \d x
 \end{equation}
 and that
 \begin{equation} \label{strangetermbeta0}
  \lim_{\eps\to 0} \frac{1}{\mu_\eps} \int_{\Gamma_\eps}
   u_\eps \varphi \, \d\sigma
  = \int_{\Omega} u_0 \varphi \, \d x.
 \end{equation}
 Once these claims are proved, we will be able to pass to the limit
 in the weak formulation of~\eqref{main_eq} and conclude that~$u_0$
 satisfies~\eqref{hom_eq_beta}. Since the limit problem
 has a unique solution, this argument also proves convergence
 for the whole sequence~$\eps\to 0$.

 The identity~\eqref{strangetermbetag} is an immediate
 consequence of Lemma~\ref{lemma:surfacetobulk}
 (and Remark~\ref{rk:stb-otherregimes}),
 applied with~$h = g\varphi$ and~$\tau_\eps = 1$.
 The proof of~\eqref{strangetermbeta0} relies on
 the trace inequality in Lemma~\ref{lemma:trace},
 particularly on~\eqref{traceeps},
 which remains valid under assumption~\eqref{hp:strangetermbeta}
% independently of~\eqref{hp:critical}--\eqref{hp:limit}
 (see Remark~\ref{rk:tracenoncritical}).
 Applying the H\"older inequality
 and~\eqref{traceeps}, we obtain
 \begin{equation*}
  \begin{split}
   \frac{1}{\mu_\eps} \int_{\Gamma_\eps}
    \abs{(u_\eps - u_0) \varphi} \d\sigma
   &\leq \left(\frac{1}{\mu_\eps} \int_{\Gamma_\eps} (u_\eps - u_0)^2 \,\d\sigma\right)^{1/2}
   \left(\frac{1}{\mu_\eps} \int_{\Gamma_\eps} \varphi^2 \,\d\sigma\right)^{1/2} \\
   &\leq C\left( \frac{\eps^2}{\lambda(\eps,\, \kappa)}
    \int_{\Omega_\eps} \abs{\nabla(u_\eps - u_0)}^2 \, \d x
    + \kappa \int_{\Omega_\eps} (u_\eps - u_0)^2 \, \d x\right)^{1/2} \\
   &\hspace{2cm} \cdot \left( \frac{\eps^2}{\lambda(\eps,\, \kappa)}
    \int_{\Omega_\eps} \abs{\nabla \varphi}^2 \, \d x
    + \kappa \int_{\Omega_\eps} \varphi^2 \, \d x\right)^{1/2}
  \end{split}
 \end{equation*}
 for all~$\kappa > 1$ and any small enough~$\eps$.
 We can assume~$\widetilde{E}_\eps u_\eps$ converges to~$u_0$
 weakly in~$H^1(\Omega)$ and strongly in~$L^2(\Omega)$,
 by compact Sobolev embedding.
 Then, using assumption~\eqref{hp:strangetermbeta},
 we conclude that
 \begin{equation} \label{strangetermbeta1}
  \begin{split}
   \lim_{\eps\to 0} \frac{1}{\mu_\eps} \int_{\Gamma_\eps}
    \abs{(u_\eps - u_0) \varphi} \d\sigma = 0.
  \end{split}
 \end{equation}
 Now, we fix~$\delta > 0$ and take~$v_0\in C^\infty_{\mathrm{c}}(\Omega)$
 such that~$\norm{u_0 - v_0}_{L^2(\Omega)}
 + \norm{\nabla u_0 - \nabla v_0}_{L^2(\Omega)}\leq \delta$.
 Applying the H\"older inequality and~\eqref{traceeps}
 exactly as above, we deduce
 \begin{equation} \label{strangetermbeta2}
  \begin{split}
   \limsup_{\eps\to 0} \frac{1}{\mu_\eps} \int_{\Gamma_\eps}
    \abs{(u_0 - v_0) \varphi} \d\sigma
   &\leq C\kappa \, \norm{u_0 - v_0}_{L^2(\Omega)}
    \norm{\varphi}_{L^2(\Omega)}
    \leq C\kappa \, \delta \norm{\varphi}_{L^2(\Omega)}.
  \end{split}
 \end{equation}
 The H\"older inequality also implies
 \begin{equation} \label{strangetermbeta2.5}
   \begin{split}
    \int_{\Omega} \abs{(u_0 - v_0) \varphi} \d x
     \leq \delta \norm{\varphi}_{L^2(\Omega)} \! .
   \end{split}
  \end{equation}
 Finally, applying Lemma~\ref{lemma:surfacetobulk}
 (and Remark~\ref{rk:stb-otherregimes})
 with~$h = v_0 \varphi$ and~$\tau_\eps = 1$, we deduce
 \begin{equation} \label{strangetermbeta3}
  \lim_{\eps\to 0}
   \frac{1}{\mu_\eps} \int_{\Gamma_\eps} v_0\varphi \, \d x
   = \int_{\Omega} v_0\varphi \, \d x.
 \end{equation}
 Combining~\eqref{strangetermbeta1}, \eqref{strangetermbeta2},
 \eqref{strangetermbeta2.5}, and~\eqref{strangetermbeta3}, we obtain
 \[
  \limsup_{\eps\to 0} \abs{ \frac{1}{\mu_\eps}
   \int_{\Gamma_\eps} u_\eps\varphi \, \d x
   - \int_{\Omega} u_0\varphi \, \d x}
   \leq C\kappa\delta\norm{\varphi}_{L^2(\Omega)}
 \]
 and, since~$\delta > 0$ is arbitrary,
 the claim~\eqref{strangetermbeta0} follows.
\end{proof}

Now, let~$(r_\eps)_{0 < \eps < 1}$ be a family of positive numbers
such that~$r_\eps/\eps\to 0$ as~$\eps\to 0$, and let~$\mcP_\eps$
be given as in~\eqref{model_example_bis}.
% for some bounded domain~$D\subseteq\R^n$ of class~$C^2$,
% with~$n\geq 3$ and such that~$\R^n\setminus\overline{D}$ is connected.
We define the rescaled reference hole in the unit cube as
\[
 H_\eps := \eps^{-1} r_\eps \overline{D}
\]
and take~$(\mu_\eps)_{0 < \eps < 1}$ satisfying~\eqref{mu_eps}.

\begin{prop} \label{prop:otherregimes}
 Let~$\gamma := n/(n - 2)$, for~$n\geq 3$. Suppose that
 \begin{equation} \label{r_eps_sup}
  \lim_{\eps\to 0} \frac{r_\eps}{\eps^\gamma} = 0.
 \end{equation}
 Then, the unique solution~$u_\eps$ to~\eqref{main_eq}
 satisfies
 \begin{equation} \label{otherregimes_conv}
  \lim_{\eps\to 0} \norm{u_\eps - u_0}_{L^2(\Omega_\eps)} = 0,
 \end{equation}
 where~$u_0$ is the unique solution to~\eqref{hom_eq_0}. Instead, if
 \begin{equation} \label{r_eps_sub}
  \lim_{\eps\to 0} \frac{r_\eps}{\eps^\gamma} = + \infty,
  \qquad \lim_{\eps\to 0} \frac{r_\eps}{\eps} = 0,
 \end{equation}
 then the unique solution~$u_\eps$ to~\eqref{main_eq}
 converges (in the sense of~\eqref{otherregimes_conv})
 to the unique solution of~\eqref{hom_eq_beta}.
\end{prop}
In the case of~\eqref{r_eps_sup},
Proposition~\ref{prop:otherregimes} is consistent with
the results in~\cite{CioranescuMurat} for the Dirichlet problem:
the holes are too small to have an effect on the limit problem.
Instead, in the case of~\eqref{r_eps_sub}, the additional
term proportional to~$\beta$ in the limit problem
can simply be interpreted as an average
of surface integrals over the boundary of the holes,
as per~\eqref{strangetermbeta3}.
\begin{proof}[Proof of Proposition~\ref{prop:otherregimes}]
 We only sketch the proof, as several arguments
 are very similar to those we presented already in
 the previous sections. Let~$(r_\eps)_{0 < \eps < 1}$ be any
 family of positive numbers such that~$\eps^{-1} r_\eps \to 0$
 as~$\eps\to 0$, and let~$H_\eps := \eps^{-1} r_\eps\overline{D}$.
  We consider the rescaling
 $x\mapsto r_\eps^{-1}\eps x$, which maps~$Y\setminus H_\eps$
 into~$r_\eps^{-1}\eps Y\setminus\overline{D}$.
 By scaling the variable in this way and reasoning
 as in Lemma~\ref{lemma:example_limit}, we can see that
 \begin{equation} \label{otherregimes1}
  \lim_{\eps\to 0} \left(\frac{\eps}{r_\eps}\right)^{n-2}
  \lambda(\eps, \, \kappa) = \Lambda_*(\kappa)
 \end{equation}
 for all~$\kappa > 0$, where~$\Lambda_*(\kappa)$
 is defined by~\eqref{Lambda*+}, \eqref{Lambda*-}, and that
 \begin{equation} \label{otherregimes2}
  \lim_{\eps\to 0} \left(\frac{\eps}{r_\eps}\right)^{n-2} \Dir(\eps)
   = \inf\left\{\int_{\R^n} \abs{\nabla z}^2 \, \d x
   \colon z\in V, \ z = 0 \textrm{ a.e. in } \overline{D},
   \ z(\infty) = 1 \right\} \! .
 \end{equation}
 As we have observed already, the infima on the right-hand sides of
 both~\eqref{otherregimes1} and~\eqref{otherregimes2} are
 attained (by the direct method of the calculus of
 variations) and are positive and finite. Therefore, if~$r_\eps$
 satisfies~\eqref{r_eps_sup}, then
 assumption~\eqref{hp:strangeterm0} is verified and we can apply
 Proposition~\ref{prop:strangeterm-0}, which proves the
 first part of the statement.
 Instead, if~$r_\eps$ satisfies~\eqref{r_eps_sub},
 then~\eqref{hp:strangetermbeta} holds, and the conclusion follows
 from Proposition~\ref{prop:strangeterm-beta}.
\end{proof}

\section{An alternative (but not so strange) example}
\label{sect:many_holes}

\begin{figure}[t]
 \centering
 \includegraphics[height=.35\textheight]{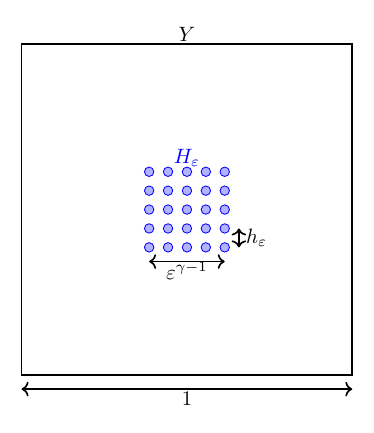}
 \caption{The set~$H_\eps$ defined in~\eqref{example-path}.}
 \label{fig:many_holes}
\end{figure}

The aim of this section is to construct a
family~$(H_\eps)_{0 < \eps < 1}$ which satisfies
assumptions~\eqref{hp:first}--\eqref{hp:last}
but has arbitrarily large surface area.
As in the previous section, we (assume~$n\geq 3$ and) define
\begin{equation} \label{gamma-path}
 \gamma := \frac{n}{n-2}.
\end{equation}
Given~$\eps\in (0, \, 1)$, we consider the
cube~$Q_\eps := [-\eps^{\gamma-1}/2, \, \eps^{\gamma-1}/2]^n$
and take an arbitrary integer~$N_\eps > 0$. We consider
a cubic grid of size~$h_\eps := \eps^{\gamma-1}/N_\eps$
in~$Q_\eps$, label its vertices as
$x_\eps^1, \, \ldots x_\eps^{N_{\eps}^n}$, and define
\begin{equation} \label{example-path}
 H_\eps := \bigcup_{j=1}^{N_\eps^n}
  \overline{B}\!\left(x_\eps^j, \, \frac{h_\eps}{4}\right) \! .
\end{equation}
The set~$H_\eps$ is represented in Figure~\ref{fig:many_holes}.
For the family~$(H_\eps)_{0 < \eps < 1}$, we have
\begin{equation} \label{mu_-path}
 \begin{split}
  \mu_\eps := \frac{\sigma(\partial H_\eps)}{\eps}
   = \frac{\sigma_n N_\eps^n \, h_\eps^{n-1}}{4^{n-1} \eps }
   = C_n N_\eps \eps^{(\gamma - 1)(n-1) - 1}
   = C_n N_\eps \eps^\gamma,
 \end{split}
\end{equation}
where, as before, $\sigma_n := \sigma(\partial B^n)$
and~$C_n := 4^{1-n}\sigma_n$
is a dimensional constant. Therefore, $\mu_\eps$
can be made arbitrarily large by taking~$N_\eps$ large enough.
Since~$\mu_\eps$ satisfies~\eqref{mu_eps},
it follows that~$\sigma(\Gamma_\eps)$ can also be made arbitrarily large.
In fact, we can even make sure that
$\sigma(\Gamma_\eps)\to+\infty$ as~$\eps\to 0$,
by choosing~$N_\eps$ appropriately.
We claim that, no matter what~$N_\eps$ is,
the family~$(H_\eps)_{0 < \eps < 1}$ satisfies
conditions~\eqref{hp:first}--\eqref{hp:last},
at least after extraction of a subsequence.

\begin{prop} \label{prop:example-path}
 Let~$(H_\eps)_{0 < \eps < 1}$ be defined
 as in~\eqref{gamma-path}--\eqref{example-path}
 for an arbitrary family~$(N_\eps)_{0 < \eps < 1}$
 of positive integers.
 Then, there exists a subsequence~$\eps_j\to 0$
 such that~$(H_{\eps_j})_{j\in\N}$ satisfies
 conditions~\eqref{hp:first}--\eqref{hp:last}.
\end{prop}

Proposition~\ref{prop:example-path} implies that
Theorem~\ref{th:strange_term} is applicable
to the family~$(H_{\eps_j})_{j\in\N}$,
yielding a homogenised problem of the form~\eqref{hom_eq}
with a non-zero ``strange term''.
However, when~$(N_\eps)_{0 < \eps < 1}$ is chosen so that
$\sigma(\Gamma_{\eps})\to+\infty$
as~$\eps\to 0$, we have~$\mu_\eps\to +\infty$
for any family~$(\mu_\eps)_{0 < \eps < 1}$
satisfying~\eqref{mu_eps}. In this case,
the Robin coefficient~$\beta/\mu_\eps$ in~\eqref{main_eq}
converges to zero as~$\eps\to 0$, but the
large surface area of~$\Gamma_\eps$ compensates
for the smallness of the coefficient, thus
still producing a non-trivial effect in the limit.

In the proof of Proposition~\ref{prop:example-path},
we will make use of the following results,
which provide sufficient conditions for
assumptions~\eqref{hp:critical}--\eqref{hp:limit} to hold.

\begin{lemma} \label{lemma:suffcond}
 Let~$(H_\eps)_{0 < \eps < 1}$ be a family of sets
 that satisfies conditions~\eqref{hp:first}--\eqref{hp:extension}.
 Assume that there exist~$\eps_0 \in (0, \, 1)$
 and $\eps$-independent positive
 constants~$C_1$, $C_2$ such that
 \begin{equation} \label{hp:suffcond}
  C_1 \, \eps^2 \leq \St(\eps), \qquad
  \Cap(\eps) \leq C_2 \, \eps^2
 \end{equation}
 for all~$\eps \in (0, \, \eps_0]$.
 Then, there exists a subsequence~$\eps_j\to 0$
 such that~$(H_{\eps_j})_{j\in\N}$ satisfies
 conditions~\eqref{hp:critical}--\eqref{hp:limit} too.
\end{lemma}
\begin{proof}
 This lemma follows from Lemma~\ref{lemma:inequalities},
 whose validity depends on assumptions~\eqref{hp:first}--\eqref{hp:extension}
 but not on~\eqref{hp:critical}--\eqref{hp:limit}.
 Inequality~\eqref{St-Cap} in Lemma~\ref{lemma:inequalities},
 combined with assumption~\eqref{hp:suffcond},
 implies that
 \[
  C_1 \, \eps^2 \leq \St(\eps) \leq \Cap(\eps) \leq C_2\, \eps^2.
 \]
 Moreover, it follows from inequality~\eqref{Dir-Cap}
 \[
  \Dir(\eps) \leq \frac{\Cap(\eps)}{(1 - C\eps)^2}
  \leq C\eps^2
 \]
 for some~$\eps$-independent~$C$ and all~$\eps$ small enough.
 Hence, by~\eqref{Cap-Dir},
 \[
  C_1 \, \eps^2 \leq \Cap(\eps) \leq \frac{\Dir(\eps)}{(1 - C\eps)^2},
 \]
 which implies~$\Dir(\eps) \geq C \eps^2$ for a possibly different~$C$
 (that still does not depend on~$\eps$) and any small enough~$\eps$.
 As a consequence, we can extract a subsequence~$\eps_j\to 0$
 so that the limits $\Dir_* := \lim_{j\to+\infty} \eps_j^{-2} \Dir(\eps_j)$,
 $\St_* := \lim_{j\to+\infty} \eps_j^{-2} \St(\eps_j)$
 exist in~$(0, \, +\infty)$. Then, Remark~\ref{rk:Helly}
 implies that condition~\eqref{hp:limit} is
 also satisfied along a (non-relabelled) subsequence.
\end{proof}

\begin{lemma} \label{lemma:suffcondbis}
 Let~$(H_\eps)_{0 < \eps < 1}$ be a family of sets
 that satisfies conditions~\eqref{hp:first}--\eqref{hp:extension}.
 Assume that there exist~$\eps_0 \in (0, \, 1)$, $\kappa > 1$
 and an $\eps$-independent constant~$C_1 > 0$ such that
 \begin{equation} \label{hp:suffcondbis-lower}
  \fint_{\partial H_\eps} v^2 \, \d\sigma
  \leq \frac{C_1}{\eps^2} \int_{\T^n\setminus H_\eps} \abs{\nabla v}^2 \, \d x
  + \kappa \int_{\T^n\setminus H_\eps} v^2 \, \d x
 \end{equation}
 for all~$v\in H^1(\T^n\setminus H_\eps)$ and all~$\eps \in (0, \, \eps_0]$.
 Finally, assume that there exists an $\eps$-independent
 constant~$C_2 > 0$ such that
 \begin{equation} \label{hp:suffcondbis-upper}
  \Cap(\eps) \leq C_2 \, \eps^2
 \end{equation}
 for all small enough~$\eps \in (0, \, 1)$.
 Then, there exists a subsequence~$\eps_j\to 0$
 such that~$(H_{\eps_j})_{j\in\N}$ satisfies
 conditions~\eqref{hp:critical}--\eqref{hp:limit} too.
\end{lemma}
\begin{proof}
 According to Lemma~\ref{lemma:singletrace},
 the number~$1/\lambda(\eps, \, \kappa)$,
 for~$\kappa > 1$, is the optimal constant
 in the trace inequality~\eqref{singletrace+}.
 Therefore, assumption~\eqref{hp:suffcondbis-lower}
 implies that
 \[
  \frac{1}{\lambda(\eps, \, \kappa)} \leq \frac{C_1}{\eps^2},
 \]
 that is,
 \begin{equation} \label{suffcondbis1}
  \lambda(\eps, \, \kappa) \geq \frac{\eps^2}{C_1}.
 \end{equation}
 Inequality~\eqref{St-Cap}, together with~\eqref{hp:suffcondbis-upper},
 implies that~$\St(\eps) \leq \Cap(\eps) \leq C_2\, \eps^2$.
 Therefore, by virtue of~\eqref{suffcondbis1}
 and~\eqref{lambda_bound+0}, one has
 \[
  \frac{\eps^2}{C_1} \leq \lambda(\eps, \, \kappa)
  \leq \frac{\St(\eps)}{(1 - C\eps)^2},
 \]
 which yields~$\St(\eps) \geq C\eps^2$ for
 some~$\eps$-independent constant~$C$.
 The result now follows from Lemma~\ref{lemma:suffcond}.
\end{proof}

\begin{proof}[Proof of Proposition~\ref{prop:example-path}]
 It is clear that~$(H_\eps)_{0 < \eps < 1}$
 satisfies~\eqref{hp:H} and~\eqref{hp:decreas}.
 In what follows, we show that it satisfies~\eqref{hp:extension},
 \eqref{hp:suffcondbis-lower}, and~\eqref{hp:suffcondbis-upper},
 and then apply Lemma~\ref{lemma:suffcondbis}.

 \medskip
 \noindent
 \textit{The family~$(H_\eps)_{0 < \eps < 1}$ satisfies~\eqref{hp:extension}.}
 Arguing as in Lemma~\ref{lemma:example-extension},
 we see that there exists a linear extension operator
 $E\colon H^1(B_2^n\setminus\overline{B}^n)\to H^1(B^n_2)$
 and a constant~$C$ such that
 \[
  \norm{\nabla (Eu)}_{L^2(B_2^n)}
   \leq C \norm{\nabla u}_{L^2(B_2^n\setminus\overline{B}^n)}
 \]
 for all~$u\in H^1(B_2^n\setminus\overline{B}^n)$.
 The set~$H_\eps$ is a union of balls $B_\eps^j := B(x_\eps^j, \, h_\eps/4)$
 whose centres~$x_\eps^j$ are the vertices of a
 cubic grid of size~$h_\eps$. Each ball~$B_\eps^j$
 is contained in a larger (open) ball,
 $2B_\eps^j := B(x_\eps^j, \, h_\eps/2)$, and the balls~$2B_\eps^j$
 are pairwise disjoint. We can therefore construct an extension
 operator~$E_\eps\colon H^1(Y\setminus H_\eps)\to H^1(Y)$
 satisfying the (scale-invariant) estimate in~\eqref{hp:extension},
 based on the above extension operator$E$ and a scaling argument.

 \medskip
 \noindent
 \textit{The family~$(H_\eps)_{0 < \eps < 1}$
 satisfies~\eqref{hp:suffcondbis-upper}.}
 We next prove an upper bound for~$\Cap(\eps)$
 by considering a suitable test function for the
 minimisation problem~\eqref{Cap_eps}. More precisely, set
 \begin{equation} \label{expath-r_eps}
  r_\eps := \sqrt{n}\eps^{\gamma-1} + h_\eps
  = \left(\sqrt{n} + \frac{1}{N_\eps}\right) \eps^{\gamma-1} \! .
 \end{equation}
 Note that the ball centred at the origin of radius~$r_\eps/2$
 contains~$H_\eps$ and that~$r_\eps < 1$ for small enough~$\eps$.
 For such small values of~$\eps$, we also define
 \[
  \xi_\eps(\rho) :=
  \begin{cases}
   0 &\textrm{if } \rho < \dfrac{r_\eps}{2} \\
   \dfrac{r_\eps^{2-n} - 2^{2-n}\rho^{2-n}}{r_\eps^{2-n} - 1}
    &\textrm{if } \dfrac{r_\eps}{2} < \rho \leq \dfrac{1}{2} \\
   1 &\textrm{if } \rho > \dfrac{1}{2}.
  \end{cases}
 \]
 and we identify~$\xi_\eps$ with a radially symmetric
 function~$\xi_\eps(x) := \xi_\eps(\abs{x})$, $x\in Y$.
 Then, $\xi_\eps$ is an admissible
 competitor for the minimisation problem in~\eqref{Cap_eps},
 so we have the upper bound
 \[
  \begin{split}
   \Cap(\eps) \leq \int_Y \abs{\nabla \xi_\eps(x)}^2 \, \d x
   = \sigma_n \int_{r_\eps/2}^{1/2} \xi_\eps^\prime(\rho)^2 \rho^{n-1}\, \d\rho
   = \frac{2^{2-n}\sigma_n}{(n-2)\left(r_\eps^{2 - n} - 1\right)} ,
  \end{split}
 \]
 where~$\sigma_n := \sigma(\partial B^n)$.
 By injecting the definition of~$r_\eps$ and~$\gamma$
 --- i.e. Equations~\eqref{expath-r_eps}, \eqref{gamma-path}
 respectively --- into the right-hand side, we obtain
 \[
  \begin{split}
   \Cap(\eps) \leq \frac{C_n}
    {\left(\sqrt{n} + N_\eps^{-1}\right)^{2 - n}\eps^{-2} - 1}
   \leq \frac{C_n \eps^2}{\left(\sqrt{n} + 1\right)^{2-n} - \eps^2},
  \end{split}
 \]
 where~$C_n := 2^{2-n}\sigma_n/(n-2)$ is a dimensional constant.
 Therefore, we have the estimate $\Cap(\eps) \leq \eps^2$,
 where~$C$ depends only on~$n$, at least for small enough~$\eps$.

 \medskip
 \noindent
 \textit{The family~$(H_\eps)_{0 < \eps < 1}$
 satisfies~\eqref{hp:suffcondbis-lower}.}
 By a standard trace inequality (and a scaling argument),
 there exists a constant~$C$ that satisfies
 \begin{equation} \label{expath1}
  \fint_{\partial B_r^n} u^2 \, \d\sigma
  \leq \frac{C}{r^{n-2}}
   \int_{B_{2r}^n\setminus\overline{B}^n_r} \abs{\nabla u}^2 \, \d x
   + \frac{C}{r^{n}} \int_{B_{2r}^n\setminus\overline{B}^n_r} u^2 \, \d x
 \end{equation}
 for all radii~$r > 0$ and all $u\in H^1(B_{2r}^n\setminus\overline{B}^n_r)$.
 Now, let~$v\in H^1(\T^n\setminus H_\eps)$.
 We apply the inequality~\eqref{expath1} on each of the balls
 $B_\eps^j := B(x_\eps^j, \, h_\eps/4)$, contained in~$H_\eps$,
 Since the balls~$2B_\eps^j := B(x_\eps^j, \, h_\eps/2)$
 are pairwise disjoint and contained in
 $U_\eps := B_{r_\eps/2}^n\setminus H_\eps$,
 with~$r_\eps$ as in~\eqref{expath-r_eps}, we are led to
 \begin{equation*}
  \frac{1}{\sigma_n h_\eps^{n-1}}
  \int_{\partial H_\eps} v^2 \, \d\sigma
  \leq \frac{C}{h_\eps^{n-2}}
   \int_{U_\eps} \abs{\nabla v}^2 \, \d x
   + \frac{C}{h_\eps^n} \int_{U_\eps} v^2 \, \d x.
 \end{equation*}
 Dividing both sides of this inequality by~$N_\eps^n$
 (i.e., the number of balls in~$H_\eps$)
 and substituting~$h_\eps = \eps^{\gamma - 1}/N_\eps$,
 $\gamma = n/(n-2)$ into the right-hand side, we obtain
%  \begin{equation*}
%   \fint_{\partial H_\eps} v^2 \, \d\sigma
%   \leq \frac{C}{h_\eps^{n-2} N_\eps^n}
%    \int_{U_\eps} \abs{\nabla v}^2 \, \d x
%    + \frac{C}{h_\eps^n N_\eps^n} \int_{U_\eps} v^2 \, \d x
%  \end{equation*}
%  \begin{equation*}
%   \fint_{\partial H_\eps} v^2 \, \d\sigma
%   \leq \frac{C}{\eps^{(\gamma-1)(n-2)} N_\eps^2}
%    \int_{U_\eps} \abs{\nabla v}^2 \, \d x
%    + \frac{C}{\eps^{(\gamma - 1)n}} \int_{U_\eps} v^2 \, \d x
%  \end{equation*}
 \begin{equation} \label{expath2}
  \fint_{\partial H_\eps} v^2 \, \d\sigma
  \leq \frac{C}{\eps^2 N_\eps^2}
   \int_{U_\eps} \abs{\nabla v}^2 \, \d x
   + \frac{C}{\eps^p} \int_{U_\eps} v^2 \, \d x,
 \end{equation}
 where~$p := 2n/(n-2)$. It remains to
 estimate the last term on the right-hand side of~\eqref{expath2}.
 To this end, we let~$\bar{v} := \fint_{Y\setminus H_\eps} v \, \d x$.
 The  inequality $(a+b)^2 \leq 2a^2 + 2b^2$
 and the H\"older inequality imply
 \[
  \begin{split}
   \int_{U_\eps} v^2 \, \d x
   &\leq 2\int_{U_\eps} (v - \bar{v})^2 \, \d x
    + 2 \abs{U_\eps} \bar{v}^2 \\
   &\leq 2 \abs{U_\eps}^{2/n}
    \left(\int_{\T^n\setminus H_\eps} (v - \bar{v})^p \, \d x\right)^{2/p}
    + 2 \abs{U_\eps} \fint_{\T^n\setminus H_\eps} v^2 \, \d x.
  \end{split}
 \]
 The uniform Poincar\'e-Sobolev inequality in~$Y\setminus H_\eps$
 (Remark~\ref{rk:Poincarep}), which follows from the previously established
 condition~\eqref{hp:extension}, implies
 \[
  \begin{split}
   \int_{U_\eps} v^2 \, \d x
   \leq C\abs{U_\eps}^{2/n} \int_{\T^n\setminus H_\eps} \abs{\nabla v}^2 \, \d x
    + C \abs{U_\eps} \int_{\T^n\setminus H_\eps} v^2 \, \d x
  \end{split}
 \]
 for some~$\eps$-independent constant~$C$.
 Finally, since~$\abs{U_\eps} \leq C r_\eps^n$ with~$r_\eps \leq C\eps^{2/(n-2)}$
 as in~\eqref{expath-r_eps}, we conclude that
 \begin{equation} \label{expath3}
  \begin{split}
   \int_{U_\eps} v^2 \, \d x
   \leq C \eps^{p - 2} \int_{\T^n\setminus H_\eps} \abs{\nabla v}^2 \, \d x
    + C \eps^p \int_{\T^n\setminus H_\eps} v^2 \, \d x.
  \end{split}
 \end{equation}
 Combining~\eqref{expath2} with~\eqref{expath3},
 we see that~$(H_\eps)_{0 < \eps < 1}$ satisfies~\eqref{hp:suffcondbis-lower}
 for all~$\kappa > \max(C, \, 1)$.
 This completes the proof.
\end{proof}

\section{Acknowledgements}
Part of the work leading to this paper was carried out while GC was visiting the Basque Centre of Applied Mathematics, whose hospitality is gratefully acknowledged.
GC was partially supported by INdAM--GNAMPA, Project E53C25002010001.
KC was supported by EPSRC grant EP/V013025/1. KC also acknowledges the hospitality of the Basque Centre of Applied Mathematics and the University of Verona. The work of AZ has been partially supported by the Basque Government through the BERC 2022-2025 program and by the Spanish State Research Agency through BCAM Severo Ochoa excellence accreditation Severo Ochoa CEX2021-00114 and through project PID2023-146764NB-I00
funded by MICIU/AEI/10.13039/501100011033. AZ was also partially supported by a grant of the Ministry of Research, Innovation and Digitization, CNCS-UEFISCDI, project number PN-IV-P2-2.1-T-TE-2023-1704, within PNCDI IV.

\medskip
\noindent
\textbf{Data availability statement.} Data sharing not applicable to this article as no datasets were
generated or analysed during the current study.

\medskip
\noindent
\textbf{Conflict of interest.} The authors have no competing interests to declare that are relevant to the content of this article.

\bibliographystyle{plain}
\bibliography{homogeneisation}

\end{document}